\documentclass[12pt]{amsart}
\usepackage{amsmath}
\usepackage{amssymb}
\usepackage{graphicx}
\usepackage{pinlabel}
\usepackage[all]{xy}

\usepackage{amsmath,amssymb,amsfonts,amsthm,graphicx}

\usepackage{pinlabel}

\usepackage[usenames,dvipsnames]{color}

\usepackage[usenames,dvipsnames]{color}

\newtheorem{dummy}{dummy}[section]
\newtheorem{lemma}[dummy]{Lemma}
\newtheorem{theorem}[dummy]{Theorem}

\newtheorem{corollary}[dummy]{Corollary}
\newtheorem{proposition}[dummy]{Proposition}

\theoremstyle{definition}
\newtheorem{definition}[dummy]{Definition}

\newtheorem{example}[dummy]{Example}
\newtheorem{remark}[dummy]{Remark}

\numberwithin{equation}{section}

\newcommand{\R}{\mathbb {R}}

\newcommand{\Z}{\mathbb {Z}}
\newcommand{\F}{\mathbb {F}}

\newcommand{\dd}{\partial}

\newcommand{\alg}{\mathcal{A}}

\newcommand{\calC}{\mathcal{C}}

\newcommand{\Char}{\mbox{char}\,}

\def\graded#1{\mathbf{#1}}

\begin{document}

\title[Augmented Legendrian cobordism in $J^1S^1$]{Augmented Legendrian cobordism in $J^1S^1$}

\author{Yu Pan}
\address{Tianjin University}
\email{paulinenk65@gmail.com}

\author{Dan Rutherford}
\address{Ball State University}
\email{rutherford@bsu.edu}

\begin{abstract}  We consider Legendrian links and tangles in $J^1S^1$ and $J^1[0,1]$ equipped with Morse complex families over a field $\mathbb{F}$ and classify them up to Legendrian cobordism.  When the coefficient field is $\mathbb{F}_2$ this provides a cobordism classification for Legendrians equipped with augmentations of the Legendrian contact homology DG-algebras.  A complete set of invariants, for which arbitrary values may be obtained, is provided by the  fiber cohomology, a graded monodromy matrix, and a mod $2$ spin number.  We apply the classification to construct augmented Legendrian surfaces in $J^1M$ with $\dim M = 2$ realizing any prescribed monodromy representation, $\Phi:\pi_1(M,x_0) \rightarrow \mathit{GL}(\mathbf{n}, \mathbb{F})$.  
\end{abstract}

\maketitle

\section{Introduction}
In this article, we classify augmented Legendrian links  
in the $1$-jet space of the circle, $J^1S^1$,
 up to cobordism.  
By an augmented Legendrian link we mean a Legendrian link  $\Lambda$ equipped with the additional structure of an {\it augmentation}, $\epsilon:\mathcal{A} \rightarrow \mathbb{F}_2$, of its Legendrian contact homology dg-algebra.  Augmented Legendrians are
 natural objects   
from the point of view of symplectic field theory, and for some purposes they 
 can  be treated analogously to embedded Lagrangian submanifolds.  For instance, Legendrians with augmentations define objects in Fukaya categories \cite{NZ, NRSSZ}, and an analog of the Lagrangian intersection Floer complex can be defined for pairs of augmented Legendrians  using linearized homology, cf. \cite{Che}.
  Moreover, Legendrians that admit augmentations are  more rigid in their behavior than general Legendrian submanifolds.  For example, 
the number of Reeb chords of an augmented Legendrian in $J^1\R^n$ satisfies an Arnold conjecture type lower bound that can fail for general Legendrians; see \cite{EESorientation, EEMS}.  Our main results show that this heightened rigidity 
 persists at the level of cobordisms, as there are substantially more cobordism classes for Legendrians equipped with augmentations than  for those without.

Augmentations of $1$-dimensional Legendrian links are known to be equivalent to (or closely related to) some other important structures in symplectic topology such as {\it constructible sheaves} \cite{STZ, NRSSZ} and {\it generating families} \cite{ChP, FuRu}.  In this article, we treat augmented Legendrians using
 the correspondence between augmentations and {\it Morse complex families} (abbrv. {\it MCFs}) which are combinatorial analogs of generating families, cf. \cite{Henry, HenryRu3, HenryRu2, RuSu3, AGIgusa, PanRu2}.  
  As in \cite{PanRu2}, 
 working over $\mathbb{F}_2$, the correspondence between augmentations and MCFs induces a bijection between cobordism classes;  see Section 2.
 With this identification understood, we will also use the terminology {\it augmented Legendrian} when referring to an ordered pair 
 $(\Lambda, \mathcal{C})$ where $\Lambda \subset J^1M$ is a Legendrian submanifold and $\mathcal{C}$ is an MCF.  

\begin{remark} Throughout, we will consider MCFs with coefficients in an arbitrary field $\mathbb{F}$ which is a natural level of generality for our arguments.  See Section \ref{sec:generalF} for a discussion of an extension of the correspondence  between cobordism classes of MCFs and augmentations to coefficient fields more general than $\mathbb{F}_2$; currently, there are difficulties when $\mathit{Char} \, \mathbb{F} \neq 2$ stemming from the challenge in explicitly evaluating the orientation signs of holomorphic disks in the case of Legendrian surfaces.
\end{remark}

Roughly, a Morse complex family $\mathcal{C}$ over $\mathbb{F}$ for a Legendrian link or surface, $\Lambda \subset J^1M$, consists of a family of chain complexes $(C_{x}, d_x)$ of $\mathbb{F}$-modules defined for generic $x \in M$  where the generators for $C_x$ correspond to the sheets of $\Lambda|_{x}$, 
 subject to axioms 
 coming from the $1$-parameter and $2$-parameter Morse theory, cf. \cite{HatcherWagoner}.  
  In general, we will consider $\rho$-graded Legendrians and MCFs where the grading on the complexes $(C_{x},d_{x})$ is by $\Z/\rho$ with $\rho \geq 0$.  The cases $\rho = 1$ or $2$ are of some special interest due to a connection between $1$- and $2$-graded augmentations and the Kauffman and HOMFLY-PT polynomials; see \cite{HenryRu3}.  When $\rho$ is even, $\rho$-graded Legendrian cobordisms are orientable; non-orientable cobordisms may occur when $\rho$ is odd.

We now state the cobordism classification in $J^1S^1$. 	To a $\rho$-graded augmented Legendrian $(\Lambda, \mathcal{C}) \subset J^1S^1$ and a choice of base point $x_0 \in S^1$ we can associate the {\it fiber cohomology} at $x_0$, $H^*(C_{x_0},d_{x_0})$, together with a {\it monodromy map}
\[
\phi_{\Lambda, \mathcal{C}}:H^*(C_{x_0},d_{x_0}) \rightarrow H^*(C_{x_0},d_{x_0})
\]
that is a degree preserving automorphism of $H^*(C_{x_0},d_{x_0})$.  Writing $\graded{n}:\Z/\rho \rightarrow \Z_{\geq 0}$,  $\graded{n}(i) = \dim H^{i}(C_{x_0},d_{x_0})$ for the graded dimension of the fiber cohomology, a choice of basis for $H^*(C_{x_0},d_{x_0})$ leads to a {\it graded monodromy matrix}
\[
\graded{M}_{\Lambda, \mathcal{C}} \in GL(\graded{n}, \mathbb{F}):= \prod_{i \in \Z/\rho} GL(\graded{n}(i), \mathbb{F})
\]
which is a sequence of matrices indexed by $\Z/\rho$ such that $\graded{M}_{\Lambda, \calC}(i) \in GL(n(i), \mathbb{F})$ is the matrix of $\phi_{\Lambda, \mathcal{C}}:H^i(C_{x_0},d_{x_0}) \rightarrow H^i(C_{x_0},d_{x_0})$.  When $\Char \mathbb{F} \neq 2$ and $\rho$ is even,  a spin invariant $\xi(\Lambda, \mathcal{C}) \in \Z/2$ is also defined; see Section \ref{sec:invariants}.

\begin{theorem} \label{thm:s1}  Assume either $\Char \mathbb{F} =2$ and $\rho \neq 1$, or $\Char \mathbb{F} \neq 2$ and $\rho$ is even.
Then, two $\rho$-graded augmented Legendrians over $\mathbb{F}$ are cobordant if and only if  
\begin{itemize}
\item they have the same spin invariant (if $\mathit{Char}\, \mathbb{F} \neq 2$),
\item their fiber cohomologies have the same graded dimension, $\graded{n}:\Z/\rho \rightarrow \Z_{\geq 0}$, and
\item their monodromy matrices are conjugate in $GL(\graded{n}, \mathbb{F})$.
\end{itemize}
Moreover, all possible values of these cobordism invariants arise from augmented Legendrians.
\end{theorem}
 \noindent In contrast, without augmentations, two Legendrians in $J^1S^1$ are oriented cobordant if and only if they have the same rotation number 
 and are homologous, cf. \cite{ArnoldWave} and Section \ref{sec:J1S1}.

Theorem \ref{thm:s1} is established as a consequence of a cobordism classification of augmented Legendrian tangles in $J^1[0,1]$.
When working with ($\rho$-graded) Legendrians in $J^1[0,1]$ we fix the number of boundary points, $n$, as well as their grading, $\vec{\mu} \in (\Z/\rho)^n$, to be identical at the left and right boundary, and we focus on MCFs that satisfy $d_x=0$ near the boundary.   
Cobordism classes of such {\it full augmented Legendrian $n$-tangles} form a group denoted $\mathit{Cob}^\rho_{\vec{\mu}}(J^1[0,1]; \mathbb{F})$ with the operation of side-by-side concatenation.  Moreover, 
the monodromy matrix is well-defined (not just up to conjugation).
\begin{theorem}\label{thm:01}  There are group isomorphisms:
\begin{itemize}
\item When $\Char \mathbb{F} \neq 2$ and $\rho$ is even
\begin{align*}
\mathit{Cob}^\rho_{\vec{\mu}}(J^1[0,1]; \mathbb{F}) & \rightarrow  GL(\graded{n},\mathbb{F})\times \Z/2  \\
[(\Lambda, \mathcal{C})] & \mapsto \quad (\graded{M}_{\Lambda, \mathcal{C}}, \xi(\Lambda, \mathcal{C})). 
\end{align*}
\item When $\Char \mathbb{F} =2$ and $(n, \rho) \neq (0,1)$
\begin{align*}
\mathit{Cob}^\rho_{\vec{\mu}}(J^1[0,1]; \mathbb{F}) & \rightarrow  GL(\graded{n},\mathbb{F})  \\
[(\Lambda, \mathcal{C})] & \mapsto \quad \graded{M}_{\Lambda, \mathcal{C}}. 
\end{align*}
\end{itemize}
\end{theorem}
\noindent In the case that $n=0$ and $\rho=1$ we obtain partial results.  See Theorem \ref{thm:J101}.

For an augmented Legendrian surface, $(\Sigma, \mathcal{C}) \subset J^1M$ with $\dim M =2$, as in \cite{RuSu3} the monodromy map construction provides a monodromy representation
\[
\Phi_{\Lambda, \mathcal{C}}: \pi_1(M,x_0) \rightarrow GL(H^*(C_{x_0},d_{x_0})).
\]
As an application of the above theorems, we construct augmented Legendrian surfaces with arbitrary monodromy representations.  See Section \ref{sec:representation}.

\medskip

The remainder of the paper is organized as follows.  Section \ref{sec:MCFoverF} contains the definition of Morse complex families over $\mathbb{F}$ and reviews the continuation maps and monodromy representation; previously MCFs for surfaces have only been defined with $\mathbb{F}_2$ coefficients.  Section \ref{sec:MCFandAug} uses results from \cite{PanRu2} to establish the bijection between cobordism classes of Legendrians equipped with augmentations or MCFs over $\mathbb{F}_2$ and discusses the case of general $\mathbb{F}$.  Section \ref{sec:constructions} provides a set of tools for constructing and working with MCFs and in particular augmented Legendrian cobordisms including:  A characterization of augmented Legendrian cobordisms via moves on $1$-dimensional slices;  the augmented $D_4^-$ cobordism;  a standard form for MCFs (the $SR$-form) from generalized normal rulings;  and extension of MCFs over Legendrian isotopies.        
These methods have grown out of the literature on MCFs.  Eg., M. B. Henry studies equivalence of MCFs via moves on slices in \cite{Henry}; various constructions of $2$-dim MCFs appear in \cite{RuSu3} and \cite{PanRu2}; for knots in $\R^3$ the $SR$-form MCFs are considered in \cite{Henry} and \cite{HenryRu3}.  As the original sources mainly restrict to $\F_2$ coefficients and do not consider $1$-dimensional links in $J^1S^1$, we provide here a mostly self contained exposition that includes fairly detailed proofs.    

With these preparations in hand, Sections \ref{sec:invariants}-\ref{sec:J1S1} turn toward establishing Theorems \ref{thm:s1} and \ref{thm:01}.  Section \ref{sec:invariants} establishes the fiber cohomology, monodromy matrix, and spin numbers as invariants of augmented cobordism.  Section \ref{sec:J101} proves Theorem \ref{thm:01} (see Theorem \ref{thm:J101}).  This involves introducing standard forms for augmented Legendrians up to cobordism that are certain positive permutation braids equipped with particular MCFs.  It is then shown that every augmented Legendrian is cobordant to a standard form and that the standard forms are uniquely determined by their invariants.  
Section \ref{sec:J1S1} establishes Theorem \ref{thm:s1} (see Theorem \ref{thm:S1main}) by realizing augmented Legendrians in $J^1S^1$ up to cobordism as closures of full $n$-tangles and applying Theorem \ref{thm:01}.  Finally, Section \ref{sec:representation} shows how to construct augmented Legendrian surfaces with arbitrary monodromy representation, and illustrates some examples in $J^1T^2$.

\subsection{Acknowledgements}  The second author is partially supported by grant 429536 from the Simons Foundation.

\section{Morse complex families over $\mathbb{F}$}  \label{sec:MCFoverF}

The augmented Legendrians that we consider in this paper are pairs $(\Lambda, \calC)$ consisting of a ($\rho$-graded) Legendrian knot or surface, $\Lambda\subset J^1M$, in a $1$-jet space together with a ($\rho$-graded) Morse complex family (abbrv. MCF), $\calC$.  Motivated by generating families of functions and the bifurcation theory for Morse complexes in $1$ and $2$ parameter families, see \cite{HatcherWagoner}, MCFs have been considered for $1$-dimensional Legendrians knots in \cite{Pushkar, Henry, HenryRu1, HenryRu2}  (over $\F_2$) and \cite{HenryRu2} (over any $\F$) and for $2$-dimensional Legendrian surfaces and cobordisms in \cite{RuSu3} and \cite{PanRu2} (over $\F_2$).  In this section, we review Morse complex families and extend the definition to allow general coefficients in the $2$-dimensional case.  In Section \ref{sec:2-4}, we discuss the monodromy representation of an MCF $\calC$ that is a representation of $\pi_1(M)$ on the fiber cohomology of $\calC$.  

\medskip

\noindent{\bf Matrix Notation:}  Throughout the paper we will frequently use the notation $E_{i,j}$ for a matrix whose $(i,j)$-entry is $1$ and whose other entries are $0$.

\medskip

\subsection{Graded Legendrian submanifolds in $1$-jet spaces} \label{sec:Leg}

We assume basic familiarity with Legendrian knots and surfaces in $1$-jet spaces.  In particular, when $M$ is a manifold the $1$-jet space $J^1M$ has a standard contact structure that appears in coordinates as $dz - y\,dx$ where $(x,y,z)$ are coordinates on $T^*M\times \R_z$ resulting from a local coordinate $x$ on $M$.  Denote by $\pi_{xz}:J^1M \rightarrow J^0M =M \times \R$ and $\pi_{x}:J^1M \rightarrow M$ the {\bf front} and {\bf base projections}.  A Legendrian submanifold  $\Lambda \subset J^1M$ may be conveniently presented via its front projection, $\pi_{xz}(\Lambda) \subset J^0M$, which (generically) in the $1$-dimensional case has standard cusp and crossing singularities all with different base projections.  Generic singularities of $2$-dimensional front projections include $1$-dimensional cusp edges and crossing arcs (aka singularities of type $A_2$ and $A_1^2$), triple points ($A_1^3$), cusp-sheet intersections ($A_2A_1$), and swallow tail points ($A_3$).  See eg. \cite{ArnoldWave} or \cite[Figure 1]{RuSu3}.  We consider compact Legendrian submanifolds, $\Lambda \subset J^1M$, that are properly embedded and such that the front projection $\pi_{xz}(\Lambda)$ appears as a product in a collar neighborhood of $\partial M$.  
When $B \subset \partial M$ is a boundary component, the {\bf restriction} $\Lambda|_{B} \subset J^1B$ is defined by taking the intersection  $\Lambda \cap \pi_{x}^{-1}(B)$ and then projecting out the $y$-coordinate transverse to $B$.  We will refer to $1$-dimensional Legendrians $\Lambda \subset J^1[0,1]$ as {\bf Legendrian tangles} and, in the case where $\Lambda|_{\{i\}}$ consists of $n$ points for $i=0,1$, as {\bf Legendrian $n$-tangles}.  

A {\bf Legendrian cobordism} between two compact $1$-dimensional Legendrians $\Lambda_0, \Lambda_1 \subset J^1M$ is a compact Legendrian surface $\Sigma \subset J^1(M\times [0,1])$ satisfying $\Sigma|_{M\times\{i\}} = \Lambda_{\{i\}}$.  For any $\Lambda \subset J^1M$, there is a {\bf trivial} or {\bf product cobordism} that (by a slight abuse of notation) we denote as $\Lambda \times[0,1] \subset J^1(M\times[0,1])$; it appears in coordinates $(x_1,x_2, y_1,y_2, z)$ with $(x_1,y_1) \in T^*M$ and $(x_2,y_2) \in T^*[0,1]$ as $\{(x_1, x_2, y_1, 0, z)\,|\, (x_1,y_1,z) \in \Lambda\}$.

Let $\rho \geq 0$ be a non-negative integer.  A {\bf $\Z/\rho\Z$-valued Maslov potential} for $\Lambda$ is a locally constant function 
\[
\mu: \Lambda \setminus \Lambda_{\mathit{cusp}} \rightarrow \Z/\rho\Z,
\]
where $\Lambda_{\mathit{cusp}} \subset \Lambda$ is the set of points where $\pi_{xz}|_{\Lambda}$   
has a cusp or swallowtail singularity, such that the value of $\mu$ increases by one when passing from the lower sheet to the upper sheet at a cusp edge.  A {\bf $\rho$-graded Legendrian submanifold} is a pair $(\Lambda, \mu)$ such that $\mu$ is a $\Z/\rho\Z$-valued Maslov potential for $\Lambda$.  

\begin{remark}
If $\rho$ is even and $M$ is oriented, then any $\rho$-graded Legendrian  $\Lambda\subset J^1M$ receives an orientation by making the requirement that the base projection $\pi_x:\Lambda \rightarrow M$ is orientation preserving (resp. reversing) on sheets of $\Lambda$ where $\mu$ is even (resp. odd).  In particular, when considered oriented $1$-dimensional $\rho$-graded Legendrians in $J^1M$ with $M =\R,  [0,1]$ or $S^1$ we will always use this orientation so that $\mu$ is even (resp. odd) where the orientation of the front projection is in the positive (resp. negative) $x$-direction.
\end{remark}

Let $(\Lambda, \mu) \subset J^1M$ be a $\rho$-graded Legendrian submanifold of dimension $1$ or $2$, and  let $R$ be a commutative ring with identity.  
  For each $x_0 \in M$ that is not the base projection of any singular point of $\pi_{xz}(\Lambda)$ (including crossings) we refer to the points in $\pi_x^{-1}(\{x_0\}) \cap \Lambda$ as the {\bf sheets of $\Lambda$ at $x_0$}, and we label them with descending $z$-coordinates as
\[
S_1^{x_0}, \ldots, S_{n}^{x_0},  \quad  z(S^{x_0}_{i}) > z(S^{x_0}_{i+1}), \quad \mbox{for $1\leq i < n = n(x_0)$.}
\] 
To each such $x_0$, we associate a free $\Z/\rho\Z$-graded $R$-module 
\[
C_{x_0} = \mathit{Span}_R\{S_1^{x_0}, \ldots, S_n^{x_0} \}, \quad |S_i^{x_0}| = \mu(S_i^{x_0}).
\]
We will usually omit the superscript $x_0$ from the notation for sheets.

\subsection{1-dimensional Morse complex families}  Assume now that $(\Lambda, \mu) \subset J^1M$ is a $1$-dimensional $\rho$-graded Legendrian link or tangle with generic front and base projections.  In particular, no two front singularities have the same base projection.  
\begin{definition} \label{def:MCF1}
A ($\rho$-graded) {\bf Morse complex family} (abbrv. {\bf MCF}) over $R$ (a commutative ring with identity) for $(\Lambda, \mu)$ is a tuple
\begin{align*}
\mathcal{C} &= \left(\{d_x\}, H, *, \xi \right) \quad \mbox{when $\Char R \neq 2$,} \quad \mbox{or} \\
\mathcal{C} &= \left(\{d_x\}, H, *\right) \quad \quad \mbox{when $\Char R = 2$,}
\end{align*}
consisting of the following items (1)-(4): 
\begin{enumerate}
\item  A collection of {\bf homology basepoints},  $* =\{*_1, \ldots, *_r\} \subset \Lambda$, assigned values $s_i \in R^*$.
\item  When $\Char R \neq 2$, a collection of {\bf spin basepoints},
 $\xi=\{\circ_1, \ldots, \circ_s\}$, each assigned the value $-1$.  {\it When $\Char R =2$, a collection of spin basepoints is NOT included in the defining data for $\mathcal{C}$.}
\item  A handleslide set, $H$, that is a collection of formal handleslides with coefficients $b \in R$.  Here, a {\bf formal handleslide} is a triple $h = (x, u_h, l_h)$
where $x \in M$ and $u_h, l_h\in \Lambda \setminus \Lambda_{\mathit{cusp}}$ are lifts of $x$ having
\[
\pi_x(u_h)=\pi_x(l_h) =x, \quad  z(u_h) > z(l_h), \quad \mbox{and} \quad \mu(u_h) = \mu(l_h) \in \Z/\rho.
\] 
\end{enumerate}
We require that all cusps, crossings, handleslides, and basepoints occur at distinct $x$-coordinates, and we denote by $M^\calC_\mathit{reg} \subset M$ the complement of such $x$-coordinates.
\begin{enumerate}
\item[(4)] A collection of upper triangular differentials $d_x:C_x \rightarrow C_x$ of degree $+1$ mod $\rho$ defined for each $x \in M^\calC_{\mathit{reg}}$ 
 that are locally constant in $x$.  
I.e., writing $A =(a_{i,j})$ for the matrix of $d_x$, 
\[
d_x S_j = \sum_{i} a_{i,j}S_i,
\]
it is required that $A^2=0$; $a_{i,j} =0$ unless $\mu(S_i) = \mu(S_j)+1$; $a_{i,j} = 0$ when $i \geq j$; and the matrices $A$ 
remain constant except when $x$ passes a crossing, cusp, basepoint, or handleslide.  
\end{enumerate}

The data of $\mathcal{C}$ are subject to the following axioms.  Suppose that between $x_-$ and $x_+$ there is a single crossing, cusp,  handleslide, or basepoint at $x_0$.  Write $d_\pm$ and $A_\pm = (a^\pm_{i,j})$ for the differentials and their matrices at $x_-$ and $x_+$.  

\begin{enumerate}
\item[(A1)] When $x_0$ has a {\it crossing} involving strands $k$ and $k+1$,   
\[
A_+ = Q_{(k\,k+1)} A_-Q_{(k\,k+1)}
\]
where $Q_{(k\,k+1)}$ is the permutation matrix of the transposition $(k \, k+1)$.

\item[(A2)] When $x_0$ has a {\it left cusp} involving the sheets $S^+_{k}$ and $S^+_{k+1}$, $A_+$ is obtained from $A_-$ by inserting two new rows and columns at position $k$ and $k+1$ with the $2\times 2$ block 
\[
N = \left[\begin{array}{cc} 0 & 1 \\ 0 & 0 \end{array} \right]
\] 
appearing on the main diagonal and all other entries $0$.  The same requirement with the role of $A_+$ and $A_-$ interchanged is made at right cusps.

\item[(A3)] When $x_0$ has a {\it handleslide} with upper (resp. lower) endpoint on $S_k$ (resp. $S_l$) and coefficient $b$,
\[
A_+ = (I+bE_{k,l})A_-(I+bE_{k,l})^{-1} = (I+bE_{k,l})A_-(I-bE_{k,l})
\] 
or equivalently
\[
A_+(I+B) = (I+B)A_-  \quad \quad \mbox{where } B= bE_{k,l}
\]
where $E_{i,j}$ is the matrix with $i,j$-entry $1$ and all other entries $0$.

\item[(A4)] When $x_0$ has a {\it basepoint} on sheet $S_k$ with coefficient $s$ (for spin basepoints $s=-1$), let $\Delta$ be the diagonal matrix with $s$ (resp. $s^{-1}$) at position $k$ on the diagonal if $\pi_x$ is orientation
 preserving (resp. reversing) at $S_k$; the remaining diagonal entries of $\Delta$ are $1$'s.  Then,
\[
A_+ = \Delta A_-\Delta^{-1}.
\]

\end{enumerate}

\end{definition}

\begin{remark}  
\begin{enumerate}
\item Note that since the $A_\pm$ are upper triangular, (A1) implies that $a^\pm_{k,k+1} = 0$ when $x_-$ and $x_+$ border a crossing between sheets $S_k$ and $S_{k+1}$.
\item As the differentials, $d_x$, are locally constant, one can equivalently (as done in \cite{RuSu3}) assign complexes $(C_\nu, d_\nu)$ to the {\it connected components}, $R_\nu$, of $M^\calC_{\mathit{reg}}$.
\end{enumerate}
\end{remark}

See Figure \ref{fig:Handleslides} for an illustration of the coefficients that appear in the axioms (A2)-(A4). In our figures, we indicate a coefficient in a differential, $\langle d_x S_j, S_i\rangle= a_{i,j}$, as a dotted arrow at $x$ from $S_i$ to $S_j$ labeled with $a_{i,j}$;  handleslides are indicated by solid dark red line segments labeled with coefficients.

\begin{figure}[!ht]
\labellist
\small
\pinlabel $a_{ik}$ [r] at 20 206
\pinlabel $a_{lj}$ [r] at  20 146
\pinlabel $-a_{ik}b$ [l] at 106 206
\pinlabel $ba_{lj}$ [l] at 142 146

\pinlabel $b$ [r] at 62 176
\pinlabel $1$ [l] at  260 184
\pinlabel $1$ [r] at 370 184
\pinlabel $a_{ik}$ [r] at 16 60
\pinlabel $a_{ik}s^{-1}$ [l] at 102 60
\pinlabel $a_{kj}$ [r] at 32 20
\pinlabel $sa_{kj}$ [l] at 118 20
\pinlabel $s$ [b] at 68 54
\pinlabel $a_{ik}$ [r] at 266 60
\pinlabel $a_{ik}s$ [l] at 350 60
\pinlabel $a_{kj}$ [r] at 278 20
\pinlabel $s^{-1}a_{kj}$ [l] at 364 20
\pinlabel $s$ [b] at 312 54

\endlabellist
\includegraphics[width=5in]{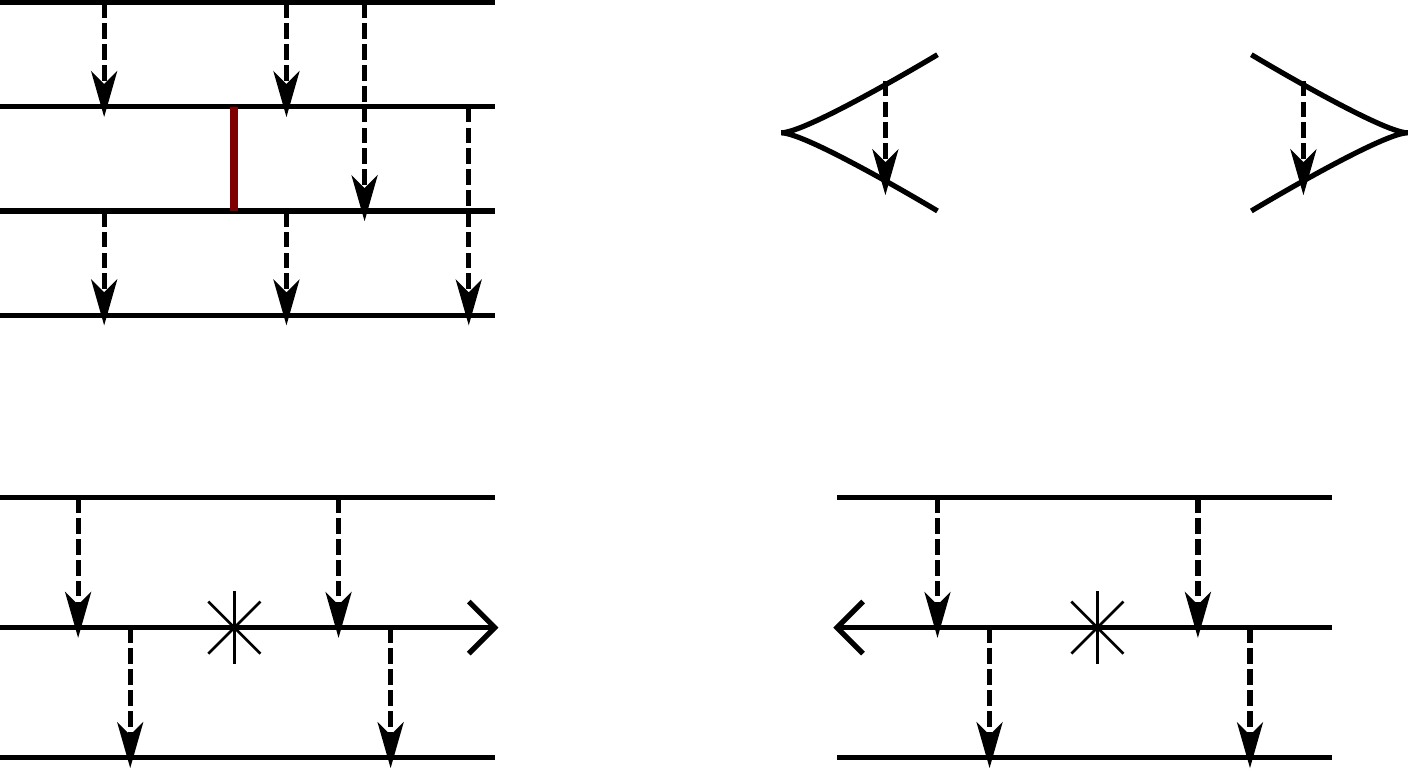}
\caption{The coefficient rules for $1$-dimensional MCFs.}
\label{fig:Handleslides}
\end{figure}
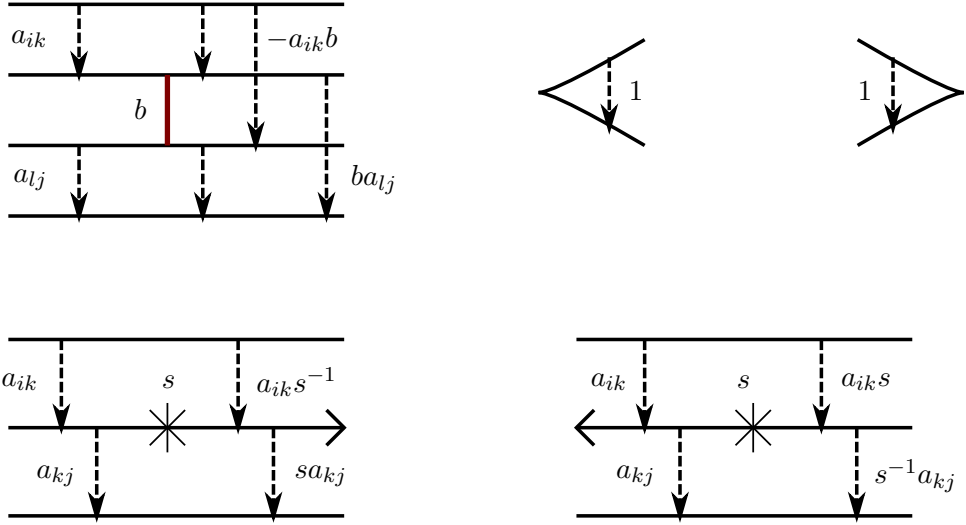

An example of an MCF for a $0$-graded Legendrian knot in $J^1S^1$ appears in Figure \ref{fig:MCFEx2}.

\begin{figure}[!ht]

\quad

\labellist
\tiny
\pinlabel $r_1$ [b] at 112 182
\pinlabel $-r_1^{-1}$ [b] at 142 182
\pinlabel $r_2$ [b] at 160 92
\pinlabel $-r_2^{-1}$ [b] at 194 100
\pinlabel $r_3$ [b] at 360 38
\pinlabel $-r_3^{-1}$ [b] at 390 36

\pinlabel $b$ [t] at 330 84
\pinlabel $s$ [b] at 256 154
\pinlabel $-1$ [b] at 448 18

\pinlabel $1$ [l] at 102 122
\pinlabel $r_1$ [l] at 152 122
\pinlabel $-r_1r_2$ [l] at 228 126
\pinlabel $-sr_1r_2$ [l] at 296 130

\pinlabel $1$ [l] at 152 60
\pinlabel $r_2$ [l] at 236 60
\pinlabel $-r_2r_3$ [l] at 416 58
\pinlabel $r_2r_3$ [bl] at 484 44

\pinlabel $2$ [r] at -2 136
\pinlabel $0$ [r] at -2 68

\endlabellist
\includegraphics[scale=.75]{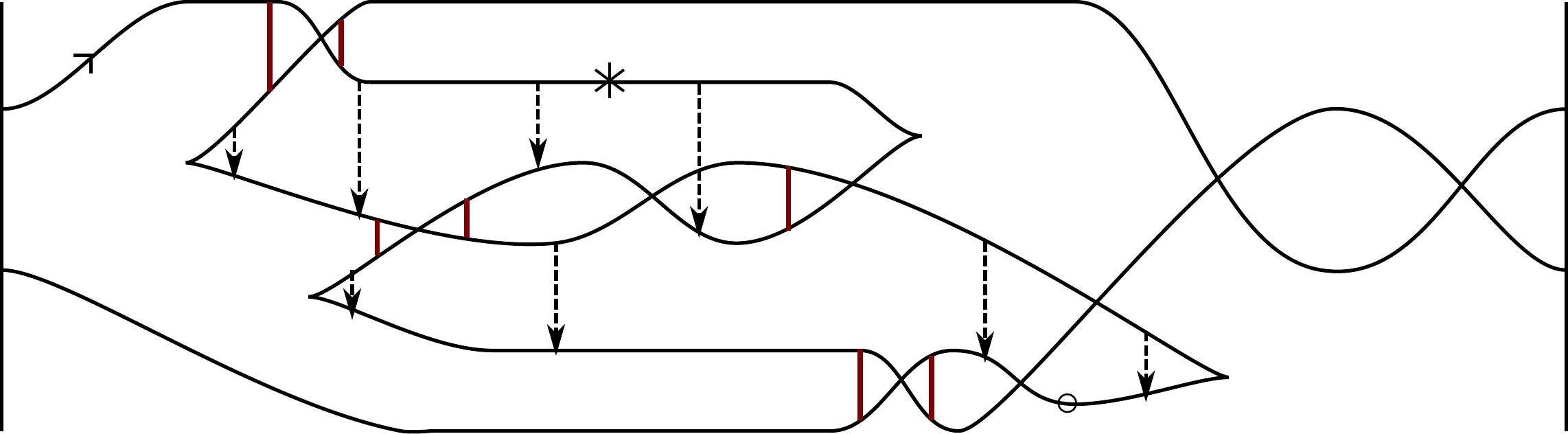}
\caption{An MCF for a $0$-graded Legendrian knot $\Lambda \subset J^1S^1$.  The Maslov potential $\mu$ is indicated at the left.  The coefficients $r_1, r_2, r_3, s \in \mathbb{F}^*$ and $b\in \mathbb{F}$ must satisfy $-sr_1r_2 =1$ and $r_2r_3 =1$.  The differentials $d_x$ have $d_x =0$ at $x=0$ and are then uniquely determined by (A1)-(A4).  Some selected coefficients $a_{ij} = \langle d_xS_j,S_i \rangle$ are indicated by the dotted arrows.}
\label{fig:MCFEx2}
\end{figure}
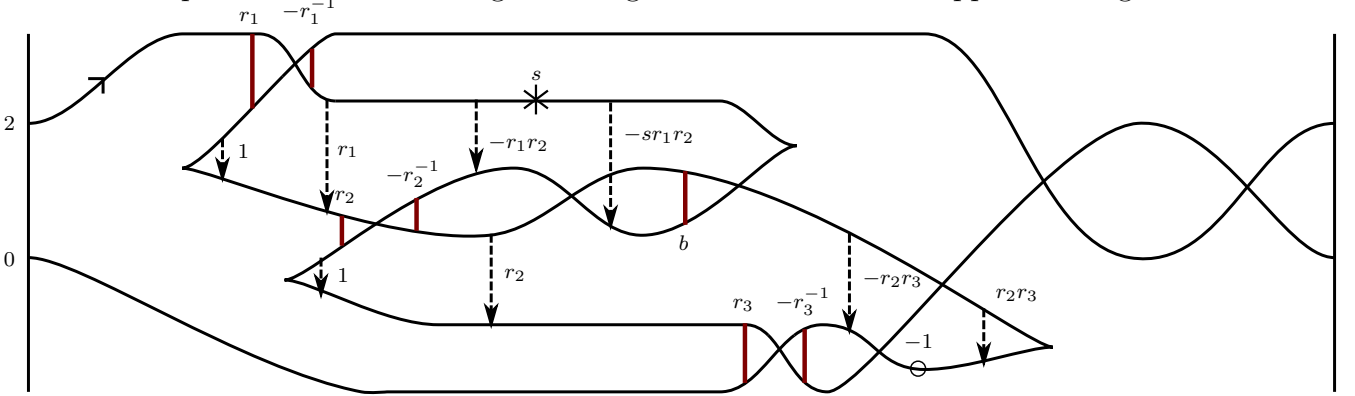

\begin{remark}
The above definition of MCF over a ring $R$ is similar to the definition from \cite{HenryRu3} with added flexibility allowed for base points and spin points and a simplified axiom at cusps.  For Legendrians in $J^1\R$, the MCFs considered in \cite{HenryRu3} become a special case of Definition \ref{def:MCF1} by choosing $*$ and $\xi$ so that on each component of $\Lambda$ a single $*_i$ point is placed next to some chosen right cusp, and spin base points are placed next to all remaining right cusps.  
\end{remark}

\subsection{2-dimensional Morse complex families}  \label{sec:2dimMCF}

Consider now a $\rho$-graded Legendrian surface $(\Lambda, \mu) \subset J^1M$ and a commutative, unital ring $R$.  
  We allow for the case where $\partial M \neq \emptyset$ with $\Lambda$ properly embedded as in Section \ref{sec:Leg}.

\begin{definition} 
\begin{itemize}
\item  A {\bf homology structure},  $* = \{*_1, \ldots, *_r\}$ is a collection of co-orientable curves in $\Lambda$ satisfying:
\begin{enumerate}
\item The $*_i$ are disjoint from one another except at endpoints and all $*_i$ are transverse
 to $\partial\Lambda$ as well as to the inverse image of the singular set of $\pi_{xz}(\Lambda)$.
\item The arc components of $*$ have their endpoints  either at boundary points of $\Lambda$, or at tri-valent vertices where three arc endpoints are simultaneously located.  The intersections of arcs at tri-valent vertices are pairwise transverse.
\end{enumerate}
\item A {\bf combinatorial spin structure}, $\xi = \{\xi_1, \ldots, \xi_s\}$, is a collection of disjoint embedded curves in $\Lambda$ whose boundary consists of the set of all swallowtail points of $\Lambda$ and some number of points on $\partial \Lambda$.

\item A multiplicative (resp. additive) {\bf co-oriented coefficient} for a co-orientable curve $C \subset \Lambda$ is an assignment of coefficients $\alpha$ and $\beta$ in $R^*$ (resp. in $R$) to the two distinct co-orientations of $C$ such that $\beta = \alpha^{-1}$ (resp. $\beta = -\alpha$.)
\end{itemize}
\end{definition}

\begin{remark} Note that the difference of two combinatorial spin structures is a relative cycle in $C_1(\Lambda, \partial \Lambda; \Z/2)$, so that combinatorial spin structures up to equivalence form an affine space over $H_1(\Lambda, \partial \Lambda; \Z/2) = H^1(\Lambda; \Z/2)$, just like geometric spin structures.  The role of $\xi$ in the definition of MCFs for Legendrian surfaces is analogous to the choice of geometric spin structure required in lifting the coefficients of the Legendrian contact homology DGA to $\Z$, cf. 
\cite{EESorientation}.  
\end{remark}

\begin{definition}  A {\bf MCF} for $(\Lambda, \mu)$ over $R$ is a tuple $\mathcal{C} = \left(\{d_x\}, H, *, \xi \right)$, when $\Char R \neq 2$, or $\mathcal{C} = \left(\{d_x\}, H, * \right)$, when $\Char R =2$, consisting of: 
\begin{enumerate}
\item A homology structure, $*$, with {\it multiplicative} co-oriented coefficients in $R^*$ assigned to the curves $*_i$ such that at each trivalent vertex of $*$ the coefficients $s_1, s_2, s_3$ of the three arcs with respect to a cyclically consistent choice of co-orientations (i.e., either all counter-clockwise or all clockwise in coordinates around the vertex) satisfy $s_1s_2s_3 = 1$.
\item A combinatorial spin structure, $\xi$, when $\Char R \neq 2$.  All curves in $\xi$ are assigned the (multiplicative) coefficient $-1$. 
\item A stratified handleslide set, 
\[
H = H_0 \sqcup H_{-1}.
\]
\begin{itemize}
\item The $H_0$ is a collection of formal {\bf handleslides curves} that are immersed curves, $h:X \rightarrow M$, with $X= S^1$ or $[0,1]$ equipped with {\it additive} co-oriented coefficients $b \in R$ and upper and lower endpoint lifts, 
\[
u,l:X \rightarrow \Lambda, \quad \pi_x\circ u = \pi_x \circ l = h,  
\]
satisfying
\[
 z(u(t)) > z(l(t)), \quad \mbox{and} \quad \mu(u(t)) = \mu(l(t)) \in \Z/\rho 
\]
for all $t \in \mbox{Int}(X)$.   
Moreover, on $\mbox{Int}(M)$ the image of $u$ and $l$ is disjoint from cusp edge and swallowtail points.  The endpoints of handleslide curves are subject to requirements specified in (B2) below.

\item The $H_{-1}$ is a set of {\bf super-handleslide points} in $M$ equipped with (i) co-orientations (in $M$, i.e. a choice of orientation of $M$ near the point), (ii) coefficients $b \in R$, and (iii) upper and lower endpoint lifts to points $u, l \in \Lambda \setminus \Lambda_{\mathit{cusp}}$ with $z(u) > z(l)$ and $\mu(u) = \mu(l) -1$.  

\end{itemize}

Except at endpoints as allowed by the Axiom (B2) below, the curves and points in $H_0$ and  $H_{-1}$ are assumed to be self transverse and transverse to one another and to the base projections of $*$, $\xi$, and singularities of $\pi_{xz}(\Lambda)$.  We will refer to a   curve in $H_0$ or a point in $H_{-1}$ having upper and lower endpoint lifts on sheets $S_i$ and $S_j$, $i<j$, as an {\bf $(i,j)$-(super)-handleslide}.  Note that for curves in $H_0$ this terminology only makes sense locally since the numbering of the sheets of $\Lambda$ can change when passing crossings or cusps of $\Lambda$.

\end{enumerate}

 Let $M^\mathcal{C}_\mathit{reg} \subset M$ consist of those points not coinciding with the base projection of any singularity of $\pi_{xz}(\Lambda)$ (crossing, cusp, or swallowtail point), homology or spin structure curve, or with any handleslide arc or super-handleslide point.  We write $M^\mathcal{C}_\mathit{sing} := M \setminus M^\mathcal{C}_\mathit{reg}$, and refer to this subset as the set of {\bf base  singularities} of $\mathcal{C}$.

\begin{enumerate}

\item[(4)] A locally constant collection, $\{d_x\}$, of upper triangular differentials $d_x:C_x \rightarrow C_x$ of degree $+1$ mod $\rho$ defined for each $x \in M_\mathit{reg}$.

\end{enumerate}

The data is subject to the following axioms.  

\begin{enumerate}
\item[(B1)]  The axioms (A1)-(A4) are required to hold along transverse slices to base projections of crossings arcs, cusp arcs, handleslide curves in $H_0$, and curves in $*$ or $\xi$.   
Here, the role of the matrices $A_-$ and $A_+$, as well as the coefficients of handleslides are specified by the orientation of the transverse slice (which specifies a co-orientation for handleslide arcs).     
For homology curves, $*_i$, a  choice of co-orientation of $*_i$ is used to specify the coefficient $s$ and the (local) orientation of the one dimensional slice of $\Lambda$ that appears in (A4).

\item[(B2)]  The only possible endpoints for handleslide arcs are as follows:

\begin{itemize}
\item (Trivalent vertices)  \quad Three handleslide arcs with the same upper and lower endpoints can all share a common endpoint at a trivalent vertex provided that the sum of the coefficients is equal to $0$ with respect to cyclically consistent co-orientations at the vertex.  See Figure \ref{fig:Endpoints0}.

\item (Commutation relation) \quad Whenever an $(i,j)$-handleslide intersects a $(j,k)$-handleslide arc at a point $x \in M$, a   
third $(i,k)$-handleslide arc has an endpoint at $x$.  The coefficients of the three handleslide arcs are related as in Figure \ref{fig:Endpoints1}.

\item (Super handleslides) \quad  Suppose a $(k,l)$-super-handleslide point at $p\in M$ 
has coefficient $c$, and let $A=(a_{ij})$ be the matrix of $d_x$ in a region of $M_\mathit{reg}$ adjacent to $c$.
Then,  for each $i<k$ a single $(i,l)$-handleslide arc with coefficient $a_{ik}c$ has an endpoint at $p$, and for each $l<j$ a single $(k,j)$-handleslide arc with coefficient $ca_{lj}$ has an endpoint at $p$.  
The coefficients of the handleslide arcs are with respect to their counter-clockwise co-orientations around $p$ as specified by the co-orientation of $p$ itself. 
See Figure \ref{fig:Endpoints2}.

\item (Swallowtail points)  \quad Handleslide arcs and spin structures in a neighborhood of swallowtail points are required to appear as in Figure \ref{fig:Endpoints3}.

\item (Boundary points)  \quad Handleslide arcs (resp. homology, and spin structure arcs) may have endpoints at transverse intersections with $\partial M$ 
(resp. $\partial \Lambda$).  Moreover, in a collar neighborhood of $M$ where $\pi_{xz}(\Lambda)$ appears as a product these arcs appear as products (of points on $M$ or $\Lambda$ with the collar coordinate) as well.

\end{itemize}

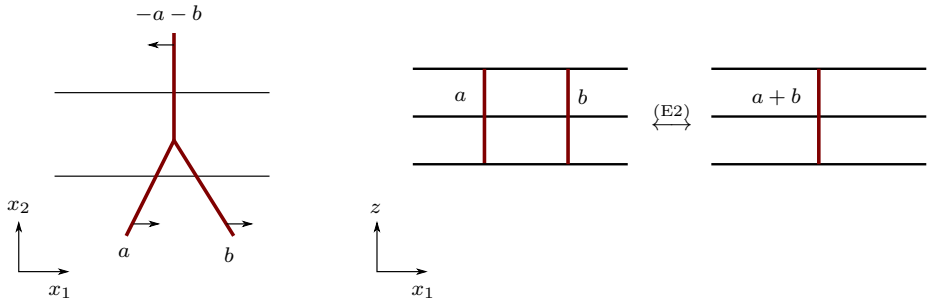
\begin{figure}[!ht]

\quad

\quad

\labellist
\tiny
\pinlabel $a$ [t] at 68 22
\pinlabel $b$ [t] at  134 22
\pinlabel $-a-b$ [b] at 96 154
\pinlabel $a$ [r] at 288 112
\pinlabel $b$ [l] at 348 112
\pinlabel $a+b$ [r] at  498 112
\pinlabel $x_1$ [t] at  28 -2
\pinlabel $x_2$ [b] at  2 34
\pinlabel $x_1$ [t] at  256 -2
\pinlabel $z$ [b] at  226 34
\pinlabel $\stackrel{\mathrm{(E2)}}{\longleftrightarrow}$ at 412 100

\endlabellist
\centerline{\includegraphics[scale=.6]{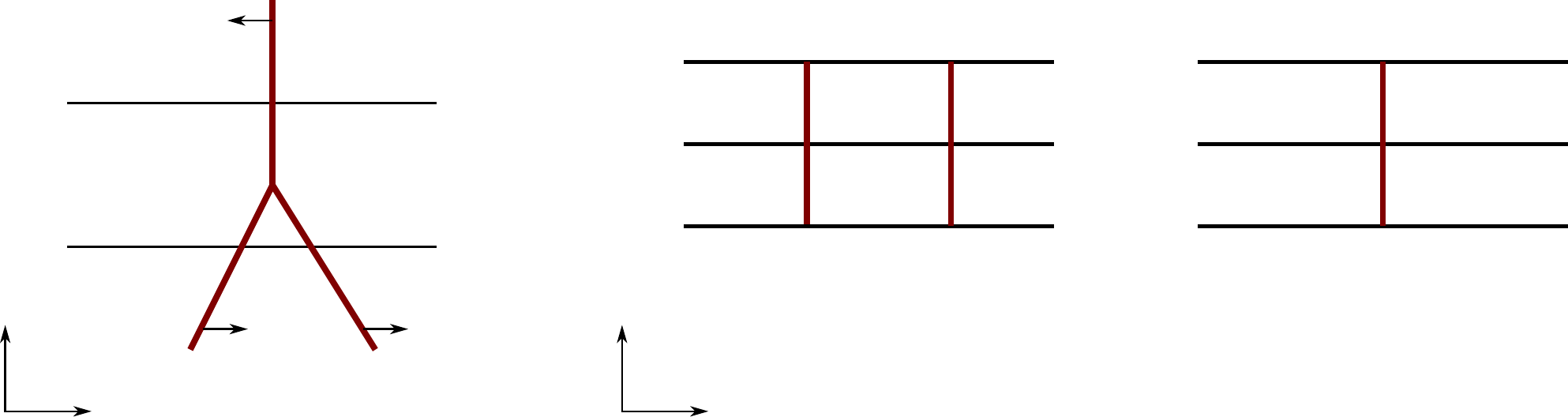}}

\quad

\caption{Handleslide coefficients in the trivalent vertex axiom. In Figures \ref{fig:Endpoints0}-\ref{fig:Endpoints3}  the left column depicts the base projection together with two horizontal lines that indicate slices of the front projection appearing in the right two columns.
The handleslide coefficients in the front projection slices are always written with respect to the left to right co-orientation.  
The label (E2) reflects the enumeration from Section \ref{sec:slice}.}
\label{fig:Endpoints0}
\end{figure}

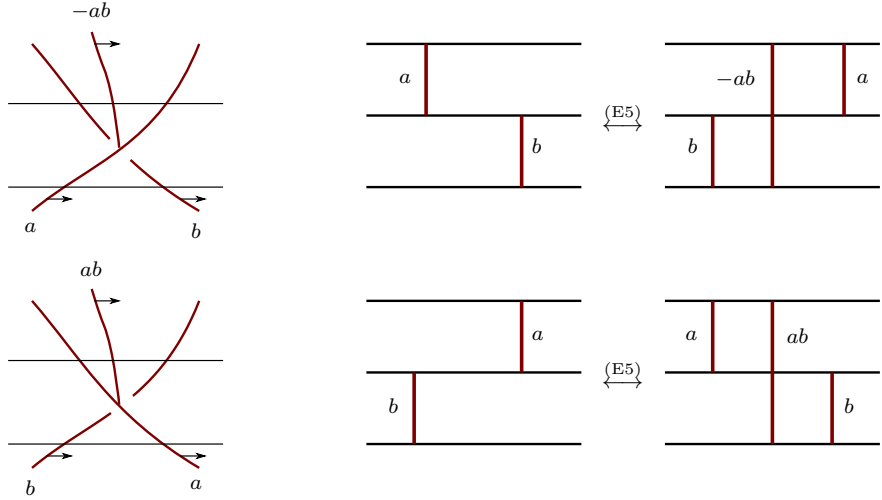
\begin{figure}[!ht]

\labellist
\tiny
\pinlabel $a$ [t] at 14 -2 
\pinlabel $b$ [t] at  118 -2
\pinlabel $-ab$ [b] at 52 116
\pinlabel $a$ [r] at 258 84
\pinlabel $b$ [l] at 324 42
\pinlabel $b$ [r] at 438 42
\pinlabel $a$ [l] at 528 84
\pinlabel $-ab$ [r] at 474 84
\pinlabel $\stackrel{\mathrm{(E5)}}{\longleftrightarrow}$ at 386  60

\endlabellist
\centerline{\includegraphics[scale=.6]{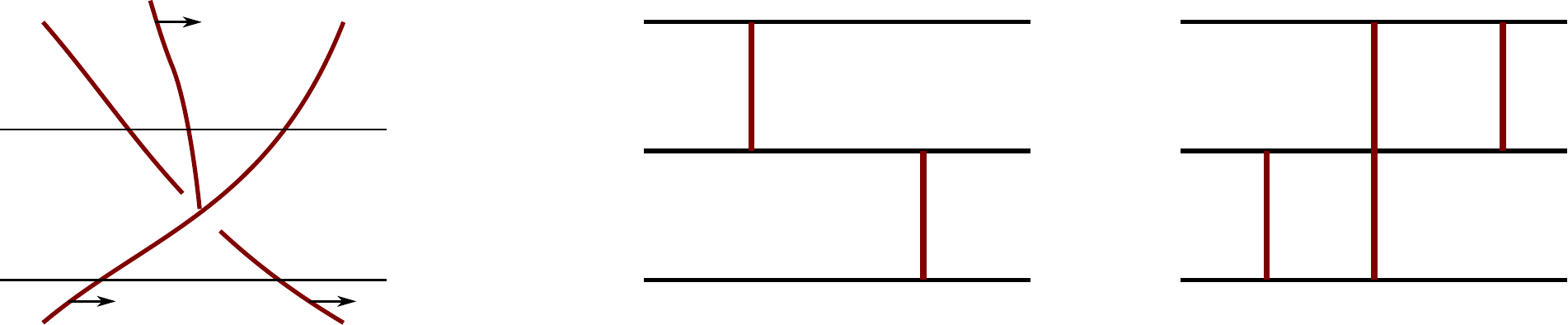}}

\quad

\quad

\labellist
\tiny
\pinlabel $b$ [t] at 14 -2
\pinlabel $a$ [t] at  118 -2
\pinlabel $ab$ [b] at 52 116
\pinlabel $b$ [r] at 250 40
\pinlabel $a$ [l] at 324 84
\pinlabel $a$ [r] at 438 84
\pinlabel $b$ [l] at 520 40
\pinlabel $ab$ [l] at 484 84
\pinlabel $\stackrel{\mathrm{(E5)}}{\longleftrightarrow}$ at 386  60

\endlabellist

\centerline{\includegraphics[scale=.6]{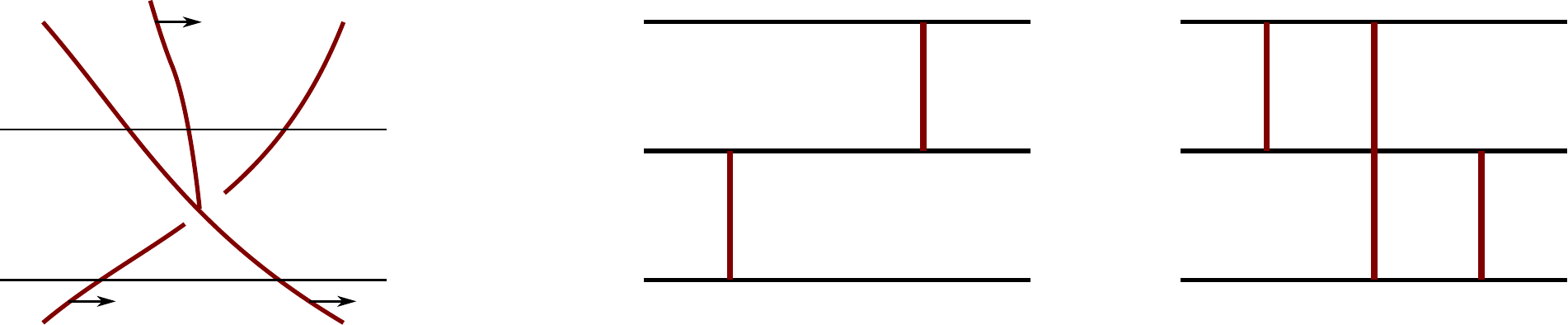}}
\caption{Coefficients in the commutation axiom. 
}
\label{fig:Endpoints1}
\end{figure}

\begin{figure}[!ht]

\labellist
\tiny
\pinlabel $a_{ik}$ [l] at 506 82
\pinlabel $a_{i'k}$ [r] at  468  112
\pinlabel $a_{lj}$ [r] at 438 22
\pinlabel $a_{ik}c$ [l] at 334 82
\pinlabel $a_{i'k}c$ [r] at 288 112
\pinlabel $ca_{lj}$ [r] at  244 22
\pinlabel $\stackrel{\mathrm{(E7)}}{\longleftrightarrow}$ at 378  66

\endlabellist
\centerline{\includegraphics[scale=.6]{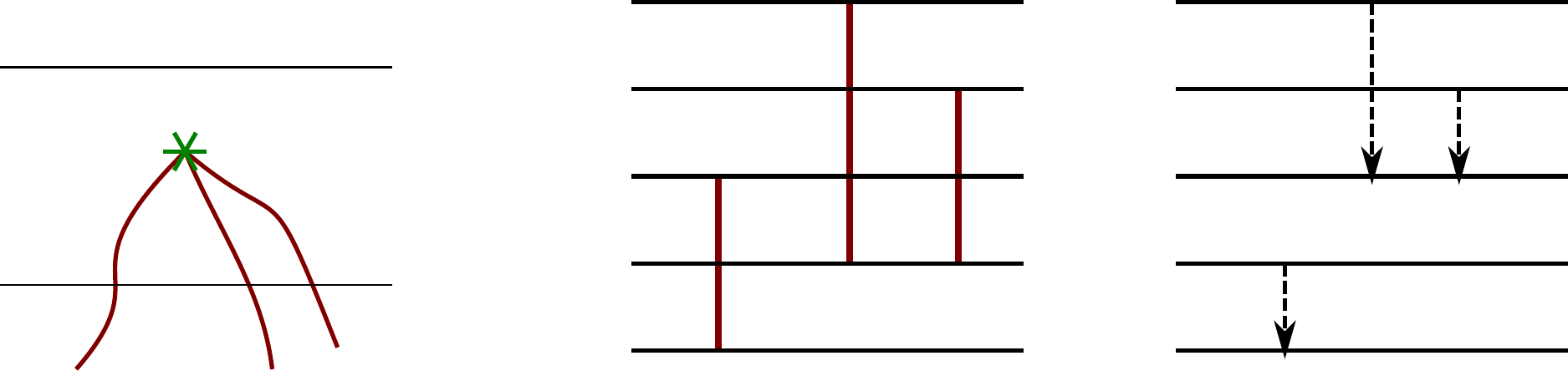}}

\caption{Coefficients in the super-handleslide axiom.}
\label{fig:Endpoints2}
\end{figure}

\begin{figure}[!ht]

\labellist
\tiny
\pinlabel $a_{i,k}$ [r] at 250 86
\pinlabel $a_{i',k}$ [r] at  274 134
\pinlabel $-a_{i,k}$ [r] at 322 86
\pinlabel $-a_{i',k}$ [r] at 346 134
\pinlabel $-1$ [r] at  438 46
\pinlabel $1$ [l] at 514 46
\pinlabel $-a_{i,k}$ [r] at 446 86
\pinlabel $-a_{i',k}$ [r] at 462 134
\pinlabel $\xi$ [l] at 80 12
\pinlabel $\xi$ [t] at 300 26
\pinlabel $\stackrel{\mathrm{(R1)}}{\longleftrightarrow}$ at 386  90

\endlabellist
\centerline{\includegraphics[scale=.6]{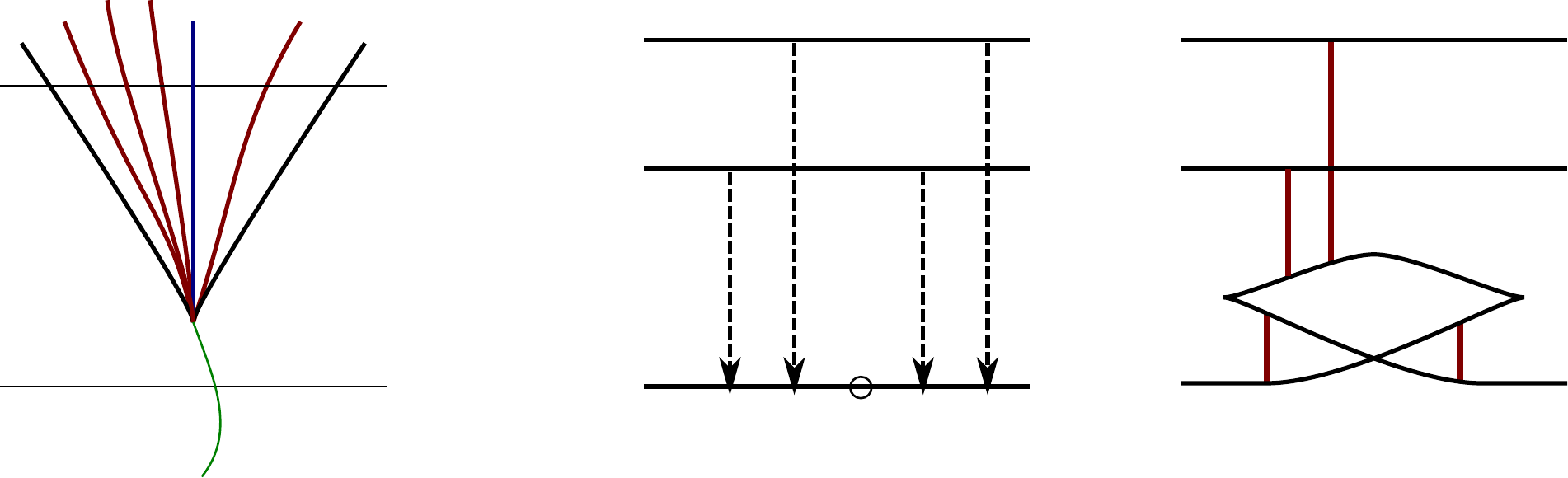}}

\vspace{1cm}

\labellist
\tiny
\pinlabel $x_1$ [l] at 36 2
\pinlabel $x_2$ [b] at 2 36
\pinlabel $x_1$ [l] at  260 2
\pinlabel $z$ [b] at 226 36
\pinlabel $a_{l,j}$ [r] at 254 128
\pinlabel $a_{l,j'}$ [r] at  276 86
\pinlabel $-a_{l,j}$ [r] at 324 128
\pinlabel $-a_{l,j'}$ [r] at 348 86
\pinlabel $1$ [r] at 440 170
\pinlabel $-1$ [l] at  514 170
\pinlabel $a_{l,j}$ [r] at 448 128
\pinlabel $a_{l,j'}$ [r] at 464 86
\pinlabel $\xi$ [l] at 80 44
\pinlabel $\xi$ [t] at 302 176
\pinlabel $\stackrel{\mathrm{(R1)}}{\longleftrightarrow}$ at 386  120

\endlabellist

\centerline{\includegraphics[scale=.6]{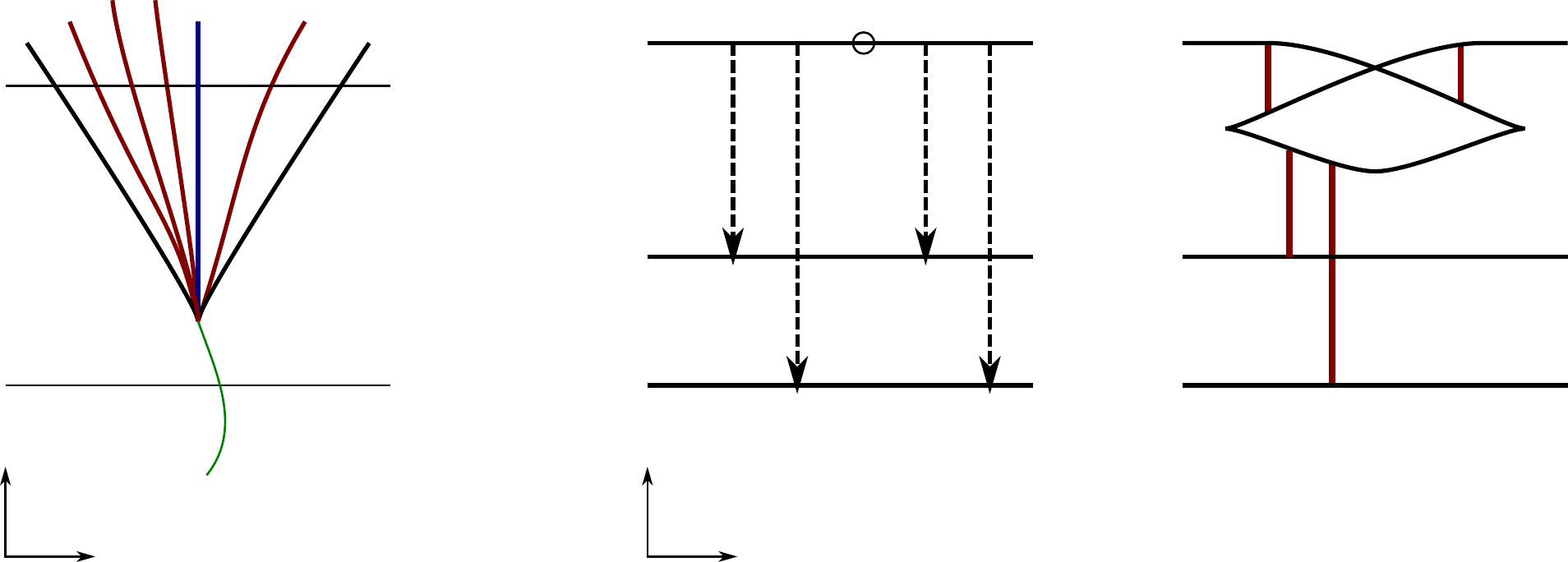}}
\caption{For an upward (resp. downward) swallowtail point, $s_0$, a spin arc $\xi$ on sheet $S_k$ (resp. $S_l$) outside of the swallow tail region, i.e. 
where the additional swallowtail cusp sheets do not exist, has an endpoint at $s_0$ (when $\Char \mathbb{F} \neq 2$). Within the swallowtail region, two $(k+1,k+2)$-handleslide arcs (resp. $(l,l+1)$-handle slide arcs) one on each side of the crossing arc have their endpoint at $s_0$.  In addition, for each $i<k$ (resp. $l<j$) an $(i,k)$- (resp. $(l+2,j+2)$-) handleslide arc with coefficient $-a_{i,k}$ (resp. $a_{l,j}$) inside the swallowtail region has endpoint at $s_0$ where the matrix coefficients  $a_{i,k} = \langle d_x S_k, S_i \rangle$ (resp. $a_{l,j} = \langle d_x S_j, S_l \rangle$) are computed with $x$ to the left of $\xi$.}
\label{fig:Endpoints3}
\end{figure}

\item[(B3)] Handleslide coefficients change when a handleslide endpoint passes a homology or spin structure curve as indicated in Figure \ref{fig:Axiom3}.  

\end{enumerate}

\begin{figure}[!ht]

\labellist
\tiny
\pinlabel $b$ [r] at 26 116
\pinlabel $s$ [t] at 66 82
\pinlabel $bs^{-1}$ [l] at  302 116
\pinlabel $b$ [r] at 26 20
\pinlabel $sb$ [l] at 302 20
\pinlabel $b$ [r] at  492 116
\pinlabel $bs$ [l] at 760 116
\pinlabel $b$ [r] at 492 20
\pinlabel $s^{-1}b$ [l] at 760 20
\pinlabel $\mathrm{(E6)}$ [b] at 172 102
\pinlabel $\mathrm{(E6)}$ [b] at 172 42
\pinlabel $\mathrm{(E6)}$ [b] at 622 42
\pinlabel $\mathrm{(E6)}$ [b] at 622 102

\endlabellist

\centerline{\includegraphics[scale=.6]{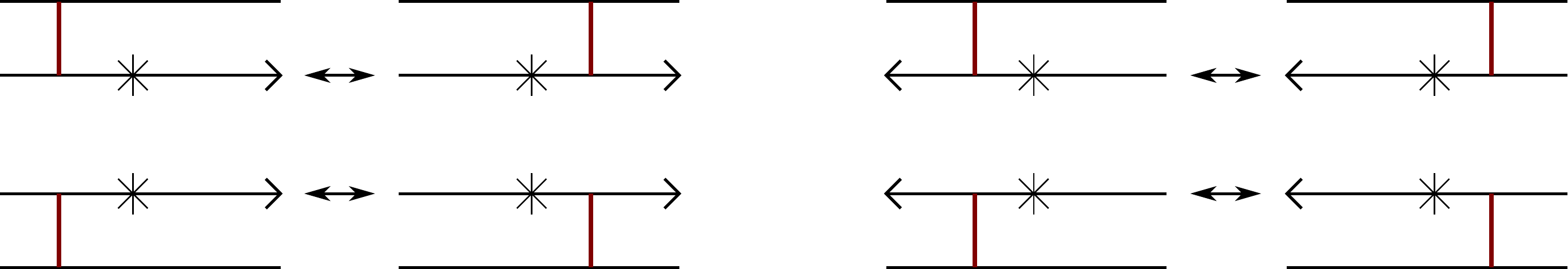}}
\caption{The arrows should be interpreted as the co-orientation of the homology curve $*_i$.}
\label{fig:Axiom3}
\end{figure}
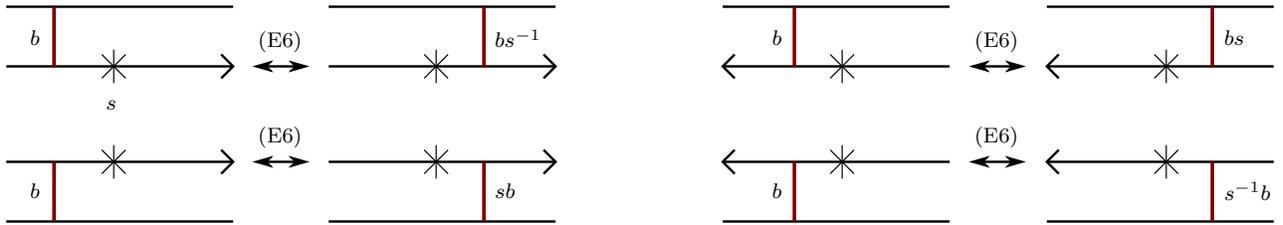

\end{definition}

\begin{example}  \label{ex:2dimMCF}
An example of a $0$-graded MCF for a Legendrian pair of pants  $\Lambda \subset J^1(S^1\times[0,1])$ is pictured in Figure \ref{fig:2DimMCF}.  The differentials $d_x$ vanish for $x=(x_1,x_2)$ where $\Lambda$ is $2$-sheeted, and are uniquely determined by the axioms for other values of $x$.   Results in Section \ref{sec:constructions} will ease considerations involving the differentials $d_x$ when constructing MCFs.
\end{example}

\begin{figure}[!ht]

\labellist
\tiny
\pinlabel $a^{-1}$ [l] at 20 70
\pinlabel $a$ [b] at 398 254
\pinlabel $-a^{-1}$ [b] at 440 254
\pinlabel $-a^{-1}$ [t] at 458 158
\pinlabel $a$ [t] at 492 160
\pinlabel $a^{-1}$ [bl] at 688 240
\pinlabel $a$ [l] at 720 224

\pinlabel $a$ [r] at  450 206
\pinlabel $1$ [r] at 690 206

\pinlabel $a^{-1}$ [b] at  610 110
\pinlabel $a$ [b] at 766 14

\pinlabel $1$ [b] at  392 106
\pinlabel $-1$ [b] at 424 106
\pinlabel $-1$ [t] at  474 4
\pinlabel $1$ [t] at 506 4

\pinlabel $\mu=1$ [r] at  354 254
\pinlabel $\mu=0$ [r] at 354 156

\endlabellist

\centerline{\includegraphics[scale=.6]{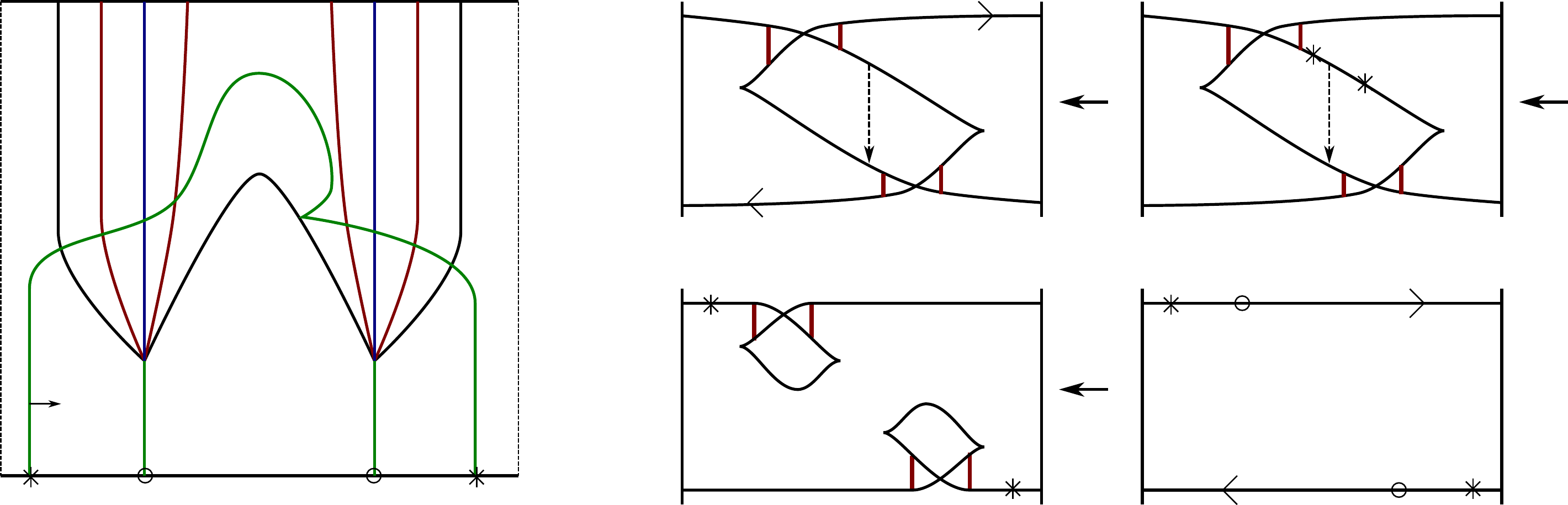}}
\caption{An example of a $2$-dimensional MCF $\mathcal{C}$ for $\Lambda \subset J^1(S^1\times[0,1])$ defined for any choice of $a \in \mathbb{F}^*$. 
On the left, the base projection of cusp (black), crossing (blue), handleslide (red), and homology and spin structure (green) curves are pictured.  Slices of the front projection at different values of $x_2 \in [0,1]$ appear on the right, with the arrows indicating the positive $x_2$ direction.  Except for the homology structure, all coefficients are with respect to left to right co-orientations.  The coefficient of the homology curve with respect to the global co-orientation indicated in the base projection is $a^{-1}$.  However, the coefficients of homology basepoints in the front slices are written   with respect to the co-orientation specified by the orientation of the slice of $\Lambda$.}
\label{fig:2DimMCF}
\end{figure}

\begin{remark}  \label{rem:MC2F}
\begin{enumerate} 
\item  
It may be worthwhile to emphasize that for handleslide arcs the (additive) co-oriented coefficients are with respect to co-orientations in $M$ while for homology structure curves $*_i$ the co-orientations are in $\Lambda$.  When $\mathit{char} R =2$, the co-orientation can be ignored when considering coefficients of handleslides, but this is not so for the $*_i$.  If $M$ (resp. $\Lambda$) is oriented one could equivalently use orientations rather than co-orientations for specifying coefficients of handleslide (resp. $*_i$) arcs.  However, we will make use of 
 non-orientable Legendrians, eg. when $\rho$ is odd the cobordisms constructed in the proofs of Theorems \ref{thm:s1} and \ref{thm:01} may be non-orientable.

\item  
Note that the coefficients $a_{ik}$ and $a_{lj}$ of $d_x$ that appear in (B2) (Super handleslides) are not affected by the handleslide arcs with endpoints at $p$, and hence are independent of the choice of the point $x$ near $p$.  If one associates the matrices $C=cE_{k,l}$ to the $(k,l)$-super-handleslide at $p$ with coefficient $c$, and $B_m = b_{m} E_{i_m,j_m}$ for $1\leq m \leq M$ to the $(i_m,j_m)$-handleslides with endpoints at $p$ where the $b_m$ are their (counterclockwise) coefficients, then the axiom at handleslide points is equivalent to  
the identity
\begin{equation} \label{eq:superRemark}
A_xC +CA_x = \sum_m B_m = (I+B_m) \cdots (I+B_2)(I+B_1) - I.
\end{equation}

\item   
The configuration of handleslide arcs with endpoints at a swallowtail arc specified in (B2) (Swallowtail points) is essentially the only one allowed in order for the $1$-dimensional axioms (A1)-(A3) to hold when passing through a transverse slice containing the swallowtail sheets in the right column of Figure \ref{fig:Endpoints2}.  Moreover, the effect on coefficients of $d_x$ when passing through  this slice from left to right is to negate all coefficients of arrows beginning or ending on the swallowtail point, and this explains the necessity of spin arc endpoints at swallowtail points.  See Section \ref{sec:STunique}.

\end{enumerate}
\end{remark}

Note that for any $1$-dimensional Legendrian $\Lambda \subset J^1M$ with MCF, $\mathcal{C}$, there is a {\bf product MCF} on $\Lambda \times [0,1]$ that we will denote by $\mathcal{C}\times[0,1]$ where all of the data $(\{d_x\}, H,*, \xi)$ is constant in the $[0,1]$ direction.  Here, the co-oriented coefficients on $H$ and $*$ are assigned from those of $\mathcal{C}$ using the co-orientations arising from the orientations of $M$ and $\Lambda$.

\subsection{Continuation maps and monodromy}  \label{sec:2-4}

Let $\calC$ be an MCF for a Legendrian $\Lambda \subset J^1M$ of dimension $1$ or $2$, and let $\sigma : [0,1] \rightarrow M$ be a path that is transverse to the base singularities of $\calC$.  That is, $\sigma$ passes through  base projections of cusps, crossings, handleslide arcs, and basepoint arcs transversally at a finite collection of times $0< s_1 < s_2< \ldots < s_N < 1$, and $\sigma$ does not intersect co-dimension $2$ base singularities of $\calC$. 
On each of the intervals in $[0,1] \setminus \{s_1, s_2, \ldots, s_N \}$, the complexes $(C_{\sigma(s)}, d_{\sigma(s)})$  do not change.  Choose $0 = s_0' < s_1'< \ldots < s_N' =1$  
with $s_{i-1}' < s_i < s_i'$ for $1 \leq i \leq N$, and write 
$x_i = \sigma(s_i')$.   
For each $i$, the axioms (A1)-(A4) give rise to a quasi-isomorphism
\[
f_i: (C_{x_{i-1}}, d_{x_{i-1}})\to (C_{x_i}, d_{x_i})
\]
defined as follows depending on the type of singularity of $\mathcal{C}$ that occurs at $s_i$:
\begin{enumerate}
\item When $s_i$ is a {\it left cusp}, i.e. when two new sheets in positions $k$ and $k+1$ appear  as $s$ increases past the cusp point at $s_i$, we have   
\begin{equation} \label{eq:lcuspmap} f_i(S_j) = \begin{cases} S_j, & \mbox{ if } j< k;\\
S_{j+2}, & \mbox{ if } j \geq k.\end{cases}
\end{equation}
\item When $s_i$ is a {\it right cusp}, i.e. sheets $S_k$ and $S_{k+1}$ that exist at $x_{i-1}$ meet at the cusp point as $s$ increases,
\begin{equation}  \label{eq:rcuspmap}
f_i(S_j)= \begin{cases} S_j, & \mbox{ if } j< k;\\
S_{j-2}, & \mbox{ if } j>k+1\\
0 , & \mbox{ if } j=k,k+1.\end{cases}
\end{equation}
\item When $s_i$ is a {\it crossing} between sheets $S_k$ and $S_{k+1}$,  
 \begin{equation} \label{eq:crossingmap} f_i(S_j) = \begin{cases} S_{k+1}, & \mbox{ if } j=k;\\
 S_k, & \mbox{ if } j= k+1;\\
S_j, & \mbox{ if } j\neq k, k+1.\end{cases}
\end{equation}
\item When $s_i$ is a {\it handleslide} with upper and lower end points sheets $k$ and $l$ (with $k<l$) and coefficient $a \in R$ with respect to the co-orientation specified by $\sigma$,
\begin{equation} \label{eq:handleslidemap}
f_i(S_j)= \begin{cases} aS_k+S_l, & \mbox{ if } j= l;\\
S_j, & \mbox{ if } j\neq l.\end{cases}
\end{equation}
\item When $s_i$ is a {\it basepoint} on sheet $k$ with coefficient $t \in R^*$ with respect to the co-orientation specified by locally lifting $\sigma$ to $\Lambda$,
\begin{equation} \label{eq:basepointmap}
f_i(S_j)= \begin{cases} tS_k, & \mbox{ if } j= k;\\
S_j, & \mbox{ if } j\neq k.\end{cases}
\end{equation}
\end{enumerate}

Composing all of the chain maps as they occur along $\sigma$
we get a {\bf continuation map} 
$$f_\sigma: (C_{x_0}, d_{x_0})\to (C_{x_1}, d_{x_1}), \quad f_\sigma = f_N \circ \cdots \circ f_1.$$
We refer to the homology $H^*(C_{x}, d_{x})$ at $x \in M^\mathcal{C}_\mathit{reg}$ as the {\bf fiber cohomology} of $\mathcal{C}$ at $x$.  The continuation maps induce {\bf continuation isomorphisms} on the fiber cohomologies
$$\phi_{\sigma}: H^*(C_{x_0}, d_{x_0}) \stackrel{\cong}{\rightarrow} H^*(C_{x_1}, d_{x_1}).
$$

\begin{example} 
The MCF from Example \ref{ex:2dimMCF} has $H^0(C_{x}, d_{x}) \cong H^1(C_{x}, d_{x}) \cong \mathbb{F}$.  Viewing $S^1 = [0,1]/\{0,1\}$, for any $x_2 \in [0,1]$ the horizontal loop $\sigma:[0,1] \rightarrow S^1\times[0,1]$, $\sigma(t) = ([t],x_2)$ has continuation isomorphism $\phi_\sigma= -a^{-1} \mathrm{id}$. 
\end{example}

The fundamental properties of this construction are summarized in:

\begin{proposition} \label{prop:continuation} Let $\sigma,\tau:[0,1]\rightarrow M$ be paths transverse to $M^\mathcal{C}_\mathit{sing}$.
\begin{enumerate}
\item If $\sigma$ and $\tau$ are path homotopic (i.e. homotopic with endpoints fixed), then $f_\sigma$ and $f_\tau$ are chain homotopic.
\item If $\sigma(1)=\tau(0)$, then $f_{\sigma* \tau} = f_{\tau} \circ f_{\sigma}$.
\end{enumerate}
In particular,  the continuation isomorphisms provide a well-defined anti-homomorphism
\[
\Phi_{\Lambda, \mathcal{C}}:\pi_1(M,x_0) \rightarrow GL(H^*(C_{x_0}, d_{x_0})), \quad [\sigma] \mapsto \phi_{\sigma},
\]
i.e. $\Phi_{\Lambda,\mathcal{C}}([\sigma_1]\cdot[\sigma_2]) = \Phi_{\Lambda,\mathcal{C}}([\sigma_2])\cdot\Phi_{\Lambda,\mathcal{C}}([\sigma_1])$. 
\end{proposition}

The anti-homomorphism $\Phi_{\Lambda,\mathcal{C}}$ is called the {\bf monodromy representation} of $(\Lambda,\mathcal{C})$.  

\begin{remark} The proposition is established with $R = \mathbb{F}_2$ coefficients in \cite[Section 4.2]{RuSu3}, and a proof using the same method carries through in the present setting.  For completeness, and because of the more general coefficients, we have included a proof in Section \ref{sec:slice} below.  
\end{remark}

\section{MCFs and augmented Legendrian cobordism}  \label{sec:MCFandAug}

The section begins with notions of equivalence and cobordism for MCFs, and then makes the connection with augmentations of the Legendrian contact homology DGA using the results from \cite{PanRu2}.  Specifically, the notion of equivalence for MCFs translates to DGA homotopy for augmentations, and there is a bijection between MCFs and augmentations up to cobordism;  see Theorem \ref{thm:augbiject}.  While $\mathbb{F}_2$ coefficients are used in the cited literature, we include a discussion about general coefficient fields.

\subsection{Equivalence of MCFs and augmented Legendrian cobordism}  \label{sec:equivcob}

Note that when $M$ has boundary, any MCF, $\mathcal{C}$, for a Legendrian surface $\Lambda \subset J^1M$ can be restricted to a boundary component $B \subset \partial M$ to produce a $1$-dimensional MCF, denoted $\mathcal{C}|_B$ for $\Lambda|_B$.  As a slight extension of this construction, when $\Lambda \subset J^1([0,1]\times[0,1])$ we can also restrict MCFs from $\Lambda$ to the Legendrian tangles sitting above any of the boundary edges $[0,1]\times\{i\}$ or $\{i\} \times [0,1]$ for $i=0,1$.

\begin{definition} Let $(\Lambda, *, \xi)$ be a $1$-dimensional Legendrian link or tangle in $J^1M$ equipped with fixed collections of homology and spin base points, $*$ and $\xi$.  We say that two $\rho$-graded MCFs, $\mathcal{C}_0$ and $\mathcal{C}_1$, for $(\Lambda, *, \xi)$ are {\bf equivalent}  
if there exists a $\rho$-graded $2$-dimensional MCF $\mathcal{C}$ on the trivial cobordism $\Lambda \times [0,1]$ having
\begin{itemize}
\item homology and spin structure $*\times[0,1]$ and $\xi\times[0,1]$, and
\item  $\mathcal{C}|_{M\times\{i\}} = \mathcal{C}_i$ for $i =0,1$.
\item In the case where $M=[0,1]$ we require that  the handleslide set  of $\mathcal{C}$ is disjoint from a neighborhood of $\partial M \times [0,1]$.
\end{itemize}
\end{definition}

The terminology used in the following definitions is justified by the connection between MCFs and augmentations of the Legendrian contact homology dg-algebra; see Section \ref{sec:aug} below.

\begin{definition}
A {\bf ($\rho$-graded) augmented Legendrian}  in $J^1M$ is a pair $(\Lambda, \calC)$ consisting of a ($\rho$-graded) Legendrian $(\Lambda, \mu)$ in $J^1M$ with a Morse complex family $\calC$.
\end{definition}
Often we suppress the explicit reference to $\rho$ from terminology with the understanding that all augmented Legendrians under consideration are $\rho$-graded with respect to some fixed $\rho \geq 0$.
\begin{definition}
Two augmented Legendrians  $(\Lambda_0,\calC_0)$ and $(\Lambda_1,\calC_1)$ in $J^1 M$, where $M$ is $S^1$ or $[0,1]$, are {\bf cobordant} if there is an augmented Legendrian cobordism $(\Lambda, \calC)$ in $J^1(M \times  [0,1])$ such that 
in neighborhoods of $M\times\{i\}$, for $i=0,1$, $(\Lambda, \calC)$ agrees with the product MCF, $(\Lambda_i\times[0,1], \mathcal{C}_i\times[0,1])$. 
When $M = [0,1]$, we require in addition that there is a neighborhood $N$ of $\dd M \times I$ such that
\begin{itemize} \item in $N\times\R_z$ the front projection of $\Lambda$ consists of $n$ non-singular, non-intersecting sheets, and
\item all curves and points in $H$, $*$, and $\xi$ have their base projections disjoint from $N$.  In particular, the differentials $\{d_x\}$ are constant in $x$ on both components of $\dd M \times I$.
\end{itemize}
\end{definition}

Note that cobordism of augmented Legendrians defines an equivalence relation.

\subsection{MCFs and augmentations of the Legendrian contact homology DGA}   \label{sec:aug}

Recall that for compact Legendrian submanifolds in $1$-jet spaces the {\bf Legendrian contact homology dg-algebra} (abbrv. {\bf LCH DGA}) was constructed by Ekholm, Etnyre,  and Sullivan in \cite{EESR2n+1, EES}.  In its most basic form, the LCH algebra of $\Lambda \subset J^1M$, denoted as $(\mathcal{A}(\Lambda), \partial)$, is a free associative $\mathbb{F}_2$-algebra with identity generated by the Reeb chords of $\Lambda$ and graded by 
$\Z/\rho\Z$ when $\Lambda$ is $\rho$-graded.  
The degree $-1$ differential, $\partial$, is defined by a count of rigid holomorphic disks.  When $\Lambda$ is equipped with a spin structure, the coefficients can be upgraded from $\mathbb{F}_2$ to $\mathbb{Z}$, see \cite{EESorientation}, and there are various ways to incorporate  the group ring $\mathbb{F}_2[H_1(\Lambda)]$ or $\mathbb{Z}[H_1(\Lambda)]$ into $(\mathcal{A}(\Lambda), \partial)$ by paying attention to the boundary values of holomorphic disks, see eg. \cite{EESR2n+1} or \cite[Section 2]{EENS} for the ``fully non-commutative'' DGA.

A {\bf $\rho$-graded augmentation} of $\Lambda$ to $\mathbb{F}$ is a unital homomorphism of differential algebras 
\[
\epsilon: (\mathcal{A}(\Lambda), \partial) \rightarrow (\mathbb{F}, 0), \quad \epsilon(1) =1, \quad \epsilon \circ \partial =0
\] 
that preserves the $\Z/\rho$-grading, i.e.  
\[
\epsilon(x) \neq 0 \quad \Rightarrow \quad |x| = 0 \,(\mbox{mod $\rho$}).
\]
Two $\rho$-graded augmentations, $\epsilon_1$ and $\epsilon_2$, are {\bf DGA homotopic} if there exists an $(\epsilon_1, \epsilon_2)$-derivation 
\[
H:\mathcal{A}(\Lambda) \rightarrow \mathbb{F}, \quad H(xy) = H(x)\epsilon_2(y) +(-1)^{|x|} \epsilon_1(x) H(y)
\]
 of degree $1$ mod $\rho$ such that $\epsilon_1-\epsilon_2 = H \circ \partial$.  Here, when $\rho$ is odd it is required that $\mathit{Char}\, \mathbb{F} =2$.  Denote the set of all $\rho$-graded augmentations of $\Lambda$ to $\mathbb{F}$ by $\mathit{Aug}_\rho(\Lambda; \mathbb{F})$, and denote the set of DGA homotopy classes 
by  $\mathit{Aug}_\rho(\Lambda; \mathbb{F})/{\sim}$.

\begin{remark}
When considering augmentations to $\mathbb{F}_2$, eg. for computing $\mathit{Aug}_\rho(\Lambda; \mathbb{F}_2)$ or $\mathit{Aug}_\rho(\Lambda; \mathbb{F}_2)/{\sim}$, we can work with the version of $\mathcal{A}$ generated by Reeb chords only, without needing to make use of 
homology classes in $H_1(\Lambda)$.  This is because
any $\mathbb{F}_2$-valued augmentation must map all elements of $H_1(\Lambda) \subset \mathcal{A}^\times$ to $\mathbb{F}_2^\times =\{1\}$.  
\end{remark}

Consider now a $\rho$-graded Legendrian cobordism $\Sigma \subset J^1(M\times I)$ with $M=\R$ or $S^1$ between $\Lambda_0$ and $\Lambda_1$.  As in \cite{EK}, after modifying $\Sigma$ near $M\times \partial I$ to have Morse minimum ends and working with an appropriate class of almost complex structures, the LCH DGA of $\Sigma$ is well-defined and contains the DGAs of $\Lambda_0$ and $\Lambda_1$ as sub-DGAs.  See also \cite[Section 5.3]{PanRu1} and \cite[Section 2.5]{PanRu2}.  Thus, we have a diagram
\[
\mathcal{A}(\Lambda_1) \hookrightarrow \mathcal{A}(\Sigma) \hookleftarrow \mathcal{A}(\Lambda_0).
\] 
For augmentations $\epsilon_0 \in \mathit{Aug}(\Lambda_0; \mathbb{F})$ and $\epsilon_1 \in \mathit{Aug}(\Lambda_1;\mathbb{F})$ we say that $(\Lambda_0, \epsilon_0)$ and $(\Lambda_1, \epsilon_1)$ are {\bf cobordant} if there exists a Legendrian cobordism $\Sigma$ from $\Lambda_0$ and $\Lambda_1$ with an augmentation $\epsilon$ such that $\epsilon|_{\mathcal{A}(\Lambda_i)} \simeq \epsilon_i$ 
(DGA homotopy).

The works \cite{Henry, HenryRu2} established over $\mathbb{F}_2$ a bijection between MCFs and DGA homotopy classes of augmentations
for Legendrians in $J^1\R$, see \cite{PanRu2} for the case of Legendrians in $J^1S^1$.  A correspondence between MCFs and augmentations for closed Legendrian surfaces is established in \cite{RuSu2}.   
In \cite{PanRu1}, 
it is shown that a bijection arises as well at the level of cobordism classes.  More precisely, for $M = \R$ or $S^1$ let $\mathit{Cob}^{\mathit{MCF}}_\rho(J^1M;\mathbb{F}_2)$ (resp.  $\mathit{Cob}^{\mathit{Aug}}_\rho(J^1M;\mathbb{F}_2)$) denote the set of cobordism classes of pairs $(\Lambda, \mathcal{C})$ (resp. $(\Lambda, \epsilon)$) consisting of $\rho$-graded Legendrians in $J^1M$ equipped with MCFs (resp. augmentations) defined over $\mathbb{F}_2$.  Let $\mathit{MCF}_\rho(\Lambda;\mathbb{F}_2)/{\sim}$ denote the set of equivalence classes of $\rho$-graded MCFs for a Legendrian $\Lambda \subset J^1M$.

\begin{theorem} \label{thm:augbiject} Let $M= \R$ or $S^1$.  For any $\Lambda \subset J^1M$ after modifying $\Lambda$ by a Legendrian isotopy induced by a planar isotopy of $\pi_{xz}(\Lambda)$ there is a bijection  
\[
\Phi: \mathit{MCF}_\rho(\Lambda;\mathbb{F}_2)/{\sim} \stackrel{\cong}{\rightarrow} \mathit{Aug}_{\rho}(\Lambda;\mathbb{F}_2)/{\sim}.
\]
Moreover, the maps $\Phi$ induce a bijection 
\[
\mathit{Cob}^{\mathit{MCF}}_\rho(J^1M;\mathbb{F}_2) \, \cong \,\mathit{Cob}^{\mathit{Aug}}_\rho(J^1M;\mathbb{F}_2).
\]
\end{theorem}

\begin{proof} 
The bijection $\Phi$ is from   \cite[Proposition 5.17]{PanRu2}.
That the $\Phi$ collectively 
induce a bijection on cobordism classes follows from \cite[Corollary 4.12 and Proposition 5.19]{PanRu2} and the characterization of induced augmentation sets in terms of Morse minimum cobordisms obtained in \cite[Proposition 2.15]{PanRu2}.   
We discuss the details a little more carefully using terminology from \cite{PanRu2}.

Using the {\it immersed LCH functor} constructed in \cite{PanRu1}, Section 3 of \cite{PanRu2} associates to a ($\rho$-graded) {\it conical Legendrian cobordism}, $\Sigma:\Lambda_0 \rightarrow \Lambda_1$ with $\Lambda_i \subset J^1M$, an {\it induced augmentation set}, \[
I_{\Sigma} \subset\mathit{Aug}_{\rho}(\Lambda_0;\mathbb{F}_2)/{\sim}  \times  \mathit{Aug}_{\rho}(\Lambda_1;\mathbb{F}_2)/{\sim}
\]
that is a conical Legendrian isotopy invariant of $\Sigma$.  The key point for our current discussion is the following.

\begin{lemma}  Two augmentations $\epsilon_i \in \mathit{Aug}_{\rho}(\Lambda_i;\mathbb{F}_2)$, $i=0,1$ belong to the same equivalence class in $\mathit{Cob}^{\mathit{Aug}}_\rho(J^1M;\mathbb{F}_2)$ if and only if there exists a conical Legendrian cobordism $\Sigma:\Lambda_0 \rightarrow \Lambda_1$ such that $([\epsilon_0], [\epsilon_1]) \in I_{\Sigma}$.   
\end{lemma}

\begin{proof}
Assuming the $(\Lambda_i, \epsilon_i)$ are cobordant, we have a Morse minimum cobordism $\Sigma_{\mathit{min}}:\Lambda_0 \rightarrow \Lambda_1$ together with an augmentation $\alpha$ of $\alg(\Sigma_{\mathit{min}})$ that restricts to $\epsilon_i$ on $\alg(\Lambda_i)$.  Applying \cite[Proposition 2.15]{PanRu2}, when we modify $\Sigma_{\mathit{min}}$ near its ends to form an associated conical Legendrian cobordism $\Sigma_{\mathit{conic}}$ the {\it immersed DGA map} associated to $\Sigma_{\mathit{conic}}$ is {\it immersed homotopic} to 
\[
\alg(\Lambda_1)  \stackrel{i_1}{\hookrightarrow} \alg(\Sigma_{min}) \stackrel{i_0}{\hookleftarrow} \alg(\Lambda_0).
\]
It follows from the definition and homotopy invariance of the {\it induced augmentation set} (see \cite[Definition 3.1 and Proposition 3.2]{PanRu2}) that $([\epsilon_0],[\epsilon_1]) = ([i^*_1\alpha], [i^*_0\alpha]) \in I_{\Sigma_{\mathit{conic}}}$.  

Conversely, suppose that there is a conical cobordism $\Sigma:\Lambda_0 \rightarrow \Lambda_1$ such that $([\epsilon_0], [\epsilon_1]) \in I_{\Sigma}$.  Applying a conical Legendrian isotopy supported near the ends of $\Sigma$, we can arrange that $\Sigma$ is the conical cobordism associated to a Morse minimum cobordism $\Sigma_{\mathit{min}}$; this is accomplished as in the constructions in \cite[Section 2.5]{PanRu2}.  Then, \cite[Proposition 2.15]{PanRu2} shows that there is an augmentation of $\alg(\Sigma_{\mathit{min}})$ whose restrictions to the $\alg(\Lambda_i)$ agree with the $\epsilon_i$ up to DGA homotopy.
\end{proof}

Now, for any Legendrians $\Lambda_0, \Lambda_1 \subset J^1M$, after modifying each $\Lambda_i$ by a particularly simple Legendrian isotopy, $\mathcal{L}_i$, to $\widetilde{\Lambda_i}$ we have that for any (compact) Legendrian cobordism $\Sigma \subset J^1(M \times[0,1])$ from $\Lambda_0$ to $\Lambda_1$ there is a commutative diagram of bijections and inclusions
\begin{equation} \label{eq:diagram}
\begin{array}{ccccc}
\mathit{MCF}_\rho(\Lambda_0)/{\sim} \times \mathit{MCF}_\rho(\Lambda_1)/{\sim} & \stackrel{\cong}{\longrightarrow} 
& \mathit{Aug}_\rho(\widetilde{\Lambda}_0)/{\sim} \times \mathit{Aug}_\rho(\widetilde{\Lambda}_1)/{\sim}  \\  
\rotatebox[origin=t]{90}{\reflectbox{$\hookleftarrow$}} & 
& \rotatebox[origin=t]{90}{\reflectbox{$\hookleftarrow$}} \\
I^\mathit{MCF}_\Sigma  
& \stackrel{\cong}{\longrightarrow} & I_{\widetilde{\Sigma}}
\end{array}
\end{equation}
where $I^\mathit{MCF}_\Sigma$ is the set of pairs $(\mathcal{C}_0,\mathcal{C}_1)$ that can be obtained from restricting an MCF on $\Sigma$ to $\Lambda_0$ and $\Lambda_1$ and $\widetilde{\Sigma}$ is an extension of $\Sigma$ to a conical (non-compact) Legendrian cobordism.  The arrow in the top row is the composition of bijections from \cite[Proposition 5.17 and Corollary 4.12]{PanRu2} that is the map $\Phi$ in the statement of the current proposition.  That the diagram is commutative follows from \cite[Proposition 5.19 and Corollary 4.12]{PanRu2}.  (In \cite{PanRu2} the bijection factors through an intermediary set,   $\mathit{Aug}_\rho(\Lambda_0, \mathcal{E}_{||})/{\sim} \times \mathit{Aug}_\rho(\Lambda_1, \mathcal{E}_{||})/{\sim}$, that consists of augmentations of the cellular DGAs of the $\Lambda_i$.)  The diagram (\ref{eq:diagram}) together with the lemma show that the $\Phi$ induce a well-defined injective map $\mathit{Cob}^{\mathit{MCF}}_\rho(J^1M;\mathbb{F}_2) \rightarrow \mathit{Cob}^{\mathit{Aug}}_\rho(J^1M;\mathbb{F}_2)$.  To see that the map is also surjective, we use that since $\Lambda_i$ and $\widetilde{\Lambda}_i$ are Legendrian isotopic, for any augmentation $\epsilon_i$ of $\mathcal{A}(\Lambda_i)$ the pair $(\Lambda_i, \epsilon_i)$ is cobordant to $(\widetilde{\Lambda}_i, \widetilde{\epsilon}_i)$ for some augmentation $\widetilde{\epsilon_i}$ of $\mathcal{A}(\widetilde{\Lambda_i})$.  (Here, we use that the Legendrian isotopy $\mathcal{L}_i$ leads to a conical Legendrian isotopy $\Sigma_{\mathcal{L}_i}:\Lambda_i \rightarrow \widetilde{\Lambda}_{i}$ with embedded Lagrangian projection.  In this situation, there is a DGA homomorphism $f: \alg(\widetilde{\Lambda_i}) \rightarrow \alg(\Lambda_i)$ so that the induced augmentation set has the form $I_{\Sigma_{\mathcal{L}_i}} = \{([\epsilon], [f^*\epsilon])\,|\, \epsilon \in \mathcal{A}(\Lambda_i)\}$.  See \cite[Section 3.1.1]{PanRu2}.  Then, apply the lemma again.)

\end{proof}

\subsection{Remarks about general $\mathbb{F}$} \label{sec:generalF}
We expect that the bijections from Theorem \ref{thm:augbiject} should remain valid  
 with  an arbitrary field $\mathbb{F}$ provided   that one works with an appropriate version of the LCH DGA.   The approach taken in \cite{PanRu2} is extendable to the case when $\Char \mathbb{F}=2$ as follows.
 With the formulation of MCFs over $\mathbb{F}$ used in this article, the bijections $\Phi:\mathit{MCF}_\rho(\Lambda;\mathbb{F})/{\sim} \stackrel{\cong}{\rightarrow} \mathit{Aug}_{\rho}(\Lambda;\mathbb{F})/{\sim}$ will take equivalence classes of MCFs of $\Lambda$ with a fixed collection of homology basepoints, $* = \{*_1, \ldots, *_r\}$, to the set of DGA homotopy classes of augmentations of the multi-basepointed DGA, $\mathcal{A}(\Lambda, *)$, as defined in \cite{NgR}.  Recalling that $\mathcal{A}(\Lambda, *)$ has one invertible generator $t_i^{\pm1}$ for each $*_i$, we remark that we require DGA homotopy operators to satisfy $H(t_i)=0$, and  under the bijection the coefficient $s_i \in \mathbb{F}^\times$ associated to the basepoint $*_i$ as part of an MCF $\mathcal{C}$ satisfies $\epsilon_i(t_i) = s_i^{-1}$.
We also clarify the meaning of cobordism for augmentations of the multi-basepointed DGAs.  When considering augmentations $\epsilon_i: \mathcal{A}(\Lambda_i, *^i) \rightarrow \mathbb{F}$ for $i=0,1$, one should allow Morse minimum cobordisms $\Sigma \subset J^1(M \times[0,1])$ equipped with homology structures $*$ that agree with $*^{i}\times[0,1]$ near the boundary components $\Lambda_i\times\{i\}$ of $\Sigma$ for $i=0,1$.   A $2$-dimensional analog of the multi-basepointed DGA, $\mathcal{A}(\Sigma, *)$,  then results by associating invertible generators $u^{\pm1}_i$ to the (co-orientations of the) edges of $*$, subject 
to the relations $u_1u_2u_3= 1$ (assuming cyclically consistent co-orientation) at trivalent vertices. The differential of Reeb chords is modified to record intersections of the boundary of holomorphic disks with $*$ using appropriate factors of $u_i^{\pm1}$.  The pairs $(\Lambda_i, \epsilon_i)$, $i=0,1$, should then be declared cobordant if there exists some $(\Sigma, *)$ with an augmentation $\alpha: \mathcal{A}(\Sigma,*) \rightarrow \mathbb{F}$ satisfying $h_i^*\alpha \simeq \epsilon_i$ where $h_i: \mathcal{A}(\Lambda_i, *^i) \rightarrow \mathcal{A}(\Sigma,*)$ is the DGA map that is the inclusion on Reeb chords and maps $t_i \mapsto u_j^{\pm1}$ appropriately.  In fact, a closely related construction is used in the context of embedded Lagrangian cobordisms in the 
the work of Gao-Shen-Weng, \cite[Section 2.4]{GSW}.

The proof of Theorem \ref{thm:augbiject} above, essentially from \cite{PanRu2}, factors through the cellular DGA from \cite{RuSu1, RuSu2, RuSu25} (defined over $\mathbb{F}_2$) and the $2$-dimensional correspondence between MCFs and augmentations of the cellular DGA (over $\mathbb{F}_2$) from \cite{RuSu3}.    In establishing the isomorphism between the cellular DGA and the LCH DGA, the relevant rigid holomorphic disks, or rather their degenerations to gradient flow trees, are all enumerated in \cite{RuSu2,RuSu25}.  Thus, it is not difficult to record intersections between their boundaries and the curves that make up a homology structure.  When this is done the isomorphism between the cellular DGA and the LCH DGA may be upgraded to incorporate homology structures, and the arguments of \cite{PanRu2} and \cite{RuSu3} may be carried forward.  As this would take us somewhat far afield, we do not consider these details in the present article.

We expect that the bijection continues to hold as well when $\Char \mathbb{F} \neq 2$, although it is not clear that the method of proof from \cite{PanRu2} can be pushed through  
with the readily available techniques.  The cellular DGA can be formulated with signs by making use of a combinatorial spin structure in the presence of swallowtail points;  see \cite{RuSu4} for a formulation in all dimensions for Legendrians with at worst cusp singularities.  However, it is more difficult to extend the isomorphism with LCH to include signs.  The key issue is to be able  to explicitly compute the orientation sign of a holomorphic disk from the combinatorics of the corresponding GFT.  See \cite{EENS, Karlsson} for progress on this difficult technical problem.

\section{Constructions of MCFs}  \label{sec:constructions}

In this section we establish some tools for efficiently constructing MCFs and augmented Legendrian cobordisms that we will apply in later sections.  In Section \ref{sec:slice}, we present a complete list of 
local moves on $1$-dimensional slices that can be used to build up an arbitrary augmented Legendrian cobordism out of elementary ones.  For ease of use, it is helpful that one can typically perform these moves 
without having to pay undue attention to the differentials $\{d_x\}$; see Proposition \ref{prop:MCF}.    In Section \ref{sec:D4minus}, we illustrate the moves by building up the (useful but non-elementary) $D_4^-$ cobordism.  Sections \ref{sec:properly} and \ref{sec:SRform} establish standard forms for handleslide sets and MCFs up to equivalence that simplify later arguments.  Finally, Section \ref{sec:LegIso} records that MCFs can always be extended over cobordisms arising from Legendrian isotopies.

\subsection{$2$-dimensional MCFs via movies of slices} \label{sec:slice}

A $2$-dimensional augmented Legendrian, $(\Sigma, \mathcal{C})$, in $J^1(M \times [0,1])$  
can be built up out of elementary cobordisms that modify the $1$-dimensional slices in a local way as described in the following moves.  
As above, coefficients of handleslides that intersect a slice $M\times\{t\}$ are always specified with respect to the co-orientation arising from the standard orientation  of $M$ (left to right in figures);   
co-orientations for base points are indicated by arrows in the figures. 

\medskip

\noindent {\bf Equivalence moves:}   

\begin{enumerate}
\item[(E1)]  Two adjacent $(i,j)$-handleslides with coefficients $a$ and $-a$ can be cancelled or created.
\item[(E2)]  Two adjacent $(i,j)$-handleslides with coefficients $a$ and $b$ can be interchanged with a single $(i,j)$-handleslide with coefficient $a+b$.
\item[(E3)]  An $(i,j)$-handleslide can be moved past the $x$-value of a crossing or a cusp as long as the endpoints trace out a continuous path on $\Lambda \setminus \Lambda_{\mathit{cusp}}$.  
\item[(E4)]  An $(i,j)$-handleslide can be commuted with a $(k,l)$-handleslide as long as $j \neq k$ and $i\neq l$. 
\item[(E5)]  Commuting an $(i,j)$-handleslide and a $(j,k)$-handleslide results in the appearance of a new $(i,k)$-handleslide with coefficient as in Figure \ref{fig:Endpoints1}.
\item[(E6)]  An $(i,j)$-handleslide may be commuted past a homology or spin base point provided its coefficient changes as in Figure \ref{fig:Axiom3}.
\item[(E7)]  A group of handleslides is created or removed by passing a super-handleslide point as in Figure \ref{fig:Endpoints2}. 
\end{enumerate}

\medskip

\noindent {\bf Base point moves:}
\begin{enumerate}
\item[(F1)]  Two adjacent homology or spin base points with inverse coefficients, $s$ and $s^{-1}$, with respect to a consistent co-orientation can be cancelled or created.
\item[(F2)]  Two adjacent homology base points with coefficients, $s$ then $t$, with respect to a consistent co-orientation 
can be replaced with a single base point with coefficient $ts$, and vice-versa.  
\item[(F3)]  Basepoints can be moved continuously (in $\Lambda$) through cusps or crossings without change to coefficients.
\item[(F4)]  Homology basepoints can be commuted with spin base points.
\end{enumerate} 

\medskip

\begin{figure}[!ht]
\labellist
\small
\pinlabel $a$ [r] at 18 268
\pinlabel $-a$ [l] at  56 268
\pinlabel $\mbox{(E1)}$ [b] at 104 272
\pinlabel $\mbox{(E3)}$ [b] at 322 272
\pinlabel $\mbox{(E3)}$ [b] at 570 272
\pinlabel $\mbox{(E4)}$ [b] at 196 182
\pinlabel $\mbox{(E4)}$ [b] at 464 182
\pinlabel $a$ [r] at 244 254
\pinlabel $a$ [r] at 394 254
\pinlabel $a$ [r] at 484 250
\pinlabel $a$ [r] at 656 250 
\pinlabel $a$ [r] at 110 198
\pinlabel $b$ [r] at 138 152
\pinlabel $a$ [r] at 274 198
\pinlabel $b$ [r] at 244 152
\pinlabel $a$ [r] at 378 152
\pinlabel $b$ [r] at 410 198
\pinlabel $a$ [r] at 544 152
\pinlabel $b$ [r] at 514 198
\pinlabel $s$ [b] at 110 98
\pinlabel $t$ [b] at 410 98
\pinlabel $ts$ [b] at 532 98
\pinlabel $s$ [b] at 388 98
\pinlabel $s^{-1}$ [b] at 138 98
\pinlabel $\mbox{(F1)}$ [b] at 192 92
\pinlabel $\mbox{(F2)}$ [b] at 468 92
\pinlabel $\mbox{(F3)}$ [b] at 118 22
\pinlabel $\mbox{(F3)}$ [b] at 328 22
\pinlabel $\mbox{(F4)}$ [b] at 530 22
\pinlabel $s$ [t] at 34 0
\pinlabel $s$ [b] at 198 34
\pinlabel $s$ [t] at 256 0
\pinlabel $s$ [b] at 378 34
\pinlabel $s$ [b] at 446 28
\pinlabel $s$ [b] at 606 28
\pinlabel $-1$ [b] at 470 28
\pinlabel $-1$ [b] at 580 28

\endlabellist
\includegraphics[width=7in]{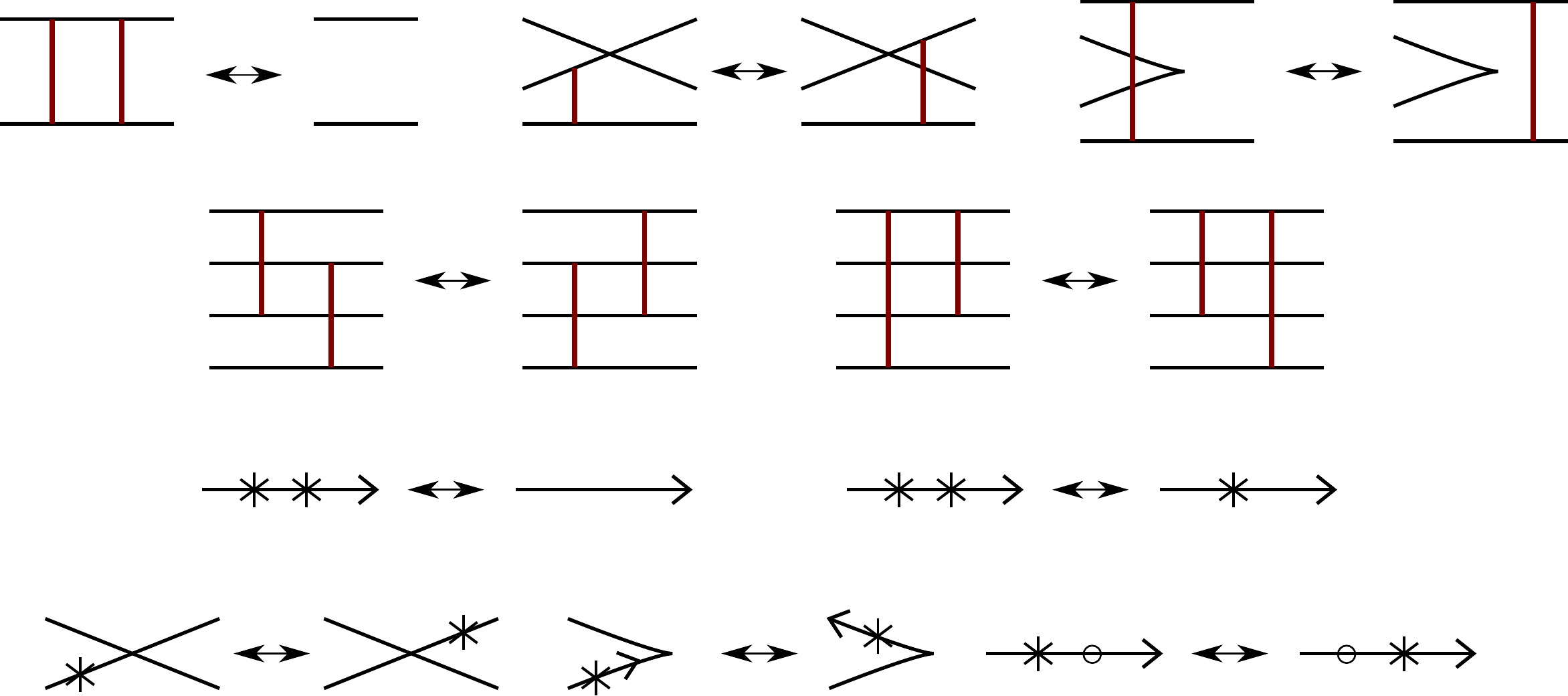}
\caption{Examples of equivalence and basepoint moves.  Note that moves (E2), (E5), (E6), and (E7) are pictured in the figures in Section \ref{sec:2dimMCF}.}
\label{fig:EMoves}
\end{figure}
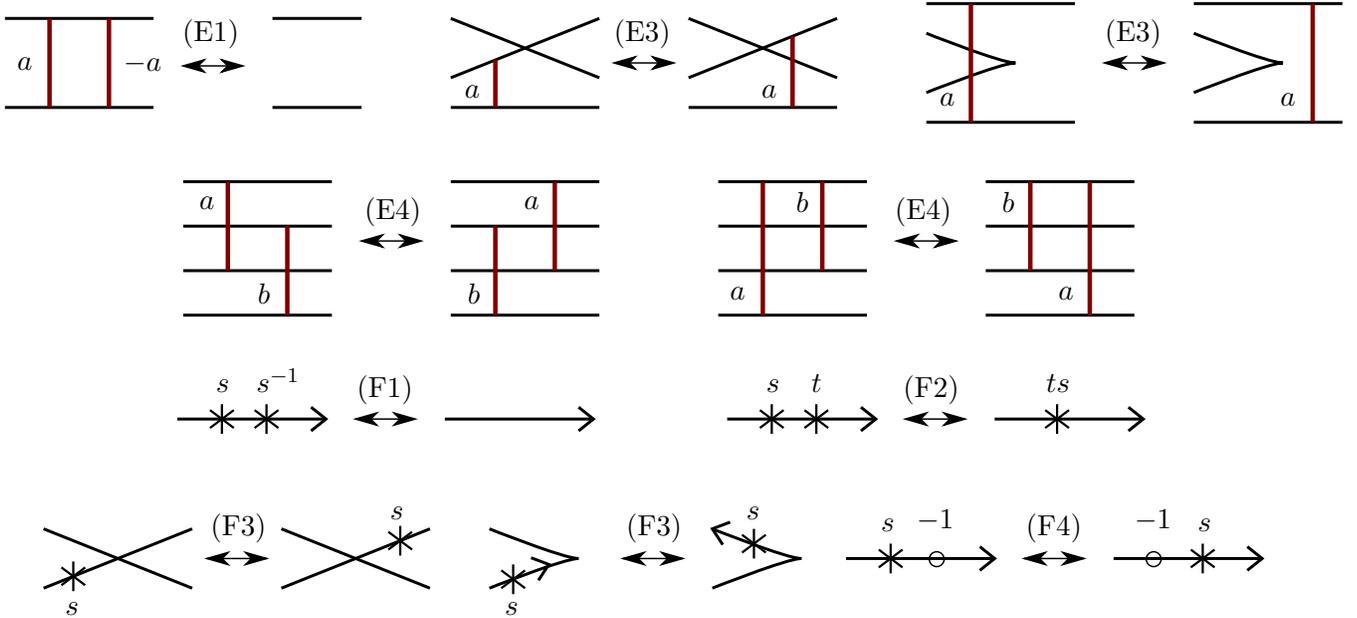

\noindent {\bf Legendrian isotopy moves:}  

\begin{enumerate}
\item[(R0)]  A pair of front singularities (two crossings, two cusps, or a crossing and a cusp) appearing in separate $z$-intervals in the front projection can have their $x$-coordinates pass one another.  

\item[(R1)]  The standard Legendrian Reidemeister Move 1 can be performed provided the handleslides and spin base points are positioned on the two sides of the move as in Figure \ref{fig:Endpoints3}.

\item[(R2)] and (R3)  Legendrian Reidemeister Moves 2 and 3 can be performed in $x$-intervals that does not contain handleslides.

\end{enumerate}

\medskip

\begin{figure}

\labellist
\small
\pinlabel $\mbox{(R1)}$ [b] at 120 50
\pinlabel $\mbox{(R2)}$ [b] at 382 50
\pinlabel $\mbox{(R3)}$ [b] at  630 50

\endlabellist

\includegraphics[width=6in]{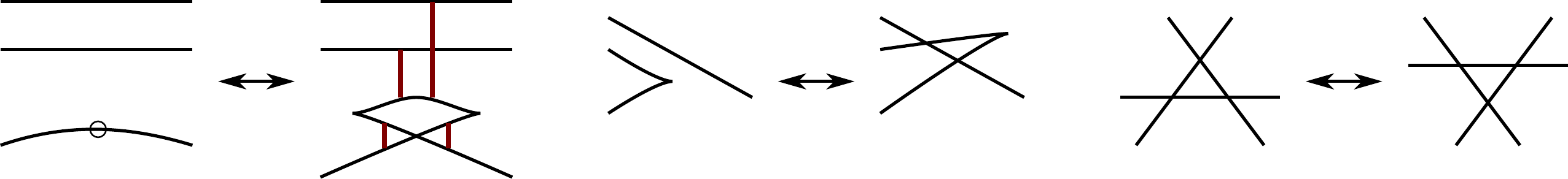}
\caption{Legendrian Reidemeister moves.  Reflections of (R1) (vertical) and (R2) (horizontal and vertical) are allowed.  The coefficients of handleslides in the (R1) move are as in   Figure \ref{fig:Endpoints3}.
}
\label{fig:RMoves}
\end{figure}

\noindent {\bf Legendrian cobordism moves:}

\begin{enumerate}
\item[(C1)] (Unknot Move)  A standard Legendrian unknot component, $U$, disjoint from other components in the front projection can be added or removed provided that there are no handleslide endpoints or basepoints on $U$.

\item[(C2)]  (Pinch Move)  Two sheets, $S_k$ and $S_{k+1}$, defined near $x=x_0$ with no other sheets appearing between them (in the $z$-direction) such that $d_x S_{k+1}= S_k$
and $\langle d_x S_{i},S_k\rangle = \langle d_x S_i, S_{k+1}\rangle = 0$ for $i \neq k,k+1$ can be interchanged with a left cusp followed by a right cusp having equal Maslov potential values (in $\Z/\rho$) on their upper sheets as in Figure \ref{fig:CMoves}.

\item[(C3)]  (Clasp Move)  Two consecutive crossings involving the same sheets, and without any handleslides located between them can be removed.  Conversely, for two adjacent sheets with $\langle d_xS_{k+1}, S_k \rangle =0$ two consecutive crossings can be added near $x$.
\end{enumerate}

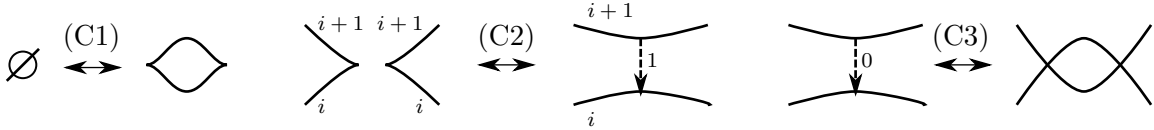
\begin{figure}

\labellist
\small
\pinlabel $\mbox{(C1)}$ [b] at 48 26
\pinlabel $\mbox{(C2)}$ [b] at 280 26
\pinlabel $\mbox{(C3)}$ [b] at  534 26
\tiny
\pinlabel $1$ [l] at 354 26
\pinlabel $0$ [l] at 474 26
\pinlabel $i+1$ [l] at 170 46
\pinlabel $i$ [l] at 170 0
\pinlabel $i+1$ [r] at 238 46
\pinlabel $i$ [r] at 238 0
\pinlabel $i+1$ [bl] at 322 46
\pinlabel $i$ [tl] at 322 0

\endlabellist

\includegraphics[width=6in]{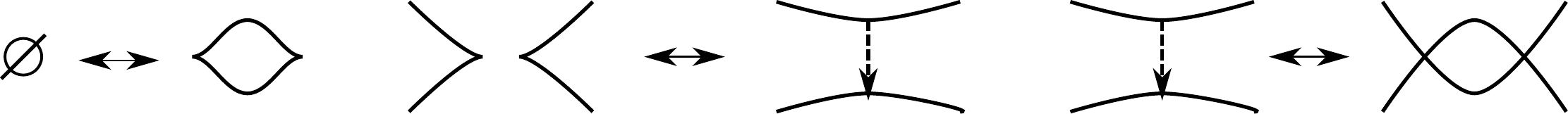}
\caption{The cobordism moves for MCFs.    
To perform $\stackrel{(C2)}{\leftarrow}$ it is required that $d_x S_{k+1}= S_k$,
and $S_k$ and $S_{k+1}$ must not appear in the differentials of any other sheets.  For $\stackrel{(C2)}{\rightarrow}$ the Maslov potentials must take the same values, $i, i+1 \in \Z/\rho$, near the left and right cusp. To perform $\stackrel{(C3)}{\rightarrow}$ $S_k$ must not appear in $d_xS_{k+1}$, although $S_k$ and $S_{k+1}$ may appear in the differentials of other sheets.  No handleslides may appear in the $x$-intervals where these local moves are performed.  In particular, there cannot be any handleslides inside the unknot or clasp in move (C1) and (C3).  
}
\label{fig:CMoves}
\end{figure}

In all of the moves, the augmented Legendrians $(\Lambda_0, \mathcal{C}_0)$ and $(\Lambda_1, \mathcal{C}_1)$ appearing on the two sides of the move are identical outside an interval $(x_a,x_b) \subset M$ where the move takes place, and no singularities of the MCFs $(\Lambda_i, \mathcal{C}_i)$ (crossings, cusps, handleslides, or basepoints) besides those referred to in the statement of the move appear in the interval $(x_a,x_b)$.  However, any number of other sheets of $\Lambda$, not pictured in the figures, without singularities in the front projection may appear as long as their positions are allowed by the statement of the move.

For each move there is a corresponding {\bf elementary augmented Legendrian cobordism} $(\Sigma, \mathcal{C}) \subset J^1(M\times[0,1])$ from $(\Lambda_0, \mathcal{C}_0)$ to $(\Lambda_1, \mathcal{C}_1)$ such that the MCF $(\Sigma, \mathcal{C})$ either (i) contains a unique co-dimension $2$ base singularity of $(\Sigma, \mathcal{C})$, or (ii) contains a single critical point of the $t$-coordinate function, $M\times[0,1]  \rightarrow [0,1]$, restricted to some co-dimension $1$ base singularity of $(\Sigma,\mathcal{C})$.  

\medskip
 
\begin{center}
\begin{tabular}{c|c}
Singularity of $(\Sigma, \mathcal{C})$ & Move  \\
\hline 
$\begin{array}{c}
\mbox{Two cusps, crossings, handleslides, or basepoints} \\ \mbox{with the same $x$-coordinate} \end{array}$  & (E3), (E4), (E5), (E6), (F3), (F4), (R0) \\
& \\
$\begin{array}{c}
\mbox{Local max or min of $t$ restricted to a} \\ 
\mbox{handleslide, basepoint, cusp, or crossing arc} \end{array}$ & (E1), (F1), (C1)-(C3) \\
 & \\
Codimension $2$ front singularity of $\Sigma$ &  (R1)-(R3) \\
 & \\
Trivalent vertex of handleslide or basepoint arcs & (E2), (F2) \\
 & \\
Super-handleslide point & (E7)
\end{tabular}
\end{center}

\medskip

We have not specified explicitly how the differentials $\{d_x\}$ for $x \in(x_a,x_b)$ are changed when a move occurs.  The following proposition guarantees their existence.

\begin{proposition} \label{prop:MCF} Let $(\Lambda_0, \mathcal{C}_0) \subset J^1M$ be an augmented Legendrian corresponding to one side (either one) of one of the above moves.   
\begin{enumerate}
\item There exists an augmented Legendrian $(\Lambda_1, \mathcal{C}_1)$ together with a cobordism $(\Sigma, \mathcal{C}) \subset J^1(M\times[0,1])$ between $(\Lambda_0, \mathcal{C}_0)$ and $(\Lambda_1, \mathcal{C}_1)$ such that the singular sets 
 of $(\Lambda_1, \mathcal{C}_1)$ and $(\Sigma, \mathcal{C})$ are as specified in the description of the move and the corresponding elementary cobordism.  
\item The family of differentials $\{d_x\}$ for $(\Lambda_1, \mathcal{C}_1)$ and $(\Sigma, \mathcal{C})$ are uniquely determined by $(\Lambda_0, \mathcal{C}_0)$.
\item The continuation maps $f_{\sigma_0}$ and $f_{\sigma_1}$ associated to the paths 
\[
\sigma_0 = [x_a,x_b] \times \{0\} \quad \mbox{and} \quad  \sigma_1 = (\{x_a\} \times [0,1]) * ([x_a,x_b] \times \{1\}) *(\{x_a\} \times [0,1])^{-1}
\]
 are chain homotopic.   
\end{enumerate}
\end{proposition}

Note that for the equivalence moves (E1)-(E7) over $\mathbb{F}_2$ a detailed proof of (1) appears in \cite[Proposition 3.8]{Henry}.

\begin{proof}
Within the rectangle $[x_a, x_b]\times[0,1]$, the preceding table indicates how the elementary cobordism $\Sigma$, as well as its the handleslide set $H= H_0 \sqcup H_{-1}$, and its homology and spin structures, $*$ and $\xi$, are all specified by the different moves. Outside of this rectangle, $(\Sigma, \mathcal{C})$ agrees with the product cobordism.  Thus, it just  needs to be verified that differentials $\{d_{(x,t)}\}$ can be defined in $[x_a, x_b]\times[0,1]$ to agree with those from $\mathcal{C}_0$ when $t=0$ and satisfy the axioms (A1)-(A4). 
Observe in all cases:
\begin{itemize}
\item[(i)] Each component of $(M\times[0,1])^\mathcal{C}_\mathit{reg}$ in $[x_a,x_b]\times[0,1]$ borders at least one of $[x_a,x_b]\times\{0\}$ or $[x_a,x_b]\times\{1\}$; 
\item[(ii)]  out of these components, only the regions that border $\{x_a\} \times [0,1]$ or $\{x_b\}\times[0,1]$ intersect both $[x_a,x_b]\times\{0\}$ and $[x_a,x_b]\times\{1\}$; and
\item[(iii)]  each codimension $1$ base singularity of $(\Sigma, \mathcal{C})$ (crossing, cusp, handleslide, homology arc, or spin arc) intersects at least one of $M\times\{0\}$ and $M\times\{1\}$. 
\end{itemize}
Thus, if one can define the differentials for $\mathcal{C}_1$ so that 
\begin{itemize}
\item[(a)] axioms (A1)-(A4) hold along $[x_a,x_b]\times \{1\}$, and
\item[(b)] $\mathcal{C}_1$ agrees with $\mathcal{C}_0$ at $x=x_a$ and $x=x_b$,
\end{itemize}
 then, after extending the differentials of $\mathcal{C}_0$ and $\mathcal{C}_1$ to all of $[x_a,x_b]\times[0,1]$ to be constant in the components of $(M\times[0,1])^\mathcal{C}_\mathit{reg}$,  it follows that the axioms (A1)-(A4) hold everywhere for the $2$-dimensional MCF $(\Sigma,\mathcal{C})$.  This would complete the construction.  Moreover, since (a) and (b) must hold, $d_{x_a}$ and (A1)-(A4) uniquely specify the remaining differentials along $[x_a,x_b]\times \{1\}$ so that the uniqueness statement (2) of the proposition follows. 

The existence of the differentials $\{d_x\}$ for $\mathcal{C}_1$ satisfying (a) and (b), as well as the existence of the chain homotopy from (3) of the proposition, is verified by a case-by-case check.  We explain the procedure that should be followed in general, and then examine the more involved cases in detail. Let $d^0_{x_a}$ and $d^0_{x_b}$ denote the differentials from $\mathcal{C}_0$ at $x_a$ and $x_b$.  Choose a sequence of $x$-values  $x_a=x_0 < x_1 < \ldots < x_N = x_b$ such that a single singularity, $s$, of $(\Lambda_1, \mathcal{C}_1)$ (crossing, cusp, handleslide, or basepoint) occurs in $[x_{i-1}, x_i]$ for $1 \leq i \leq N$.  Inductively construct the differentials $d_{x_0}, \ldots, d_{x_N}$ for $\mathcal{C}_1$, starting with $d_{x_0} =d^0_{x_a}$.  Next, $d_{x_{i}}$ is determined by $d_{x_{i-1}}$ by the singularity in $[x_{i-1},x_i]$ as uniquely specified by the Axioms (A1)-(A4).  This goes as follows:

\begin{itemize}
\item In the case of a {\it crossing}, {\it handleslide}, or {\it basepoint}, define
\begin{equation} \label{eq:dxi}
d_{x_i} = f_i d_{x_{i-1}} f_i^{-1}
\end{equation}  
where $f_i$ is the continuation map from (\ref{eq:crossingmap}), (\ref{eq:handleslidemap}), or (\ref{eq:basepointmap}).  In the case of a crossing, it is important to check that $\langle d_{x_{i-1}}S_{k+1}, S_k \rangle =0$ so that $d_{x_{i}}$ will still be upper triangular.

\item In the case of a {\it left cusp}, use the continuation map (\ref{eq:lcuspmap}) to identify
\[
C_{x_i} \cong C_{x_{i-1}} \oplus C_{\mathit{cusp}} \quad  \mbox{where} \quad C_\mathit{cusp} = \mbox{Span}_R\{S_k,S_{k+1}\}
\]
and define $d_{x_i}$ to be the direct sum
\begin{equation} \label{eq:split}
d_{x_i} = d_{x_{i-1}} \oplus d_0 \quad \mbox{where} \quad d_0S_{k+1} = S_k.
\end{equation}

\item In the case of a {\it right cusp},  it is important to verify that with respect to the direct sum decomposition $C_{x_{i-1}} = C_{x_{i}} \oplus C_{\mathit{cusp}}$ the differential $d_{x_{i-1}}$ has the form $d_{x_{i-1}} = d' \oplus d_0$.  Then, define $d_{x_i} = d'$.
\end{itemize}

During this constructive procedure, the required condition at crossings and right cusps is seen to hold by either the hypothesis on the move in (C1)-(C3) or by the fact that the axioms are known to be valid for $\mathcal{C}_0$.  Finally, one needs to verify that the resulting differential $d_{x_N}$ agrees with $d^0_{x_b}$ from $\mathcal{C}_0$, and check that the two continuation maps $f_{\sigma_0}$ and $f_{\sigma_1}$ are chain homotopic. (These two checks are usually somewhat interrelated.)  Note that since there are no singularities along the segments $\{x_a\}\times[0,1]$ and $\{x_b\} \times[0,1]$, $f_{\sigma_0}$ and $f_{\sigma_1}$ are the composition 
\[
f_{\sigma_0} = f^0_{N_0} \circ \cdots \circ f^0_1, \quad f_{\sigma_1} = f^1_{N_1} \circ \cdots \circ f^1_1
\]
of the maps  (\ref{eq:lcuspmap})-(\ref{eq:basepointmap}) from Section \ref{sec:2-4} associated to the sequence of singularities appearing along $[x_a,x_b]\times\{0\}$ and $[x_a,x_b] \times\{1\}$ respectively.

We comment now on the specifics of the above procedure for individual moves.

\medskip

\noindent {\it When there are no cusps in $[x_a,x_b]\times[0,1]$:}

\medskip

From (\ref{eq:dxi}) we have
\[
d_{x_N} = f_{\sigma_1} d^0_{x_a} f_{\sigma_1}^{-1},
\]
and we know that 
\[
d^0_{x_b} = f_{\sigma_0} d^0_{x_a} f_{\sigma_0}^{-1}.
\]
In all cases except for the super-handleslide move (E7), that is for the moves (E1), (E2), (E4)-(E6), (F1), (F2), (F4), (R3), (C3) and the non-cusp cases of (E3), (F3), (R0), quick matrix computations show that $f_{\sigma_0} = f_{\sigma_1}$, eg. (E5) follows from the ``Steinberg relation'', $[aE_{ij}, bE_{jk}] = abE_{ik}$, and (R3) from the celebrated braid relation for the permutation matrices, $Q_{(k \, k+1)} Q_{(k+1 \, k+2)} Q_{(k \, k+1)} = Q_{(k+1 \, k+2)} Q_{(k \, k+1)} Q_{(k+1 \, k+2)}$.  This simultaneously verifies that $d_{x_N} = d^0_{x_b}$ and $f_{\sigma_0}$ and $f_{\sigma_1}$ are chain homotopic.

In the case of the super-handleslide move (E7), using the matrix reformulation of the axiom (B2) discussed in Remark \ref{rem:MC2F}, the product $(I+B_m) \cdots (I+B_1)$ that appear in (\ref{eq:superRemark}) is the matrix of the continuation map, $f$, on the side of the (E7) move  where the handleslides appear while $I$ is the matrix for the other continuation map.  Thus, equation (\ref{eq:superRemark}) translates to the identity
\begin{equation} \label{eq:E7}
d_{x_a} K +K d_{x_a} = f - \mbox{id}
\end{equation}
where  $K(S_l) = cS_k$ and $K(S_i) = 0$ for $i \neq l$ and we simplified notation to $d_{x_a} = d^0_{x_a}$.  This implies that 
\begin{align*}
f  d_{x_a} & =  d_{x_a}K d_{x_a} + K d_{x_a}^2 + d_{x_a} \\
 & = d_{x_a}^2 K + d_{x_a}K d_{x_a} + d_{x_a} = d_{x_a} f. 
\end{align*}
In the case that $f_{\sigma_0} = \mathit{id}$, we have 
\[
d_{x_N} = f_{\sigma_1} d_{x_a} f_{\sigma_1}^{-1} = f d_{x_a} f^{-1}= d_{x_a} = d_{x_b}
\]
and in the case that $f_{\sigma_0} = f$ we have
\[
d_{x_N} = f_{\sigma_1} d_{x_a} f_{\sigma_1}^{-1} = \mbox{id} \,d_{x_a}\, \mbox{id}^{-1} = d_{x_a}  = f d_{x_a} f^{-1} = f_{\sigma_0} d_{x_a} f_{\sigma_0}^{-1} = d_{x_b}
\]
as required.  Moreover, (\ref{eq:E7}) shows that $f_{\sigma_1}$ and $f_{\sigma_0}$ are chain homotopic.

\medskip

\noindent {\it When there are cusps in $[x_a,x_b]\times[0,1]$:}

\medskip

  In all cases except for (C2), the continuation maps are equal, $f_{\sigma_0} = f_{\sigma_1}$.  The verification for moves (E3), (R0), (R2),  and (C1) is routine.  Consulting the  bottom row of Figure \ref{fig:Handleslides} is helpful when verifying the cusp case of move (F3). 

For the pinch move (C2),  write
\[
C_{x_a} = C_{x_b} = C' \oplus C_{\mathit{cusp}}
\]
where  $C_{\mathit{cusp}}$ is spanned by $S_{k+1}$ and $S_k$ and $C'$ by the remaining strands.  Using the hypothesis that when the move is performed in the $\leftarrow$ direction the differential $d_{x_a} = d_{x_b}$ from $\mathcal{C}_0$ splits as in (\ref{eq:split}), the construction of the differentials is clear.  The continuation maps, $f_0$ and $f_1$, associated to the sides of the move where the cusps do and do not exist  respectively have the block matrix form
\[
f_0 = \left[\begin{array}{cc} \mbox{id}_{C'} & 0 \\ 0 & 0 \end{array} \right]  \quad f_1 = \left[\begin{array}{cc} \mbox{id}_{C'} & 0 \\ 0 & \mbox{id}_{C_\mathit{cusp}} \end{array} \right]
\]
so that setting $K(S_k) = S_{k+1}$ and $K(S_i) = 0$, $i\neq k$ we have the required homotopy
\[
f_1-f_0 = d_{x_b} K + K d_{x_a}.
\] 

Finally, we turn to the (R1) move.  See \cite[Proposition 6.2]{RuSu3} for the case of $\mathbb{F}_2$ coefficients.  We consider the  upward swallowtail, the downward case is similar.  First, suppose that the direction of the (R1) move is such that the cusps exist in $(\Lambda_1, \mathcal{C}_1)$ but not $(\Lambda_0, \mathcal{C}_0)$.  Let $S_k$ be the sheet of $(\Lambda_0, \mathcal{C}_0)$ that contains the spin base point $\xi$ that terminates at the swallowtail point.  In between the cusps of $\Lambda_1$ the sheets numbered $S_k, S_{k+1}, S_{k+2}$ form the three sheets involved with the swallowtail point.  As usual the sheets are numbered from top to bottom at each value of $x$, so that the way the bottom two sheets are numbered changes at the crossing.  Let $A_- = (a_{i,j}^-)$ and $A_+ = (a_{i,j}^+)$ denote the matrices of $d_{x_a}$ and $d_{x_b}$ respectively;  because of the spin base point these matrices are related by
\[
A_+ \Delta_k = \Delta_k A_-
\]
where $\Delta_i$ denotes the diagonal matrix with $-1$ in position $i$ and $1$'s at all other diagonal entries.  Let 
\begin{align*}
H_S &= (I-a^-_{1,k}E_{1,k}) \cdots (I-a^-_{k-1,k}E_{k-1,k})(I-E_{k+1,k+2}) = I -\sum_{i<k} a^-_{i,k}E_{i,k} - E_{k+1,k+2}, \\
H_T &= I+E_{k+1,k+2}
\end{align*}
denote the products of matrices associated to the handleslides of $(\Lambda_1,\mathcal{C}_1)$ appearing to the left and right side of the crossing respectively.  To establish that the differentials can be extended along $[x_a,x_b]\times\{1\}$ to agree with $d_{x_a}$ and $d_{x_b}$ from $\mathcal{C}_0$, it needs to be shown that 
\begin{equation} \label{eq:A+}
\widehat{A}_+ (H_TQH_S) = (H_TQH_S) \widehat{A}_-
\end{equation}
where $Q = Q_{(k+1\,k+2)}$ is the permutation matrix 
and $\widehat{A}_\pm$ is the block matrix obtained from inserting two new rows and columns at positions $k$ and $k+1$ with the $2\times 2$ block $\left[\begin{array}{cc} 0 & 1 \\ 0 & 0 \end{array} \right]$ placed on the diagonal.  Note that $\widehat{A}_+ = \Delta_{k+2} \widehat{A}_-\Delta_{k+2}$.  Computing 
\begin{equation} \label{eq:HQH}
(H_TQH_S) S_{j} = \left\{\begin{array}{cr} S_j, & j \notin \{k,k+1,k+2\}, \\ S_k- \sum_{i<k} a^{-}_{i,k} S_i, & j =k, \\
 S_{k+1}+S_{k+2}, & j = k+1, \\ -S_{k+2}, & j=k+2
\end{array} \right.  
\end{equation}
one checks (\ref{eq:A+}) by verifying the equality when both sides are applied to $S_j$.  In detail, 
when $j \notin \{k,k+1,k+2\}$,  since  $\widehat{A}_- S_{j} \in \mathit{Span}(\{S_1, \ldots, S_n\} \setminus\{S_k,S_{k+1}\})$ and $(H_TQH_S)$ agrees with $\Delta_{k+2}$ when restricted to this subspace, we compute 
\begin{align*}
(H_TQH_S) \widehat{A}_- S_{j} &= \Delta_{k+2} \widehat{A}_- S_{j}  \\
 &=  \Delta_{k+2} \widehat{A}_- (\Delta_{k+2} S_{j}) = \widehat{A}_+S_j = \widehat{A}_+ (H_TQH_S) S_{j}.
\end{align*}
When $j =k$, 
\begin{align*}
(H_TQH_S) \widehat{A}_- S_k&  = 0, \quad \mbox{and} \\
\widehat{A}_+ (H_TQH_S) S_k & = \widehat{A}_+(S_k- \sum_{i<k} a^-_{ik} S_i) = 0 - \sum_{l<i<k} a^-_{li}a^-_{ik} S_l = 0 
\end{align*}
since $A_-^2=0$.  When $j=k+1$,
\begin{align*}
(H_TQH_S) \widehat{A}_- S_{k+1} &= (H_TQH_S) S_{k} = S_k - \sum_{i<k} a^-_{i,k} S_i,  \quad \mbox{and} \\
\widehat{A}_+ (H_TQH_S) S_{k+1} &= \widehat{A}_+ (S_{k+1}+S_{k+2}) \\
 &= S_k + \Delta_{k+2}\widehat{A}_-\Delta_{k+2} S_{k+2} = S_k - \sum_{i<k} a^{-}_{i,k} S_i.
\end{align*}
Finally, when $j=k+2$,
\begin{align*}
(H_TQH_S) \widehat{A}_- S_{k+2} &=  \widehat{A}_- S_{k+2},  \quad \mbox{and} \\
\widehat{A}_+ (H_TQH_S) S_{k+2} &= \Delta_{k+2}\widehat{A}_-\Delta_{k+2}(-S_{k+2})  = \widehat{A}_- S_{k+2}.
\end{align*}
Moreover, (continuing to blur the distinction between a linear map and its matrix) the continuation maps are
\[
f_{\sigma_1} = p \circ (H_TQH_S) \circ \iota \quad \quad \mbox{and} \quad \quad f_{\sigma_0} = \Delta_{k}
\]
where $\iota$ and $p$ are respectively the inclusion and projection maps from (\ref{eq:lcuspmap}) and (\ref{eq:rcuspmap}).  Using (\ref{eq:HQH}) it is then easy to check that $f_{\sigma_1} = f_{\sigma_0}$.  

If the (R1) move occurs in the opposite direction so that the cusps exist above $\Lambda_0$ and not above $\Lambda_1$, then the same calculation as above computes $A_+$ in terms of $A_-$ as $A_+ = \Delta_k A_- \Delta_k$. Thus the differentials $d_{x_a}$ and $d_{x_b}$ satisfy (A4) at the spin basepoint $\xi$ of $\mathcal{C}_1$.  
\end{proof}

A pleasant consequence of Proposition \ref{prop:MCF} is that one can construct augmented Legendrian cobordisms by presenting a sequence of $1$-dimensional slices with handleslides and basepoints indicated and related as in the above moves, but {\it without needing to explicitly specify the families of differentials $\{d_x\}$} at every step.  One just needs to know that the initial slice $(\Lambda_0, \mathcal{C}_0)$ is equipped with differentials making $\mathcal{C}_0$ into a valid MCF.  Note that occasionally, while constructing such a movie, partial  information is needed about some of the differentials $\{d_x\}$ in order to perform one of the moves, eg. (E7), (R1), (C2), and (C3).  Typically, the relevant data can be reconstructed directly from the handleslides and basepoints at the current slice; provided that $d_x$ is known for at least one value of $x$ in the slice, the axioms (A1)-(A4) uniquely specify the rest of the differentials.

In fact, this method suffices for constructing arbitrary augmented Legendrian cobordisms.

\begin{proposition} \label{prop:cobord} Let $M = \R, [0,1],$ or $S^1$.

\begin{itemize}
\item[(i)] Two MCFs for a Legendrian $\Lambda \subset J^1M$ are equivalent if and only if they can be related by the equivalence moves (E1)-(E7).

\item[(ii)]  Two augmented Legendrians $(\Lambda_0, \mathcal{C}_0)$ and $(\Lambda_1, \mathcal{C}_1)$ are cobordant if and only if they can be related by a sequence of elementary moves (E1)-(E7), (F1)-(F4), (R0)-(R3), and (C1)-(C3).
\end{itemize}
\end{proposition}
\begin{proof}
After a mild wiggling, any augmented Legendrian cobordism can be cut up into the elementary cobordisms described in the above table.   [Applying a $C^\infty$-small isotopy of $\Sigma$ induced by an isotopy of the base space $M\times[0,1]$ rel. boundary arranges that   
base singularities of $(\Sigma, \mathcal{C})$ are positioned generically with respect to the $t$ coordinate.  In addition a $C^0$-small modification of the handleslide arcs near endpoints at super-handleslides arranges that the arcs approach there endpoints from the negative $t$ direction as in (E7).]  For an equivalence, the underlying Legendrian cobordism is the product cobordism with basepoints remaining fixed, and so only the elementary cobordisms corresponding to moves (E1)-(E7) appear.
\end{proof}

\begin{remark}  \label{rem:F1}
Note that using a combination of moves (F1) and (F2) (resp. (E1) and (E2)) a homology base point with coefficient $1$ (resp. a handleslide with coefficient $0$) can be added or removed by a cobordism (resp. equivalence).  
\end{remark}

In addition, Proposition \ref{prop:MCF} allows us to prove the statements about the continuation maps from Proposition \ref{prop:continuation}.

\begin{proof}[Proof of Proposition \ref{prop:continuation}]   Note that (2) is clear from the definition.  With Proposition \ref{prop:MCF} and \ref{prop:cobord} in hand, (1) follow as in the proof for the $\mathbb{F}_2$-coefficient case found in \cite[Proposition 4.7]{RuSu3}.  Let $(\Lambda, \mathcal{C}) \subset J^1M$ be an augmented Legendrian, and let  $\sigma, \tau: [0,1] \rightarrow M$ be path homotopic and transverse to the singular set of $(\Lambda, \mathcal{C})$, $M^\mathcal{C}_\mathit{sing}$.  Let $I: [0,1]\times[0,1] \rightarrow M$, $I(s,t) = \sigma_t(s)$ with $\sigma_0 =\sigma$, $\sigma_1 = \tau$, $\sigma_t(0) = \sigma(0)$, and $\sigma_t(1) = \sigma(1)$, be a path homotopy between $\sigma$ and $\tau$ that is also transverse to $M_\mathit{sing}^\mathcal{C}$, i.e. $I$ is transverse to the base singularities of all codimension.  We can then pull back $(\Lambda, \mathcal{C})$ along $I$ (in the expected manner) to get an augmented Legendrian cobordism $I^*(\Lambda, \mathcal{C}) \subset J^1([0,1]\times[0,1])$ such that $f_{\sigma}$ and $f_{\tau}$ are the continuation maps associated to the $[0,1]\times\{0\}$ and $[0,1]\times\{1\}$ slices of the cobordism.  From Proposition \ref{prop:cobord} and Proposition \ref{prop:MCF} (3) we see that $f_{\sigma}$ and $f_{\tau}$ are chain homotopic.
\end{proof}

\subsubsection{Uniqueness of the (R1) handleslides}  \label{sec:STunique} We make an observation that simplifies considerations involving the (R1) move.  Clearly, the (R1) move can always be performed in the $\rightarrow$ direction:  If the spin basepoint is not already present at the desired location of the (R1) move, then a pair of spin basepoints can be created using (F1) so that the (R1) move can be performed (leaving one additional spin point not involved in the move).  At first glance, it may appear that there is a restriction on whether the $\leftarrow$ (R1) move can be performed, since the handleslides that appear in the interval $[x_a,x_b]$ where the move occurs must match the right column of Figure \ref{fig:Endpoints3} with coefficients determined by $d_{x_a}$.  In fact, the correctness of the coefficients is automatic once all other handleslides and basepoints have been removed from $[x_a,x_b]$.   

\begin{lemma} \label{lem:R1} Suppose that we have a 1-dimensional MCF $(\Lambda_0,\mathcal{C}_0)$ where in $[x_a, x_b] \subset M$ the Legendrian $\Lambda_0$ appears as in the top row and right side of (R1) in Figure \ref{fig:Endpoints3}, and let $x=x_l$ and $x=x_r$ be the location of the left and right cusps.  Suppose further that   
\begin{itemize}
\item there are no basepoints in $[x_a,x_b]$, and
\item all handleslides in $[x_a,x_b]$ are located in $(x_l,x_r)$ and are either $(k+1,k+2)$-handleslides (these can appear on both sides of the crossing) or $(i,k)$ handleslides for $i<k$, where the sheets that meet the cusp points are labeled $k,k+1$ in $(x_l,x_r)$.  
\end{itemize}
Then, after commutation (E4) moves and grouping handleslides with common endpoints together using (E2) moves the handleslides can be made to appear with coefficients as in Figure \ref{fig:Endpoints3}.
\end{lemma}

\begin{proof}
Clearly the handleslides can be rearranged with (E4) and (E2) moves to have their locations as in Figure \ref{fig:Endpoints3}.  For $i<k$, let $a_{i,k} = \langle d_{x_a} S_k, S_i \rangle$  
and let $b_{i,k}$ be the coefficient of the $(i,k)$-handleslide in $(x_r,x_l)$, where $b_{i,k}=0$ means there is no handleslide.  Let $b_-$ and $b_+$ denote the coefficients of the $(k+1,k+2)$-handleslides that appear to the left and right of the crossings.  Using (A1)-(A3) to determine from $d_{x_a}$ the differential $d_{x_r^{-}}$ with $x_r^{-}<x_r$ just preceding the right cusp, we compute (consulting Figure \ref{fig:Handleslides} is helpful here)
\[
\langle d_{x^-_r} S_{k+1}, S_k \rangle = -b_-,  \quad \langle d_{x^-_r} S_{k+2}, S_k\rangle = 1+b_-b_+, \quad \langle d_{x^-_r} S_{k+1}, S_{i} \rangle  =  a_{i,k} -b_{i,k}b_{-}. 
\]  
Thus, (A4) implies that $-b_-=1$, $1+b_-b_+ =0$, and  $a_{i,k} - b_{i,k}b_-=0$, so that the coefficients indeed must have the form $b_{-} =-1$, $b_+=1$, and $b_{i,k} = -a_{i,k}$.  
\end{proof}

In particular, when constructing an augmented Legendrian cobordism via a movie of slices, if one wants to perform the $\leftarrow$ (R1) move, it is enough to first arrange that the handleslides near the move are as in Lemma \ref{lem:R1}.  This can always be done; see  Proposition \ref{prop:Rmove} below.

\begin{remark}
A similar statement holds for a downward swallowtail point as in the bottom row of Figure \ref{fig:Endpoints3}.  We leave the precise formulation to the reader.
\end{remark}

\subsection{The $D_4^-$ cobordism} \label{sec:D4minus} As an example of the above method, we combine several elementary moves to establish a (non-elementary) augmented cobordism move, useful later in Section \ref{sec:J101}, that corresponds to passing a desingularized $D_4^{-}$ point on an augmented Legendrian surface.  The (non-generic) $D_4^-$ singularity has proven to be useful for the efficient construction of Lagrangian fillings and cobordisms, cf. \cite{TZ, CZ}.  See eg. loc. cit. or \cite{ArnoldGZV} for a description of the $D_4^-$ singularity and its desingularization.

\begin{proposition} \label{prop:D4minus}
For any $r \in R^*$, there is an augmented Legendrian cobordism

\quad

\[
\labellist
\small
\pinlabel $r$ [l] at 70 28
\pinlabel $r$ [b] at 240 60
\pinlabel $1$ [r] at 252 28
\pinlabel $-1$ [l] at 318 28
\pinlabel $-1$ [br] at 330 60
\pinlabel $r^{-1}$ [bl] at 342 60
\endlabellist
\includegraphics[scale=.8]{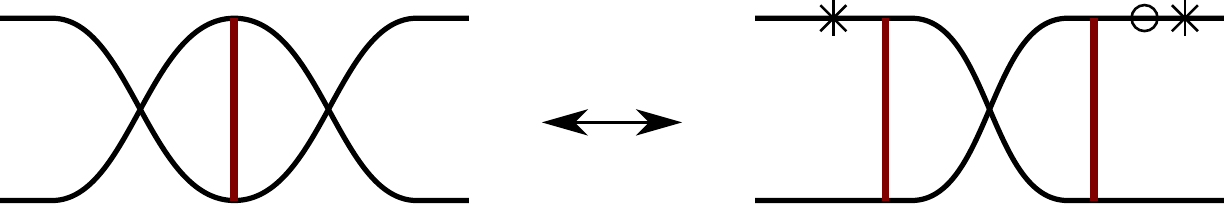}
\]
In the pictured region, the cobordism surface $\Sigma$ has a standard desingularization of the $D_4^-$ singularity and a single homology arc with  endpoints as pictured; outside of the pictured region the cobordism is a product.
\end{proposition}

\begin{proof}
Construct the cobordism via a sequence of elementary moves as indicated by:

\[
\labellist
\tiny
\pinlabel $r$ [l] at 68 362
\pinlabel $r$ [t] at 82 326
\pinlabel $r^{-1}$ [b] at 140 392
\pinlabel $-1$ [b] at 82 392
\pinlabel $-1$ [t] at 140 326
\pinlabel $r$ [b] at 228  392
\pinlabel $r^{-1}$ [b] at 344  392
\pinlabel $1$ [l] at 272 348
\pinlabel $-1$ [r] at 274 372
\pinlabel $1$ [l] at 294  372
\pinlabel $-1$ [b] at 344 338
\pinlabel $-1$ [l] at 442 396
\pinlabel $1$ [r] at 536  396
\pinlabel $-1$ [r] at 474 362
\pinlabel $1$ [l] at 504  362

\pinlabel $-1$ [t] at 114 254
\pinlabel $-1$ [r] at 154 236
\pinlabel $1$ [l] at 204 236
\pinlabel $1$ [t] at 240 254

\pinlabel $-1$ [t] at 366 262
\pinlabel $-1$ [t] at 430 226
\pinlabel $1$ [t] at 480 226
\pinlabel $1$ [t] at 542 262
\pinlabel $-1$ [b] at 510 268
\pinlabel $1$ [b] at 496 272

\pinlabel $1$ [t] at 68 122
\pinlabel $1$ [t] at 90 122
\pinlabel $-1$ [t] at 146 122
\pinlabel $1$ [t] at 196 122
\pinlabel $1$ [t] at 212 122
\pinlabel $-1$ [t] at 80 158
\pinlabel $-1$ [t] at 224 148
\pinlabel $1$ [t] at 258 158

\pinlabel $1$ [t] at 352 122
\pinlabel $-1$ [t] at 366 158
\pinlabel $1$ [t] at 406 148
\pinlabel $1$ [t] at 482 122
\pinlabel $-1$ [t] at 508 148
\pinlabel $1$ [t] at 542 158

\pinlabel $1$ [t] at 14 0
\pinlabel $-1$ [t] at 28 40
\pinlabel $1$ [t] at 68 28
\pinlabel $1$ [t] at 142 0
\pinlabel $-1$ [t] at 172 26
\pinlabel $1$ [t] at 206 38

\pinlabel $1$ [t] at 300 8
\pinlabel $1$ [t] at 359 8

\pinlabel $1$ [t] at 484 8
\pinlabel $-1$ [t] at 544 8
\pinlabel $r$ [b] at 474 56
\pinlabel $-1$ [b] at 560 56
\pinlabel $r^{-1}$ [b] at 588 56
\pinlabel (C2) [b] at 590 152
\pinlabel $1$ [l] at 440 184

\endlabellist
\includegraphics[scale=.6]{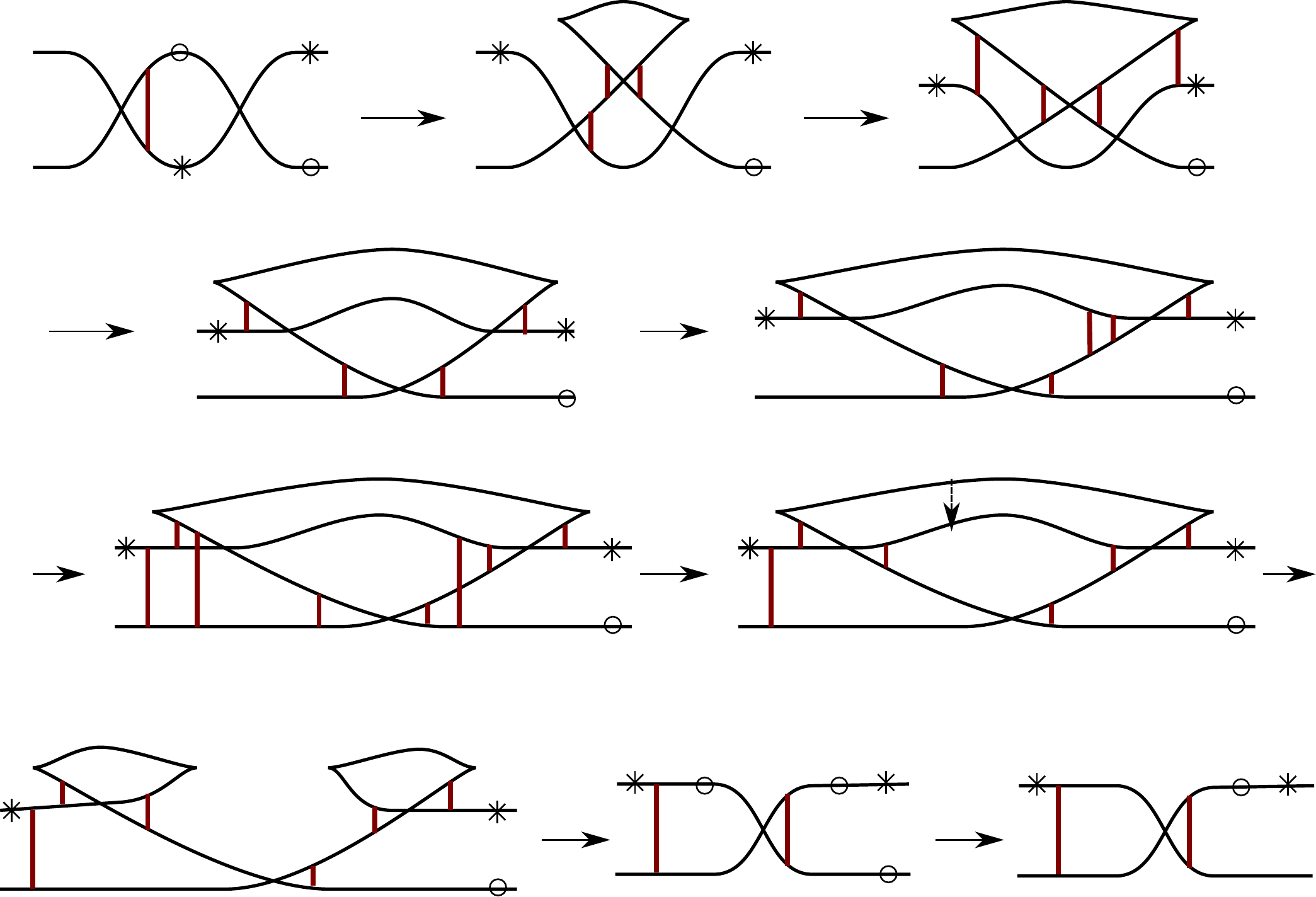}
\]

\quad

Let $A_- = (a^-_{i,j})$ denote the matrix of $d_x$ at the far left side of the pictured interval, and let $S_k$ and $S_{k+1}$ be the numbering of the two pictured sheets at this location.  The construction starts by using (F1) to create a new pair of homology base points with coefficients $r$ and $r^{-1}$, and a new pair of spin basepoints (with coefficients $-1$ and notated with $\circ$).  After the first (R1) move for each $i<k$ there is an additional $(i,k)$-handleslide (not pictured) connecting $S_i$ to the upper sheet of the cusps with coefficients $-a^{-}_{i,k+1}$.  Immediately after their appearance, we convert each of these via a tri-valent vertex into two $(i,k)$-handleslides with coefficients $-a^{-}_{i,k} r^{-1}$ and $-a^-_{i,k+1} + a_{i,k}^- r^{-1}$ positioned near the left and right cusp respectively.  (These handleslides with different values of $i$ all commute with one another, so there is no issue in the repositioning.)  At the location, $x=x_c$, where sheets $S_k$ and $S_{k+1}$ are connected by the dotted arrow just before the $\leftarrow (C2)$ move  occurs,  the sheets $S_k$ and $S_{k+1}$ form a subcomplex with $d_{x_c}(S_{k+1})= S_k$ that splits off from the rest of $(C_{x_c}, d_{x_c})$ as a direct sum.  [This follows from the identity (\ref{eq:A+}) established in the proof of Proposition \ref{prop:MCF}; in notation of (\ref{eq:A+}) $\widehat{A}_-$ is now the matrix of $d_x$ immediately to the right of the left cusp, and $\widehat{A}_+$ is the matrix of $d_{x_c}$ at the dotted arrow.  Note also that the handleslides and crossing that appear between $x_l$ and $x_c$ are consistent with the product $H_SQH_T$ appearing in (\ref{eq:A+}) with the coefficients $-a^{-}_{i,k} r^{-1}$ of the $(i,k)$-handleslides equal to the negative of the coefficient $\langle d_{x'} S_k, S_i\rangle$ with $x'$ just before the left cusp.]  Thus, the required hypothesis is satisfied so that the $\leftarrow$ direction of the pinch move $(C2)$ can be performed as pictured.  Moreover, the subsequent (R1) moves can then be carried out using Lemma \ref{lem:R1}.  
\end{proof}

\subsection{Properly ordered handleslide collections} \label{sec:properly} It will be convenient to have a standard positioning (up to equivalence) for a collection of handleslides in a region where a $1$ dimensional Legendrian has no front singularities.  

\begin{definition}  \label{def:Properly}
Let $(\Lambda, \mathcal{C}) \subset J^1M$ be a $1$-dimensional augmented Legendrian, and let $[x_a,x_b] \subset M$ be an interval  where the front projection of $\Lambda$ consists of $n$ parallel strands without crossings or cusps.  We say the handleslides of $\mathcal{C}$ appearing in the interval $[x_a, x_b]$ are {\bf properly ordered} if 
\begin{itemize}
\item[(i)] for each $1\leq i < j \leq n$ there is at most one $(i,j)$-handleslide in $[x_a, x_b]$, and
\item[(ii)] the handleslides appear in lexicographical order with respect to their indices, i.e. whenever an $(i_1, j_1)$-handleslide appears before (to the left of) an $(i_2,j_2)$-handleslide we have either $i_1< i_2$ or $i_1=i_2$ and $j_1<j_2$.
\end{itemize}
See Figure \ref{fig:Proper}.
\end{definition}

\begin{figure}

\quad

\quad

\labellist
\small
\pinlabel $b_{1,2}$ [b] at 22 116
\pinlabel $b_{1,3}$ [b] at 45 116
\pinlabel $b_{1,4}$ [b] at 68 116
\pinlabel $b_{2,3}$ [b] at 98 78
\pinlabel $b_{2,4}$ [b] at 121 78
\pinlabel $b_{3,4}$ [b] at 150 40

\endlabellist

\includegraphics[scale=.8]{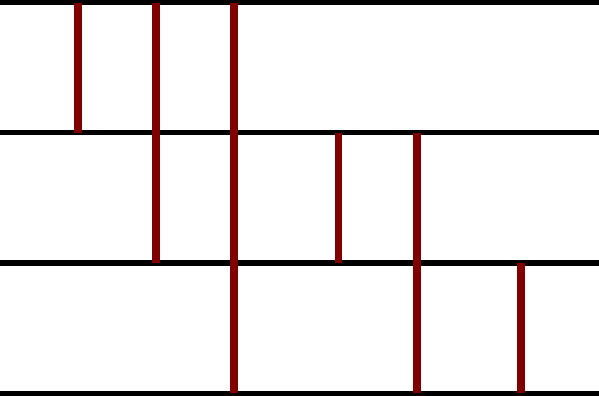}
\caption{A properly ordered handleslide collection. The continuation map has matrix $I+\sum_{i<j} b_{i,j} E_{i,j}$. 
}
\label{fig:Proper}
\end{figure}
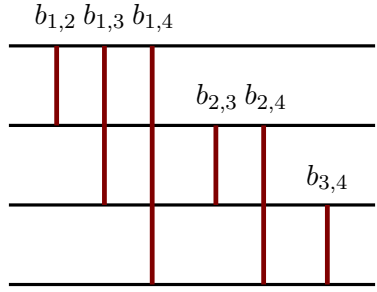

If they are not properly ordered, then we can always arrange the handleslides in $[x_a,x_b]$ to become properly ordered by an equivalence that leaves $\mathcal{C}$ unchanged outside of $[x_a,x_b]$.  

\begin{proposition} \label{prop:proper}
Let $(\Lambda, \mathcal{C}) \subset J^1[x_a,x_b]$ be an augmented Legendrian without crossings or cusps in the interval $[x_a,x_b] \subset M$.
\begin{enumerate}
\item There is an equivalence of MCFs, $\mathcal{C} \sim \mathcal{C}'$, such that the handleslides of $\mathcal{C}'$ are properly ordered.
\item If $\mathcal{C}$ has no basepoints and the handleslides are properly ordered, then the continuation map $f_{[x_a,x_b]}$ has matrix $I + \sum_{i<j} b_{i,j}E_{i,j}$ where $b_{i,j}$ is the coefficient of the $(i,j)$-handleslide of $\mathcal{C}$ and is $0$ if no such handleslide exists. 
\end{enumerate}
\end{proposition}
\begin{proof}
The proof of (1) is a straightforward induction on the number of handleslides of $\mathcal{C}$, cf. \cite[Section 6e1]{Henry}.  The proof of (2) is a calculation: The matrix of $f_{[x_a,x_b]}$ is the product of handleslide matrices $I+b_{i,j}E_{i,j}$ in the order that the handleslides appear from {\it right to left}, and all of the $(b_{i_2,j_2}E_{i_2,j_2})(b_{i_1,j_1}E_{i_1,j_1})$ terms vanish since $i_1< j_2 $.
\end{proof}

\subsection{$SR$-from MCFs} \label{sec:SRform}

Brad Henry introduced in \cite{Henry} a standard form for MCFs of Legendrian knots in $\Lambda \subset J^1\R$ called the $SR$-form.  MCFs in $SR$-form have their handleslides in standard positions near the switches and returns of a normal ruling of $\Lambda$, and every MCF is equivalent to one in $SR$-form.  Here, we extend this definition and result to MCFs in $J^1[0,1]$ or $J^1S^1$.

\subsubsection{Generalized normal rulings} Let $M = [0,1]$ or $S^1$.  A {\bf $\rho$-graded generalized normal ruling}, $\sigma$, of a $\rho$-graded Legendrian link $\Lambda$ is a family of involutions, 
\[
\sigma_{x_0}:\Lambda \cap \pi^{-1}_x(\{x_0\}) \rightarrow \Lambda \cap \pi^{-1}_x(\{x_0\}), \quad  \sigma_{x_0}^2 = \mathit{id}, 
\]
defined for all $x_0 \in M$ not equal to the base projection of any crossing or cusp of $\pi_{xz}(\Lambda)$, and subject to several requirements.  We refer the reader to \cite{LaRu1} for the thorough definition, and give a briefer summary here.  For each $x_0$, the involution $\sigma_{x_0}$ partitions the sheets of $\Lambda$ at $x_0$, $S_1, \ldots, S_n \in C_{x_0}$, into some number of {\it fixed point sheets} and some sets of {\it paired sheets}.  The fixed point sheets satisfy $\sigma_{x_0}(S_i) = S_i$ and the paired sheets form two element subsets $\{S_i,S_j\}$ with $\sigma_{x_0}(S_i) = S_j$.  The $\rho$-graded condition is: if $S_i$ and $S_j$ are paired and $z(S_i) > z(S_j)$, then the Maslov potential satisfies $\mu(S_i) = \mu(S_j) +1$.  
At cusp points, the two sheets that meet at the cusp must be paired by $\rho$ for $x_0$ near the cusp point. 
At crossings of $\pi_{xz}(\Lambda)$, the two sheets that meet at the crossing cannot be paired for $x_0$ near the crossing.  Crossings where the pairing of sheets changes are called {\bf switches}.  At switches the vertical ordering ($z$-coordinate) of the two crossing strands and their pairs (if they exist) must satisfy the normality condition pictured in Figure \ref{fig:NormC}.  The other crossings of $\pi_{xz}(\Lambda)$ are called {\bf departures} and {\bf returns} depending on whether the normality condition is satisfied to the left or to the right of the crossing.  By convention, crossings between fixed point sheets are considered to be returns, but not departures or switches.  See Figure \ref{fig:RulinEx} for an example.

\begin{figure}[!ht]

\labellist
\small

\endlabellist

\centerline{\includegraphics[scale=.8]{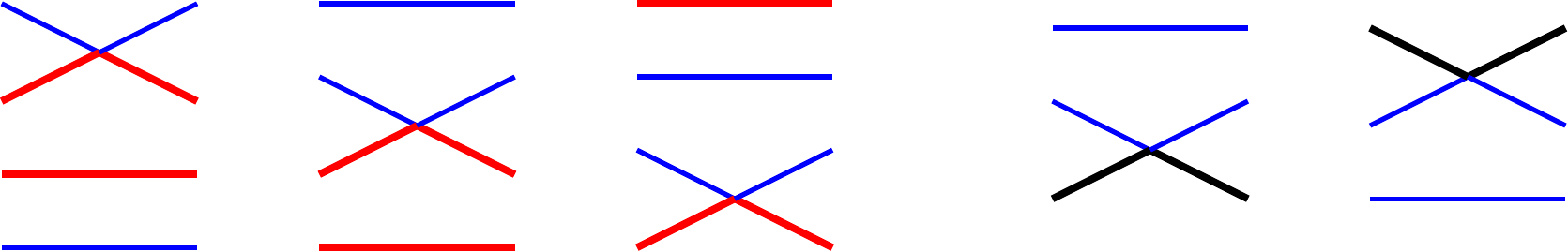}}

\caption{The normality condition at switches.  Sheets of the same color are paired.  The strands pictured in black are fixed point sheets.  Crossings between fixed point sheets are not considered to be switches.  (There may be any number of other strands above, below and in between the pictured ones.)}
\label{fig:NormC}
\end{figure}

\begin{figure}[!ht]

\labellist
\small

\endlabellist

\centerline{\includegraphics[scale=.6]{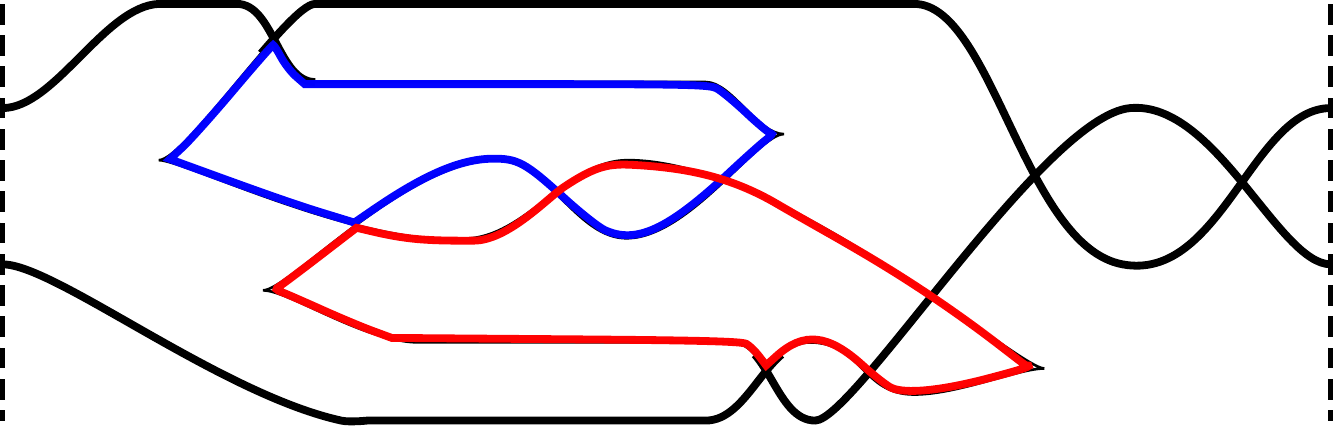}}

\caption{A generalized normal ruling of a Legendrian link in $J^1S^1$.  The black strands are fixed point sheets.  From left to right the crossings are two switches, a departure, a return, a switch, a departure, and three returns.}
\label{fig:RulinEx}
\end{figure}

\begin{remark}  The terminology {\it generalized normal ruling} is from \cite{LaRu1} where its purpose is to highlight the difference from ordinary normal rulings where the involutions $\sigma_{x_0}$ are required to be fixed point free, eg. in \cite{Fuchs, ChP}.  In this article, we will not have much need for the fixed point free requirement, and will sometimes use the shortened terminology {\it normal ruling} to refer to any generalized normal ruling whether or not the $\sigma_{x_0}$ have fixed points.
\end{remark}

\subsubsection{Barannikov normal form}  

Let $C$ be a $\Z/\rho$-graded $\mathbb{F}$-vector space  with basis $\{e_1, \ldots, e_n\}$ and let $\sigma:\{1, \ldots, n\}\rightarrow \{1, \ldots, n\}$ be an involution satisfying $|e_i| = |e_j|+1$ if $\sigma(j) = i$ and $i<j$.  A differential $d_\sigma:C^* \rightarrow C^{*+1}$ is said to be in {\bf Barannikov normal form with respect to $\sigma$} if  
\[
d(e_j) = \left\{\begin{array}{cr} \alpha_{i,j}e_i, & \mbox{if $i= \sigma(j) < j$}, \\ 0, & \mbox{else}, \end{array} \right.
\]
with the $\alpha_{i,j} \neq 0$.

\begin{proposition} \label{prop:Barannikov}
For any upper triangular differential 
\[
d:C^* \rightarrow C^{*+1}, \quad d^2 =0, \quad d(e_j) \in \mbox{Span}_{\mathbb{F}}\{e_1, \ldots, e_{j-1}\} \, \mbox{for all $j$},
\]
there exists a unique involution $\sigma: \{1, \ldots, n\} \rightarrow \{1, \ldots, n\}$ and a (non-unique) grading preserving upper triangular change of basis of the form
\[
\Phi(e_j) = e_j + \sum_{i<j} \beta_{i,j} e_j
\]
such that $\Phi \circ d \circ \Phi^{-1}$ is in Barannikov normal form with respect to $\sigma$.
\end{proposition}
\begin{proof}
The existence of $\Phi$ and $\sigma$ can be proven using induction on $n$.  The involution $\sigma$ is uniquely determined by the homology of the subquotients $H^*(C_{\leq j}/C_{\leq i})$ where $C_{\leq j}$ is the sub-complex spanned by $e_1, \ldots, e_j$.   See \cite[Lemma 2]{Bar} for a detailed proof.
\end{proof}
 
\begin{remark}  
In \cite{Bar}, Barannikov also analyzed the way that the involution $\sigma$ can change when the order of two generators is interchanged.  This is the key ingredient in a construction of normal rulings from generating families due to Chekanov and Pushkar; see  eg. \cite{ChP, FuRu} for more discussion.  Henry generalized this in \cite{Henry, ChP} to the following map from equivalence classes of MCFs to normal rulings.
\end{remark}

Given an MCF $\mathcal{C}$ for a $1$-dimensional Legendrian $\Lambda \subset J^1M$ define a family of involutions $\sigma(\mathcal{C})_{x_0} : \Lambda \cap \pi_{x}^{-1}(\{x_0\}) \rightarrow \Lambda \cap \pi_{x}^{-1}(\{x_0\})$ for $x_0 \in M_{\mathit{reg}}^\mathcal{C}$ by taking $\sigma(\mathcal{C})_{x_0}$ to be the involution of the Barannikov normal form of $d_{x_0}:C_{x_0} \rightarrow C_{x_0}$.  

\begin{proposition} \label{prop:nruling}
The construction 
\[
\mathcal{C} \mapsto \sigma(\mathcal{C})
\] 
defines a mapping from the set of $\rho$-graded MCFs to the set of $\rho$-graded generalized normal rulings of $\Lambda$. Moreover, the mapping is well defined on equivalence classes of MCFs.
\end{proposition}
\begin{proof}
See \cite[Lemma 3.14 and Proposition 3.15]{Henry} for Legendrians in $J^1\R$ with $\mathbb{F}_2$ coefficients.  See also \cite[Section 5.B]{LaRu1} for a related statement for Legendrians in $J^1S^1$ again over $\mathbb{F}_2$.  Barannikov's result \cite[Lemma 4]{Bar}  establishes the normality condition at switches.  Since this statement is valid over any {\it field}, the proof goes through over  $\mathbb{F}$.
\end{proof}

\subsubsection{$SR$-form MCFs with respect to a generalized normal ruling}
Let $\Lambda \subset J^1M$ with $M = [0,1]$ or $S^1 = [0,1]/\{0,1\}$.  For each crossing, $c$, choose a small open interval $I_c \subset M$ containing the crossing and no other singularities of $\pi_{xz}(\Lambda)$.  In the case $M =[0,1]$ choose small half open intervals $I_0$ and $I_1$ containing $x = 0$ and $1$ and no crossings or cusps of $\Lambda$; when $M =S^1$ choose such an open interval $I_0$ containing $x = 0$.  The intervals should be chosen small enough as to all be disjoint from one another.  

\begin{definition} \label{def:SRform} Let $\sigma$ be a ($\rho$-graded) generalized normal ruling for $\Lambda$.  An MCF $\mathcal{C}$ for $\Lambda$ is in {\bf $SR$-form with respect to $\sigma$}, if for some choice of intervals as above the following conditions are satisfied:
\begin{enumerate}
\item For all $x$ not belonging to any of the intervals $I_0, I_1$, or $I_c$, the differential $d_{x}:C_{x} \rightarrow C_{x}$ is in Barannikov normal form with respect to $\sigma_x$.  
\item The only handleslides of $\mathcal{C}$ are arranged as follows.
\begin{itemize}
\item[(i)]  In the intervals $I_c$ when $c$ is a switch or return collections of handleslides appear as in Figure \ref{fig:SR}.
\item[(ii)]  If $M= [0,1]$ properly ordered handleslide collections appear in $I_0$ and $I_1$.
\item[(iii)] If $M=S^1$ a properly ordered handleslide collection appears in $I_0$.
\item[(iv)] If $\rho=1$, i.e. if $(\Lambda, \mathcal{C})$ is $1$-graded, then handleslides may appear at right cusps as in Figure \ref{fig:SR}.
\end{itemize}
\end{enumerate}
\end{definition}

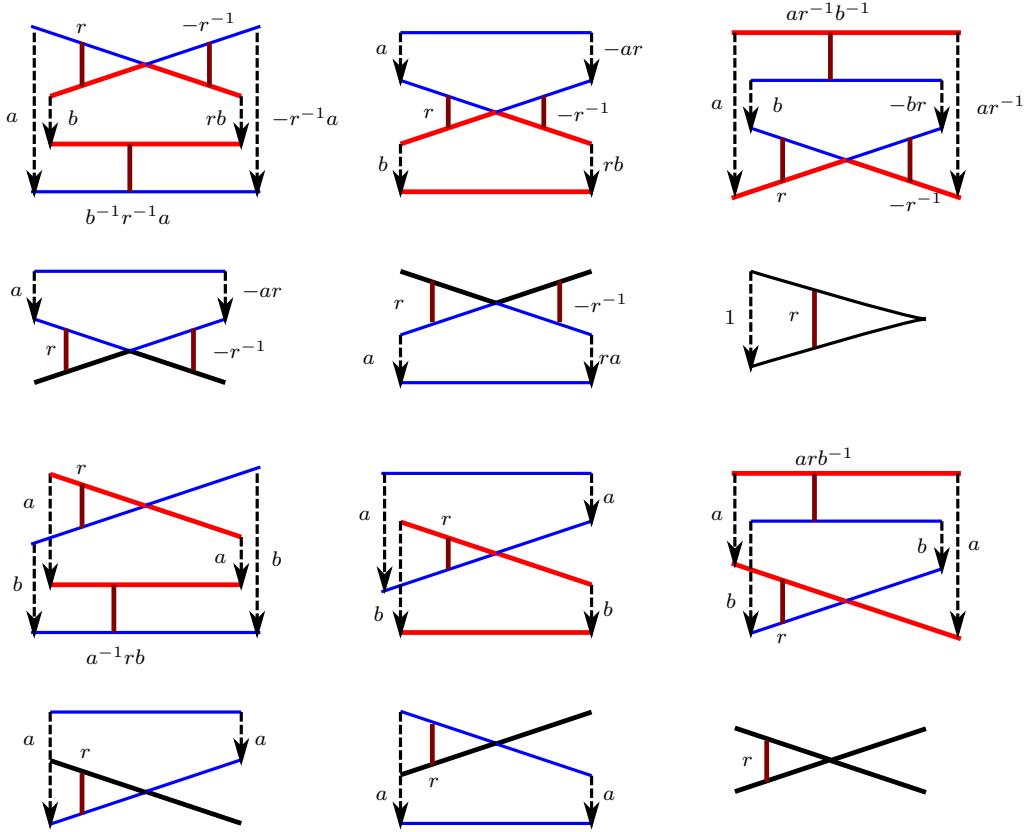
\begin{figure}[!ht]

\labellist
\tiny
\pinlabel $a$ [r] at 0 128
\pinlabel $b$ [l] at 16 128
\pinlabel $r$ [b] at 26 164
\pinlabel $-r^{-1}$ [b] at 86 164
\pinlabel $b^{-1}r^{-1}a$ [t] at 48 90
\pinlabel $rb$ [r] at 98 128
\pinlabel $-r^{-1}a$ [l] at 112 128

\pinlabel $a$ [r] at 174 160
\pinlabel $b$ [r] at 174 106
\pinlabel $-ar$ [l] at 268 160 
\pinlabel $rb$ [l] at 268 106
\pinlabel $r$ [r] at 196 130
\pinlabel $-r^{-1}$ [l] at 246 130

\pinlabel $a$ [r] at 332 134
\pinlabel $b$ [l] at 348 134
\pinlabel $-br$ [r] at 428 134
\pinlabel $ar^{-1}$ [l] at 444 134
\pinlabel $ar^{-1}b^{-1}$ [b] at 378 170
\pinlabel $r$ [t] at 356 96
\pinlabel $-r^{-1}$ [t] at 419 96

\pinlabel $a$ [r] at 2 46
\pinlabel $-ar$ [l] at 98 46
\pinlabel $r$ [r] at 18 18
\pinlabel $-r^{-1}$ [l] at 84 18

\pinlabel $a$ [r] at 168 14
\pinlabel $ra$ [l] at 266 14
\pinlabel $r$ [r] at 182 40
\pinlabel $-r^{-1}$ [l] at 254 40

\pinlabel $1$ [r] at 338 34
\pinlabel $r$ [r] at 368 34

\endlabellist

\centerline{\includegraphics[scale=.8]{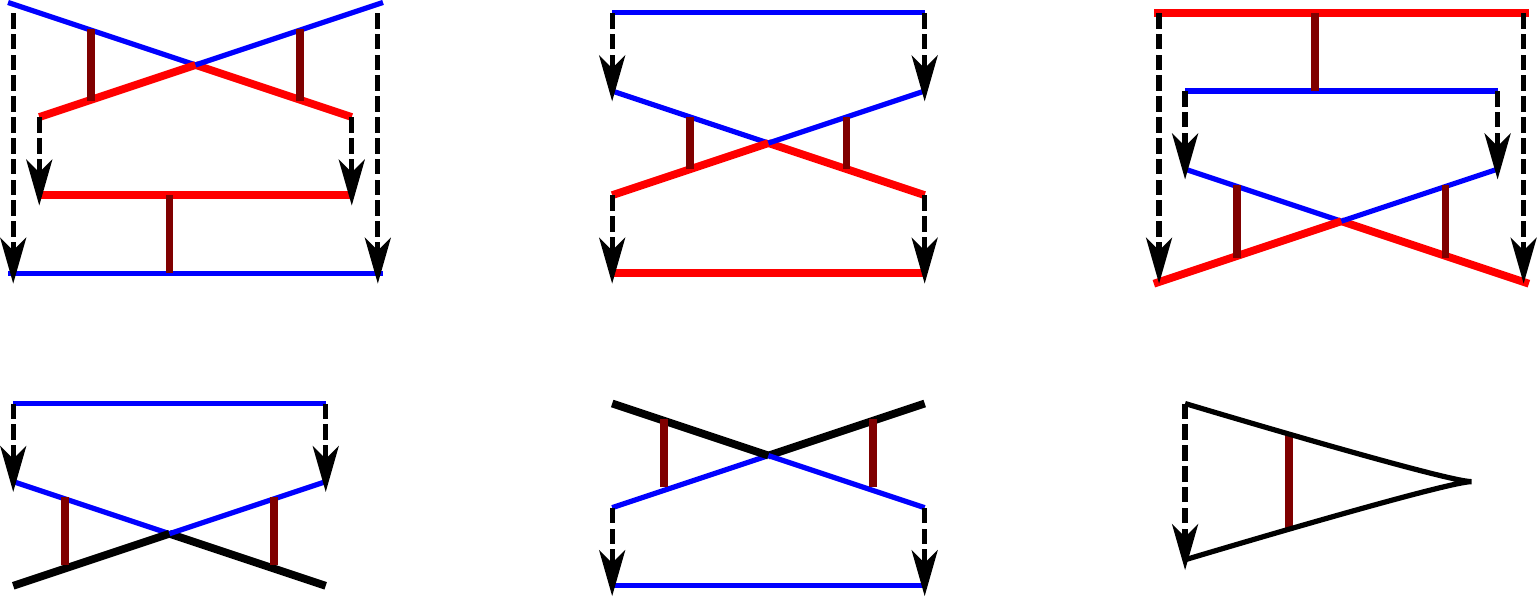}}

\quad

\quad

\labellist
\tiny
\pinlabel $a$ [r] at 8 154
\pinlabel $b$ [r] at 2 116
\pinlabel $b$ [l] at 112 128
\pinlabel $a$ [r] at 98 128
\pinlabel $a^{-1}rb$ [t] at 42 90
\pinlabel $r$ [b] at 26 164

\pinlabel $a$ [r] at 166 148
\pinlabel $b$ [r] at 172 102
\pinlabel $r$ [b] at 198 140
\pinlabel $a$ [l] at 268 156
\pinlabel $b$ [l] at 268 104

\pinlabel $a$ [r] at 332 148
\pinlabel $b$ [r] at 338 108
\pinlabel $r$ [t] at 356 96
\pinlabel $arb^{-1}$ [b] at 376 168
\pinlabel $b$ [r] at 428 134
\pinlabel $a$ [l] at 440 134

\pinlabel $a$ [r] at 8 42
\pinlabel $r$ [b] at 28 30
\pinlabel $a$ [l] at 104 42

\pinlabel $a$ [r] at 174 16
\pinlabel $r$ [t] at 192 28
\pinlabel $a$ [l] at 268 16

\pinlabel $r$ [r] at 346 32

\endlabellist

\centerline{\includegraphics[scale=.8]{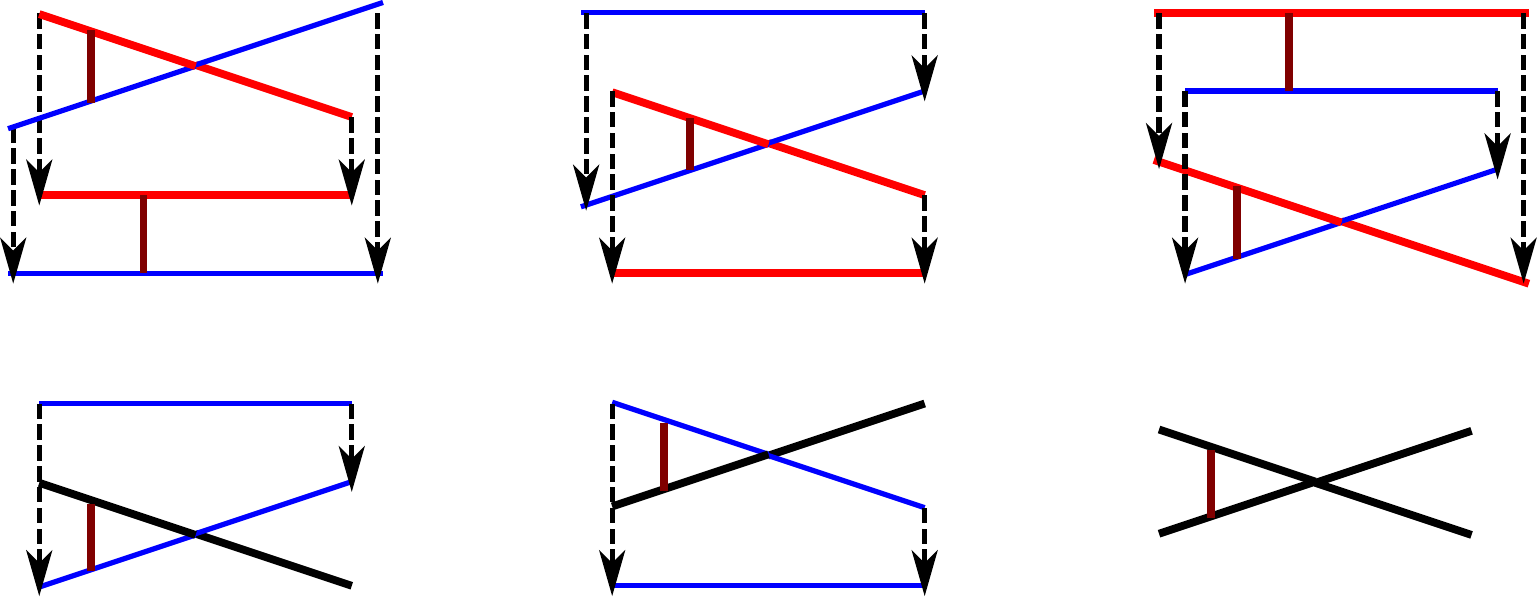}}

\caption{Handleslides near switches, right cusps, and returns for an $SR$-form MCF.}
\label{fig:SR}
\end{figure}

As an example, the MCF from Figure \ref{fig:MCFEx2} is in $SR$-form with respect to the normal ruling from Figure \ref{fig:RulinEx}.

\begin{proposition}  \label{prop:SRalgorithm}
Any $\rho$-graded MCF, $\mathcal{C}$, for a $1$-dimensional Legendrian $\Lambda \subset J^1M$ is equivalent to an $SR$-form MCF, $\mathcal{C}'$, for a $\rho$-graded generalized normal ruling, $\sigma$,
of $\Lambda$.  Moreover, if $M=[0,1]$ and $d_0:C_0 \rightarrow C_0$ is already in Barranikov normal form, then $\mathcal{C}'$ can be constructed without any handleslide collection in the interval $I_0$ containing $x=0$. 
\end{proposition}

\begin{proof}  For Legendrians in $J^1\R$ with $\mathbb{F}_2$ coefficients this is proven in \cite[Section 6]{Henry}.  In our proof, we pay particular attention to the changes required for Legendrians in $J^1[0,1]$ or $J^1S^1$ with $\mathbb{F}$ coefficients.  

Let $(\Lambda, \mathcal{C}) \subset J^1M$ with $M = [0,1]$ or $S^1$.  Beginning at $x=0$, we apply the Proposition \ref{prop:Barannikov} to find an upper-triangular isomorphism $\Phi:C_{0} \rightarrow C_{0}$ and a Barannikov normal form differential $d_{\sigma_0}$ with $\Phi d_{0} = d_\sigma \Phi$ where $\sigma_{0}:\{S_1, \ldots, S_{n_0}\} \rightarrow \{S_1, \ldots, S_{n_0}\}$ is an involution of the sheets of $\Lambda$ at $x=0$.  Using Proposition \ref{prop:proper}, there exists a properly ordered handleslide collection $\mathcal{H}_0$ whose continuation map is $\Phi$.  Now, applying a sequence of $\leftarrow$ (E1) moves, we can modify $\mathcal{C}$ in the interval $I_0$ to begin with two handleslide collections $\mathcal{H}_0$ followed by $\mathcal{H}_1$ where $\mathcal{H}_1$ is the reflection of $\mathcal{H}_0$ across a vertical line with coefficients negated.  We leave $\mathcal{H}_0$ in the interval $I_0$ and then follow Henry's procedure of ``sweeping'' $\mathcal{H}_1$ past the singularities of $(\Lambda, \mathcal{C})$ from left to right, identifying the generalized normal ruling $\sigma$  as we go.  Observe, if $d_0$ was already in Barranikov normal form, then we can take $\Phi= \mathit{id}$ so that $\mathcal{H}_0 =\emptyset$ as claimed. 

Let $0= x_0 < x_1 < x_2 < \ldots < x_N=1$ be so that $x_i$ with $1 \leq i \leq N-1$ are the $x$-coordinates of the singularities of $(\Lambda, \mathcal{C})$.  Suppose inductively that we have constructed a normal ruling $\sigma$ of $\Lambda$ on $[x_0, x_i)$ and an equivalence from $\mathcal{C}$ to $\mathcal{C}_i$ such that $\mathcal{C}_i$ satisfies the conditions of Definition \ref{def:SRform} from $x_0$ up to the location of a properly ordered collection of handleslides $\mathcal{H}_i$ that is located in a small interval of the form $(x_i-\epsilon, x_i)$.  To extend this situation up to $x_{i+1}$ we consider cases.  

\medskip

\noindent {\bf Case 1.}  $x_{i}$ is a {\it handleslide}.  Then, we simply group the new handleslide into $\mathcal{H}_i$, apply Proposition \ref{prop:proper} (1) to produce a new properly ordered collection $\mathcal{H}_{i+1}$, and then slide $\mathcal{H}_{i+1}$ over to $x_{i+1}$.

\medskip 

\noindent {\bf Case 2.}  $x_i$ is a {\it crossing} $c$ between sheets $S_{k}$ and $S_{k+1}$.  We sweep the handleslides in $\mathcal{H}_i$ past $x_i$ using (E3)-(E5) to become $\mathcal{H}_{i+1}$.  If $\mathcal{H}_i$ has no $(k,k+1)$-handleslide, this can be done without issue, and $\sigma$ is extended so that $c$ is not a switch.  If $\mathcal{H}_i$ does have a $(k,k+1)$-handleslide with coefficient $r \in \mathbb{F}^*$, it cannot be slid past the crossing; instead it is left behind adjacently to the left of the crossing.  All other handleslides slide past the crossing to form part of $\mathcal{H}_{i+1}$.  In doing so, additional handleslides may be produced from (E5) when the $(k,k+1)$-handleslide is passed, but these additional handleslides can then themselves be moved into $\mathcal{H}_{i+1}$.  Next, if the normality condition for $\sigma$ holds (resp. does not hold) at the left side of $c$, then $\sigma$ extends so that $c$ is a switch (resp. return).  Move (E1) is then applied to produce the remaining required handleslides (with coefficients determined by $r$ and $d_{x_i-\epsilon}$) as in Figure \ref{fig:SR}.  The use of the (E1) move produces handleslides in pairs, and the extra handleslides that are not needed for Figure \ref{fig:SR} are then moved into $\mathcal{H}_{i+1}$, leaving the handleslides in $I_c$ in the required form.

An important point is that with the handleslides and coefficients as in Figure \ref{fig:SR}, and with $d_{x_i-\epsilon}$ known to be Barannikov form, a computation shows that $d_{x}$ is also in Barannikov form for $x$ to the right of $I_c$ up to $\mathcal{H}_{i+1}$.  The way the Barannikov form differential changes is indicated in Figure \ref{fig:SR}.    

\medskip 

\noindent {\bf Case 3.}  $x_i$ is a {\it basepoint}. We use (E6) to sweep all of the handleslides past the basepoint, changing their coefficients appropriately.

\medskip

\noindent {\bf Case 4.}  $x_i$ is a {\it left cusp}.  We use (E3) to continuously sweep $\mathcal{H}_i$ past the cusp.  The new cusp sheets are paired by $\sigma$.

\medskip

\noindent {\bf Case 5.}  $x_i$ is a {\it right cusp} connecting sheets $S_k$ and $S_{k+1}$.  If there is a $(k,k+1)$-handleslide, $h_{k,k+1}$, (this is only possible if $\rho =1$), then pull it out of the left side of $\mathcal{H}_i$.  Next, note that for any $x$-value within $\mathcal{H}_i$ we have $\langle d_{x} S_{k+1}, S_k\rangle =1$.  [This is true just to the left of the cusp by (A2), and passing handleslides cannot change the matrix coefficients of $d_x$ corresponding to vertically adjacent sheets.]  For each $i<k$, if there is an $(i,k+1)$-handleslide, $h_{i,k+1}$, in $\mathcal{H}_i$ with coefficient $b_{i,k+1} \neq 0$, then we apply a $(i,k)$-super-handleslide point, aka Move $\leftarrow$(E7), at its location with coefficient $-b_{i,k+1}$.  Because $\langle d_{x} S_{k+1}, S_k\rangle =1$, this produces a new $(i,k+1)$-handleslide with coefficient $-b_{i,k+1}$ that we use to cancel $h_{i,k+1}$ and possibly some additional $(i',k)$- and $(i,j)$-handleslides with $i'<i$ and $k+1<j$.  Similarly, for each $k+1<j$, if there is a $h_{k,j}$ handleslide in $\mathcal{H}_i$ then we use a $(k+1,j)$-super-handleslide to cancel it, possibly introducing some additional $(i,j)$ and $(k+1,j')$-handleslides with $i<k$ and $j<j'$.  At this point there are no $(i,k+1)$ or $(k,j)$-handleslides left in $\mathcal{H}_i$, so the only handleslides with endpoints on the cusp sheets are $h_{k,k+1}$ (off to the left, should it exist), and some number of $(i,k)$- and $(k+1,j)$-handleslides.  We now sweep all other handleslides past the right cusp using (E3).  Possibly some additional $(i,k)$ and $(k+1,j)$-handleslides are created during this process, (but it is important to note that no new $(i,k+1)$ or $(k,j)$ handleslides appear).  Now we use (E2) to group the remaining $(i,k)$ and $(k+1,j)$ handleslides so that for each $i<k$ and $k+1<j$ there is at most one such handleslide.  We argue that the coefficients, $b_{i,k}$ and $b_{k+1,j}$, of all such handleslides actually has to be $0$.  The reason is that, since $d_x$ is in Barannikov form for $x$ to the left of $h_{k,k+1}$ and $\mathcal{H}_i$, $\langle d_xS_{k+1},S_k \rangle$ is the only nonzero matrix coefficient involving $S_k$ and $S_{k+1}$.  Thus, for such $d_x$ we have $\langle d_x S_{k+1}, S_i \rangle = 0$ for all $i<k$.  Therefore, one sees that after passing through the remaining handleslides of $\mathcal{H}_{i}$ we get $\langle d_{x'} S_{k+1}, S_i \rangle = b_{i,k}$ for $x'$ immediately to the left of the cusp.  Similarly, for $k+1 <j$ we get $\langle d_{x'} S_k, S_j \rangle = -b_{k+1,j}$.  Now, (A3) tells us that these matrix coefficients of $d_{x'}$ have to be zero.  At this point, the only handleslide remaining to the left of the right cusp is $h_{k,k+1}$ (should it exist), and its positioning is as allowed in Figure \ref{fig:SR}.

\medskip

Once we get to $x_i = x_N$ the construction is completed as follows:  In the $J^1[0,1]$ case, $\mathcal{H}_N$ just becomes the properly ordered handleslide collection at $x=1$.  In the $J^1S^1$ case, we combine $\mathcal{H}_{N}$ and $\mathcal{H}_0$ into a single group of handleslides and properly order them with Proposition \ref{prop:proper}.

\end{proof}

\subsection{Augmented cobordism and Legendrian isotopy}  \label{sec:LegIso}
Recall that when $\Lambda_t,  0 \leq t \leq 1$ is a Legendrian isotopy in $J^1M$ there is a Legendrian cobordism $\Sigma: \Lambda_0 \rightarrow \Lambda_1$ in $J^1(M \times [0,1])$ such that the $t$ slices of the front projection of $\Sigma$ coincide with the front projections of the $\Lambda_t$.  See eg. \cite[Section 3.3]{ArnoldWave}.

\begin{proposition} \label{prop:Rmove}
Let $(\Lambda_0, \mathcal{C}_0) \subset J^1M$ be an augmented Legendrian, and let $\Lambda_t$, $0\leq t\leq 1$ be a generic Legendrian isotopy with associated Legendrian cobordism $\Sigma \subset J^1([0,1]\times M)$.  Then, there exists MCFs $\mathcal{C}$ and $\mathcal{C}_1$ for $\Sigma$ and $\Lambda_1$ respectively such that $(\Sigma, \mathcal{C})$ is an augmented Legendrian cobordism from $(\Lambda_0,\mathcal{C}_0)$ to $(\Lambda_1, \mathcal{C}_1)$. 
\end{proposition}

\begin{proof}
See item (B2) in Section 7.2 of \cite{PanRu2} for a proof over $\mathbb{F}_2$.
From Proposition \ref{prop:MCF} we know that the MCF can be extended during a Reidemeister move of type (R0), (R2), or (R3) as long as there are no handleslides or basepoints in the interval where the move occurs.  Using the equivalence and basepoint moves (E1)-(E7) and (F1)-(F4) it is straightforward to arrange this.  [Considerations may be simplified by first applying Proposition \ref{prop:SRalgorithm} to arrange that $\mathcal{C}_0$ is in $SR$-from.  For the (R2) $\leftarrow$ move reflected in the horizontal direction, in the case that there is a handleslide at the second crossing (necessarilly a return of any normal ruling) a super-handleslide move (E7) can be applied to remove the handleslide as in \cite{PanRu2}.]  To perform (R1) $\rightarrow$ we can first apply (F1) to create the necessary spin basepoint.  To perform (R1) $\leftarrow$, note that it is enough to arrange for handleslides to be
positioned as in the statement of Lemma \ref{lem:R1}, as then the coefficients are correct in order for the move to be performed.  With $\mathcal{C}_0$ in $SR$-form the handleslides are in the correct positions, except in the case where $\rho=1$ when there may be an additional handleslide adjacent to the right cusp and connecting the two cusp sheets.  In this case, Figure \ref{fig:Type1handle} indicates how to move this handleslide out of the way, should it exist, arranging the positioning of Lemma \ref{lem:R1}.

\end{proof}

\begin{figure}[!ht]
\labellist
\tiny
\pinlabel $\mbox{(A)}$ [b] at -10 196
\pinlabel $\mbox{(B)}$ [b] at -10 42

\pinlabel $a$ [r] at 4 166
\pinlabel $-b$ [t] at 32 180
\pinlabel $1$ [b] at 44 230
\pinlabel $-1$ [t] at 60 180
\pinlabel $1$ [t] at 120 180
\pinlabel $b$ [t] at 142 208
\pinlabel $-a$ [l] at 176 166

\pinlabel $-ba$ [t] at 240 142
\pinlabel $-b$ [t] at 254 180
\pinlabel $-1$ [b] at 270 202
\pinlabel $1$ [b] at 330 202
\pinlabel $b$ [b] at 352 228

\pinlabel $-1$ [t] at 480 180
\pinlabel $1$ [t] at 538 180
\pinlabel $-ba$ [t] at 534 120
\pinlabel $b$ [b] at 548 232
\pinlabel $-b$ [t] at 574 180
\pinlabel $-ba$ [t] at 578 120
\pinlabel $-a$ [t] at 592 142
\pinlabel $1$ [b] at 562 230

\pinlabel $-1$ [t] at 690 180
\pinlabel $1$ [t] at 750 180
\pinlabel $-ba$ [t] at 794 142

\pinlabel $a$ [b] at 8 88
\pinlabel $-b$ [t] at 30 0
\pinlabel $1$ [b] at 44 50
\pinlabel $-1$ [t] at 60 0
\pinlabel $1$ [t] at 118 0
\pinlabel $-a$ [b] at 82 88
\pinlabel $b$ [b] at 142 48
\pinlabel $-a$ [b] at 172 88

\pinlabel $-b$ [b] at 256 50
\pinlabel $-1$ [t] at 270 0
\pinlabel $-a$ [b] at 292 88
\pinlabel $1$ [t] at 330 0
\pinlabel $b$ [b] at 352 48

\pinlabel $-1$ [r] at 478 12
\pinlabel $ab$ [t] at 494 0
\pinlabel $-a$ [l] at 506 72
\pinlabel $1$ [b] at 536 22
\pinlabel $b$ [b] at 546 52
\pinlabel $-b$ [t] at 572 0
\pinlabel $1$ [b] at 562 48

\pinlabel $ab$ [t] at 638 0
\pinlabel $-1$ [t] at 690 0
\pinlabel $-a$ [b] at 712 88
\pinlabel $1$ [t] at 750 0

\endlabellist

\centerline{\includegraphics[scale=.6]{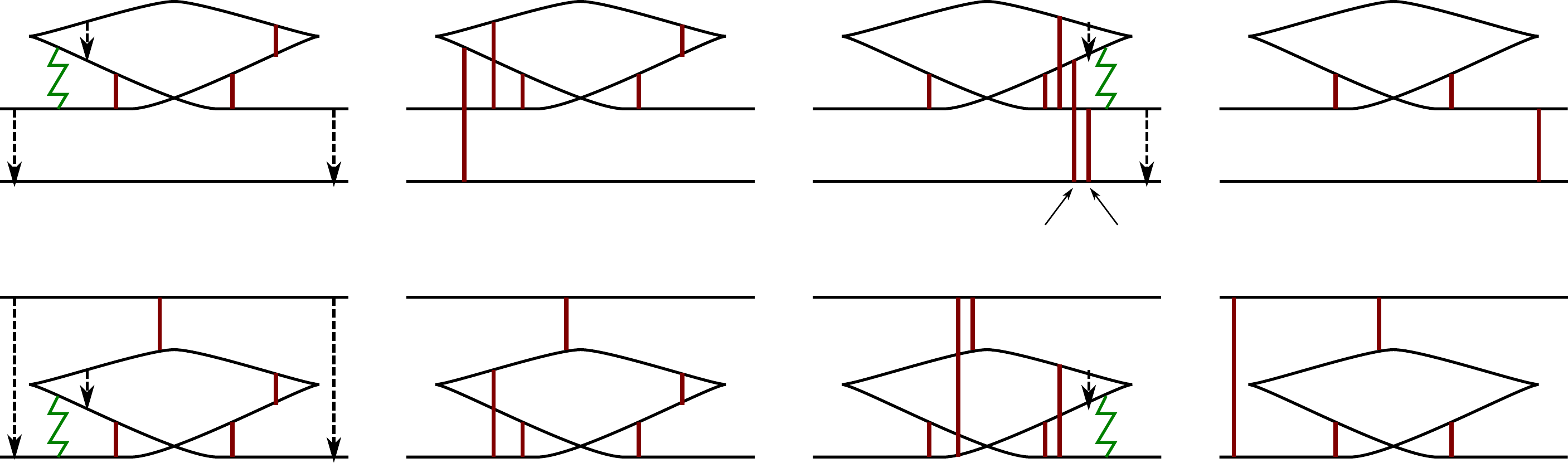}}

\quad

\caption{Moving a handleslide at a right cusp in an $SR$-form MCF away from an (R1) move.  Two cases (A) and (B) are pictured depending on the  configuration of the normal ruling at the (R1) move.  The green lightning bolts are super-handleslides.  A similar cobordism (not pictured) applies in the case where the bottom strand of the (R1) move is a fixed point strand.  The vertical reflection is similar.  
}
\label{fig:Type1handle}
\end{figure}
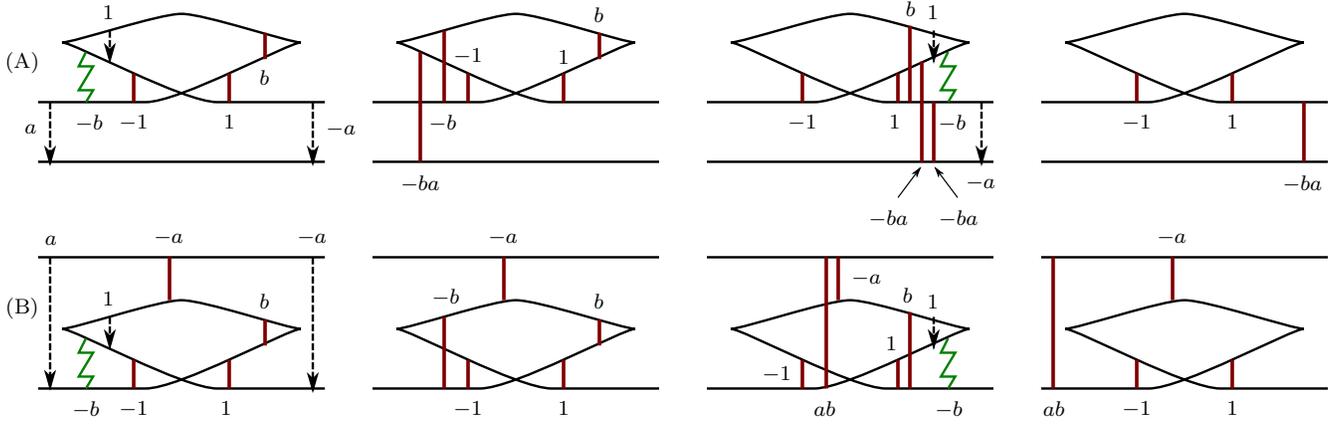

\section{Cobordism invariants of augmented Legendrians} \label{sec:invariants}
In this section we record a complete set of cobordism invariants for $\rho$-graded augmented Legendrians in $J^1S^1$ and (full) $n$-tangles in $J^1[0,1]$.  The spin invariant, defined in the case when $\mbox{Char} \, \mathbb{F} \neq 2$ and $\rho$ is even is an element of $\Z/2$ determined by the combinatorial spin structure of  $(\Lambda,\mathcal{C})$.  A richer invariant, the graded monodromy matrix, arises from the continuation maps of Section \ref{sec:2-4}.  In the $J^1S^1$ case, only the conjugacy class of the monodromy matrix is a well-defined invariant.

\subsection{Full augmented Legendrian $n$-tangles}  \label{sec:fullaug} Let $n \geq 0$ and $\vec{\mu} =( \mu_1, \ldots, \mu_n) \in (\Z/\rho)^n$.  We say that a $\rho$-graded Legendrian $n$-tangle, $\Lambda \subset J^1[0,1]$, has {\bf boundary Maslov potential} $\vec{\mu}$ if 
on the $n$ boundary points of $\Lambda$ at $x=0$ and $x=1$ the Maslov potential satisfies
\[
\mu(S^0_i) = \mu(S^1_i) = \mu_i, \quad \mbox{for $1 \leq i \leq n$.}
\]

\begin{definition} An augmented Legendrian $n$-tangle $(\Lambda, \mathcal{C})$ is {\bf full}  if the differentials of $\mathcal{C}$ at $x=0$ and $x=1$ are both $0$, i.e. $d_0:C_0 \rightarrow C_0$ and $d_1:C_1 \rightarrow C_1$ have $d_0=d_1 =0$.  For fixed $n$ and $\vec{\mu}$ we write $\mathit{Leg}^\rho_{\vec{\mu}}(J^1[0,1]; \mathbb{F})$ for the set of $\rho$-graded full augmented Legendrian $n$-tangles with boundary Maslov potential $\vec{\mu}$ and $\mathit{Cob}^\rho_{\vec{\mu}}(J^1[0,1]; \mathbb{F})$ for the corresponding set of cobordism classes. 
\end{definition}

Observe that concatenation induces a well-defined operation on $\mathit{Cob}^\rho_{\vec{\mu}}(J^1[0,1]; \mathbb{F})$, 
\[
[(\Lambda_1, \calC_1)] \cdot [(\Lambda_2, \calC_2)] = [(\Lambda_1 *\Lambda_2, \calC_1* \calC_2)],
\] 
with $\Lambda_1$ placed to the left of $\Lambda_2$ when forming $\Lambda_1 *\Lambda_2$; since the augmented tangles are full the differentials from $\calC_1$ and $\calC_2$ agree where $\Lambda_1$ and $\Lambda_2$ are joined.  Moreover, for $(\Lambda, \mathcal{C}) \in \mathit{Leg}^\rho_{\vec{\mu}}(J^1[0,1]; \mathbb{F})$ we can form its {\bf closure}, $(\widehat{\Lambda}, \widehat{\mathcal{C}}) \subset J^1S^1$, by identifying $x=0$ with $x=1$ and gluing corresponding sheets of $\Lambda$.  This provides a well-defined closure map on cobordism classes, $\mathit{Cob}^\rho_{\vec{\mu}}(J^1[0,1]; \mathbb{F}) \rightarrow \mathit{Cob}^\rho(J^1S^1; \mathbb{F})$.

\begin{proposition} For any $\rho$ and $\vec{\mu}$,  
$\mathit{Cob}^\rho_{\vec{\mu}}(J^1[0,1]; \mathbb{F})$ is a group with respect to concatenation.
\end{proposition}
\begin{proof}  Clearly the augmented tangle consisting of $n$ parallel strands is an identity element $\mathbf{1}_n$. 
The inverse for an augmented $n$-tangle $(\Lambda, \mathcal{C})$ in $\mathit{Cob}^\rho_{\vec{\mu}}(J^1[0,1]; \mathbb{F})$ is provided by the {\bf reverse}, $(\overline{\Lambda},\overline{\mathcal{C}})$, obtained from reflecting the front projection of $(\Lambda,\mathcal{C})$ across the vertical line $x=1/2$ and negating (resp. inverting) all coefficients of handleslides (resp. basepoints).  [A cobordism from $(\Lambda, \mathcal{C})*(\overline{\Lambda},\overline{\mathcal{C}})$ to $\mathbf{1}_n$ arises from applying a combination of the moves (E1) $\rightarrow$, (F1) $\rightarrow$, (C3) $\leftarrow$, (C2) $\rightarrow$, and (C1) $\leftarrow$ at the location where the concatenation occurs to successively cancel each singularity of $(\Lambda, \mathcal{C})$ with its reflection in $(\overline{\Lambda},\overline{\mathcal{C}})$. Moreover, $(\overline{\Lambda},\overline{\mathcal{C}})* (\Lambda, \mathcal{C})= (\overline{\Lambda},\overline{\mathcal{C}})* (\overline{\overline{\Lambda}},\overline{\overline{\mathcal{C}}}) \sim \mathbf{1}_n$ follows also.]
\end{proof}

\subsection{Spin invariants}
When $\mbox{Char}\, \mathbb{F} \neq 2$ and the grading $\rho \in \Z_{\geq 0}$ is even (in particular, this implies that $\rho$-graded cobordisms are orientable)
we can associate a {\bf spin invariant}, $\xi(\Lambda, \mathcal{C}) \in \Z/2$, to an augmented Legendrian with coefficients in $\mathbb{F}$ as follows:
\begin{itemize}
\item For $(\Lambda, \mathcal{C}) \subset J^1S^1$, define $\xi(\Lambda, \mathcal{C})$ as the mod $2$ sum
\[
\xi(\Lambda, \mathcal{C}) =  \# \mbox{(spin basepoints)} +  \# \mbox{(right cusps)} + \# \mbox{(components of $\Lambda$)}.
\]
\item For $(\Lambda, \mathcal{C}) \in \mathit{Leg}^\rho_{\vec{\mu}}(J^1[0,1]; \mathbb{F})$, define
\[
\xi(\Lambda, \mathcal{C}) = \xi(\widehat{\Lambda}, \widehat{\mathcal{C}}) +n.
\]
\end{itemize}

\begin{proposition} \label{prop:SpinInvariant} 
The spin invariant $\xi(\Lambda, \mathcal{C}) \in \Z/2$ depends only on the cobordism class of $(\Lambda, \mathcal{C})$. 
\end{proposition}

\begin{proof}
From Proposition \ref{prop:cobord} it suffices to establish that $\xi(\Lambda, \mathcal{C})$ is unchanged when one of the moves
(E1)-(E7), (F1)-(F4), (R0)-(R3), and (C1)-(C3) is performed.  
All of the quantities involved in the definition of $\xi$ are unchanged by any of the moves (E1)-(E7), (F2)-(F4), (R0), (R2)-(R3), and (C3).  In move (F1), the number of spin basepoints increases or decreases by $2$. In (R1), a spin basepoint disappears but is replaced with a new right cusp.  In (C1), the number of right cusps and components both decrease by $1$ when the unknotted component is removed.  Finally, when (C2) $\rightarrow$ is performed the number of right cusps decreases by $1$, and the number of components must either decrease or increase by $1$ depending on whether the two cusps belonged to different components or the same component before the move.  [In the case that they belong to the same component, it is crucial that the grading $\rho$ is even so that the two parallel strands on the right side of (C2) have opposite orientations.  If the orientations were the same then the number of components would not be changed.]
\end{proof}

\begin{remark}
The reason for the addition of $n$ in defining $\xi$ in the $J^1[0,1]$ case is so that the spin invariant will provide a group homomorphism $\mathit{Cob}^\rho_{\vec{\mu}}(J^1[0,1]; \mathbb{F})\rightarrow \Z/2$.
\end{remark}

\subsection{Monodromy matrices}

Given $\vec{\mu} = (\mu_1, \ldots, \mu_n)$ we write
\[
\graded{n}: \Z/\rho \rightarrow \Z_{\geq 0}, \quad \graded{n}(l) = \#\{i \,|\, \mu_i = l\}.
\] 
For $(\Lambda, \mathcal{C}) \in \mathit{Leg}^\rho_{\vec{\mu}}(J^1[0,1]; \mathbb{F})$, 
$\graded{n}$ is the {\bf graded dimension} of the fiber cohomology 
\[
\graded{n}(l) = \dim H^l( C_{x}, d_x), \quad l \in \Z/\rho
\]
for $x \in [0,1]^\mathcal{C}_\mathit{reg}$.  [This is clear at $x=0$ or $1$ since $d_0=d_1=0$ as $(\Lambda,\mathcal{C})$ is full, and the fiber homology is independent of $x$ with isomorphisms provided by the continuation maps.]
For $m=0,1$ we have an ordered basis $[S^m_{1}], \ldots, [S^m_{n}] \in H^*(C_m, d_m)$ giving rise to ordered bases in each graded component $H^l(C_m,d_m)$ for $l \in \Z/\rho$.  
To  $(\Lambda, \mathcal{C})$ we associate the {\bf graded monodromy matrix} 
\[
\graded{M}_{\Lambda, \calC}  \in GL(\graded{n}, \mathbb{F}) := \prod_{l \in \Z/\rho} GL(\graded{n}(l), \mathbb{F})
\]
that is the matrix with respect to the above bases of the continuation isomorphism (as defined in Section \ref{sec:2-4}),
\[
\phi_{\Lambda, \calC}:H^*(C_0,d_0) \rightarrow H^*(C_1,d_1), 
\] 
associated to the path $\sigma(t) = t$ from $0$ to $1$ in $[0,1]$.  That is, $\graded{M}_{\Lambda, \calC}(l) \in GL(\graded{n}(l), \mathbb{F})$ is the matrix of $\phi_{\Lambda, \calC}:H^l(C_0,d_0) \rightarrow H^l(C_1,d_1)$ in cohomological grading $l$. 

\begin{remark} \label{rmk:onemono} When it is convenient, we can also view $\graded{M}_{\Lambda, \calC}$ as the single $n\times n$ matrix $M_{\Lambda, \calC} \in GL(n,\mathbb{F})$ of $\phi_{\Lambda, \calC}$ considered as a linear map  $\oplus_l H^l(C_0,d_0) \rightarrow \oplus_l H^l(C_1,d_1)$.  The individual $\graded{M}_{\Lambda, \calC}(l)$ are the sub-matrices of $M_{\Lambda, \calC}$ obtained from the rows and columns corresponding to sheets with Maslov potential $\mu(S_m^k) = l$.  This amounts to embedding $GL(\graded{n}, \mathbb{F})= \prod_{l \in \Z/\rho} GL(\graded{n}(l), \mathbb{F}) \subset GL(n,\mathbb{F})$ as the subgroup of degree $0$ linear automorphisms.
\end{remark}

\begin{proposition}\label{prop:inva}
If $(\Lambda_i, \calC_i)$ for $i=0,1$ are cobordant full augmented Legendrian $n$-tangles in $J^1[0,1]$, then their monodromy matrices are equal, $\graded{M}_{\Lambda_0, \calC_0}=\graded{M}_{\Lambda_1, \calC_1}$.
\end{proposition}

\begin{proof}
Given a cobordism $(\Sigma, \calC) \subset J^1([0,1]\times [0,1])$  from  $(\Lambda_0, \calC_0) \subset J^1([0,1]\times\{0\})$ to $(\Lambda_1, \calC_1) \subset J^1([0,1]\times\{1\})$, the continuation maps for $(\Sigma,\calC)$ associated to the two paths around the
boundary of $[0,1]\times[0,1]$ from $(0,0)$ to $(1,1)$ agree with $\phi_{\Lambda_0, \calC_0}$ and $\phi_{\Lambda_1, \calC_1}$
respectively.  [The continuation maps for the vertical edges $\{0\} \times[0,1]$ and $\{1\} \times[0,1]$ are the identity map because these edges do not intersect any singularities of $(\Sigma, \calC)$.]  Thus, Proposition \ref{prop:continuation} implies that $\phi_{\Lambda_0, \calC_0} = \phi_{\Lambda_1, \calC_1}$.  
\end{proof}

For an augmented Legendrian $(\Lambda, \calC)$ in $J^1(S^1)$,  
we have a monodromy map 
\[
\phi_{\Lambda, \calC}: H^*(C_0, d_0) \to H^*(C_0, d_0)
\]
that is the continuation map on homology for the loop $\sigma: [0,1] \rightarrow S^1= [0,1]/\{0,1\}, \sigma(t) =t$.
In this case, we form the {\bf graded monodromy matrix} $\graded{M}_{\Lambda, \calC} \in GL(\graded{n},\mathbb{F})$, where $\graded{n}$ is the graded dimension of the fiber cohomology of $\calC$, by making a choice of basis of $H^l(C_0, d_0)$ for each $l\in \Z/\rho$.  Note that $\graded{M}_{\Lambda, \calC}$ is well-defined up to conjugacy in $GL(\graded{n},\mathbb{F})$.

\begin{proposition}\label{prop:invaS1}
If two augmented Legendrians $(\Lambda_0, \calC_0)$ and $(\Lambda_1, \calC_1)$ in $J^1S^1$ are cobordant, then their fiber cohomologies have the same graded dimension $\graded{n}$ and their monodromy matrices $\graded{M}_{\Lambda_0, \calC_0}$ and $\graded{M}_{\Lambda_1, \calC_1}$ are conjugate in $GL(\graded{n},\mathbb{F})$. 
\end{proposition}

\begin{proof}
If $(\Sigma, \mathcal{C}) \subset J^1(S^1\times [0,1])$ is a cobordism between $(\Lambda_0, \calC_0)$ and $(\Lambda_1, \calC_1)$, then 
the continuation map, $\phi_{\sigma_0}$, associated to the path $\sigma_0(t) = (0, t)$ provides an isomorphism between the fiber cohomologies of $\calC_0$ and $\calC_1$.  Moreover, using Proposition
\ref{prop:continuation} 
we have 
\[
\phi_{\Lambda_0, \calC_0}=\phi^{-1}_{\sigma_0} \circ \phi_{\Lambda_1, \calC_1} \circ \phi_{\sigma_0}.
\]
Choosing bases we see that the monodromy matrices are conjugate.
\end{proof}

\section{Computation of cobordism classes in $J^1[0,1]$}  \label{sec:J101}

As in previous sections, the augmented Legendrians under consideration, $(\Lambda, \mathcal{C})$, are $\rho$-graded Legendrians equipped with MCFs.
In this section we establish the following classification of full augmented Legendrian $n$-tangles in $J^1[0,1]$ up to cobordism.  Recall the notations $\mathit{Cob}^\rho_{\vec{\mu}}(J^1[0,1]; \mathbb{F})$ for the set of cobordism classes of full augmented $n$-tangles with boundary Maslov potential $\vec{\mu}$ with concatenation; $\xi(\Lambda, \calC)$ and $\graded{M}_{\Lambda, \calC}$ are the spin and monodromy matrix invariants.  We use $GL(\graded{n}, \mathbb{F})^\mathit{op}$ to denote $GL(\graded{n}, \mathbb{F})$ with the opposite operation, $\mathbf{A}\cdot \mathbf{B} = \mathbf{BA}$.

\begin{theorem}\label{thm:J101}
Let $n \geq 0$ and $\vec{\mu} \in (\Z/\rho)^{n}$.
\begin{enumerate}
\item When $\mbox{Char}\,\mathbb{F}\neq 2$ and $\rho$ is even, the map
\[
\begin{array}{ccc}
\Phi: \mathit{Cob}^\rho_{\vec{\mu}}(J^1[0,1]; \mathbb{F})  & \stackrel{\cong}{\rightarrow} & GL(\graded{n}, \mathbb{F})^\mathit{op} \times \Z/2 \\
  \mbox{$[(\Lambda, \calC)]$}  & \mapsto & \left(\graded{M}_{\Lambda, \calC}, \xi(\Lambda, \calC)\right) 
\end{array}
\]
is a group isomorphism.  
\item When $\mbox{Char} \, \mathbb{F} =2$, except in the case when $n=0$ and $\rho = 1$, the map
\[
\begin{array}{ccc}
\Phi: \mathit{Cob}^\rho_{\vec{\mu}}(J^1[0,1]; \mathbb{F})  & \stackrel{\cong}{\rightarrow} & GL(\graded{n}, \mathbb{F})^\mathit{op}  \\
  \mbox{$[(\Lambda, \calC)]$}  & \mapsto & \graded{M}_{\Lambda, \calC} 
\end{array}
\]
is a group isomorphism.
\item When $n=0$ and $\rho=1$, 
$\mathit{Cob}^\rho_{\vec{\mu}}(J^1[0,1]; \mathbb{F}) = \{[(U, \mathcal{C}_b)] \,|\, b \in \mathbb{F}\}$ where $(U, \mathcal{C}_b)$ is the standard Legendrian unknot having a single handleslide with coefficient $b$ as in Figure \ref{fig:UnknotStd}.   
\end{enumerate}
\end{theorem}

\begin{remark}
In Theorem \ref{thm:J101} (3) it should be emphasized that it is not known whether the $[(U, \mathcal{C}_b)]$ are all distinct or not; for any $b \in \mathbb{F}$ the monodromy matrix is $0$.  We conjecture that they are distinct simply because we have not been able to construct a cobordism between $(U,\mathcal{C}_{b_1})$ and $(U,\mathcal{C}_{b_2})$ when $b_1 \neq b_2$.
\end{remark}

The proof of Theorem \ref{thm:J101} (appearing at the end of the section) is based on several ingredients.  From Proposition \ref{prop:inva}, we have that $\Phi$ is a well-defined map.  
Next, Proposition \ref{prop:surj} will establish the surjectivity of $\Phi$.  Injectivity requires more work and is established
	according to the following outline.  After a preparatory discussion about positive braids 
we introduce 
a class of standard form augmented Legendrian $n$-tangles that for $n \geq 1$ are certain {\it simple positive permutation braids} with MCFs 
in a special case of the $SR$-form with the underlying normal ruling consisting entirely of fixed point strands.  To show that $\Phi$ is injective, we show (i) any two standard form augmented Legendrians with the same monodromy matrix and spin invariant (when defined) are cobordant (in fact, isotopic), and (ii) an arbitrary full augmented Legendrian $n$-tangle is cobordant to one in standard form.  See Corollary \ref{cor:standard}  and Proposition \ref{prop:standardform} respectively.

\subsection{Surjectivity of $\Phi$} \label{sec:surjPhi}

We call a (full augmented) Legendrian $n$-tangle in $J^1[0,1]$ without cusps a  ({\bf full augmented}) {\bf Legendrian $n$-braid}.

\begin{proposition}\label{prop:surj}
For any boundary Maslov potential $\vec{\mu}$, and any graded matrix $\graded{P}\in GL(\graded{n}, \mathbb{F})$, there exists a full augmented Legendrian $n$-braid $(\Lambda, \calC) \in \mathit{Leg}^\rho_{\vec{\mu}}(J^1[0,1]; \mathbb{F})$ with monodromy matrix $\graded{M}_{\Lambda, \calC} = \graded{P}$.  Moreover, when $\mbox{Char}\, \mathbb{F} \neq 2$ and $\rho$ is even, the spin invariant $\xi(\Lambda, \calC)\in \Z/2$ may be arbitrary. 
\end{proposition}

\begin{proof}  First, considering the case where the grading is concentrated in a single degree, we  prove that any matrix $P \in GL(n,\mathbb{F})$ can be realized as the monodromy matrix of a full augmented Legendrian $(\Lambda, \calC) \subset J^1[0,1]$ whose boundary strands all have the same grading, $\mu_i = l$.  Given $P \in GL(n,\Z_2)$, the usual Gauss-Jordan elimination algorithm decomposes $P$ into a product of the following types of elementary matrices used in the algorithm for row interchanges, row additions, and row scaling:
\begin{itemize}
\item[(i)] $Q_{(i\,j)}$ for $i<j$: obtained from interchanging the $i$-th and $j$-th rows of the identity matrix;
\item[(ii)] $I+bE_{i,j}$ for $i<j$, $b\in \mathbb{F}$;
\item[(iii)] $\Delta_{i}(s)$ for $1 \leq i \leq n$, $s \in \mathbb{F}$: diagonal matrices with $s$ at the $(i,i)$-position and $1$'s at the other diagonal entries.  
\end{itemize}
[Note that $I+bE_{j,i}$ with $i<j$ can be obtained as $Q_{(i\,j)}(I+bE_{i,j})Q_{(i\,j)}$.]
Any such elementary matrix is the monodromy matrix of a full augmented Legendrian $n$-tangle:  
(i) A single crossing between the $i$-th and $(i+1)$-th strands with no handleslides has monodromy matrix $Q_{(i\,i+1)}$, and so arbitrary permutation matrices, eg. the $Q_{(i\,j)}$, can be realized by concatenating crossings. (ii)  The matrix $I+bE_{i,j}$ is the monodromy matrix of a trivial $n$-braid having a single handleslide with coefficient $b$ connecting the $i$-th and $j$-th strands. (iii) A base point on the $i$-th strand with coefficient $s$ (with the left to right co-orientation) has monodromy  matrix $\Delta_i(s)$.  Concatenating appropriately then produces $(\Lambda, \calC)$ with $M_{\Lambda, \calC} = P$.  

Now, for general $\vec{\mu}$ and $\graded{P} \in GL(\graded{n}, \mathbb{F})$ we proceed as follows.  For each graded part $\graded{P}(l)$ of $\graded{P}$, from Step 1 we have an augmented braid $(\Lambda(l), \calC(l))$ with all the braid strands having grading $n(l)$ and with monodromy matrix $\graded{P}(l)$.  Stack the $(\Lambda(l), \calC(l))$ on top of one another so that the boundary strand gradings are non-decreasing from top to bottom, and call this augmented braid $(\widetilde\Lambda, \widetilde{\calC})$; its graded monodromy matrix is $\graded{P}$.  The boundary Maslov potential, $\vec{\mu}'$ of $(\widetilde\Lambda, \widetilde{\calC})$ is a reordering of $\vec{\mu}$.  We then form a concatenation
\[
(\Lambda, \calC) := (B, \calC_B)* (\widetilde\Lambda, \widetilde{\calC})* (\overline{B}, \overline{\calC_B}).
\]
Here, $B$ is a positive braid with left (resp. right) boundary Maslov potential $\vec{\mu}$ (resp. $\vec{\mu}'$) such that all the crossings of the $B$ are between strands with different potentials; the MCF $\calC_B$ on $B$ has no handleslides and all $d_x=0$; and $(\overline{B}, \overline{\calC_B})$ denotes the reverse of $(B, \calC_B)$.
The concatenation $(\Lambda, \calC)$ now has the correct boundary Maslov potential $\vec{\mu}$ and has graded monodromy matrix $\graded{P}$ since the crossings of $B$ and $\overline{B}$ between strands with different potentials do not affect the graded components of the monodromy matrix.

For the remaining claim regarding $\xi(\Lambda, \calC)$, note that adding a pair of consecutive base points consisting of a spin point and a homology base point with coefficient $-1$ will change the value of the spin invariant without effecting the monodromy matrix.
\end{proof}

\subsection{Aside: Simple positive permutation braid words} \label{sec:posbraid}
A Legendrian $n$-braid may be written as a product (left to right concatenation) of positive braid generators, $\sigma_1, \ldots, \sigma_{n-1}$, of the braid group $\mathcal{B}_n$ where we view $\sigma_i$ as the Legendrian $n$-tangle consisting of a single crossing between strands $i$ and $i+1$.  Moreover, 
it is a result of Garside \cite[Theorem 4]{Garside} that two positive braid words $w_1$ and $w_2$ are equal in $\mathcal{B}_n$ if and only they can be related using the braid relations $\sigma_{i}\sigma_{i+1}\sigma_{i} = \sigma_{i+1}\sigma_{i}\sigma_{i+1}$ for $1\leq i \leq n-2$ and $\sigma_i\sigma_j = \sigma_j\sigma_i$ for $|i-j|\geq 2$ in the monoid generated by the $\sigma_i$ {\it without ever needing to introduce inverses of the $\sigma_i$}.  As the relations arise from Legendrian isotopies, we see that two positive braid words that are equal in $\mathcal{B}_n$ represent Legendrian isotopic Legendrian braids.

A {\bf positive permutation braid} is
 a positive braid such that each pair of strands intersects at most once.   
There is a one-to-one correspondence between positive permutation braids (as elements of $\mathcal{B}_n$)  and elements of the symmetric group $S_n$; see \cite{EM}.  In particular, for each $\pi \in S_n$ there is a {\bf Legendrian permutation braid} $B_\pi$ well defined up to Legendrian isotopy.  
Here, we use the convention\footnote{This convention results in the map from $B_n$ to $S_n$ being an anti-homomorphism rather than a homomorphism.} 
that the permutation $\pi \in S_n$ associated to a braid in $B_n$ is the permutation of strands read from {\it left to right}, i.e. the strand that starts at position $i$ at $x=0$ ends in position $\pi(i)$ at $x=1$. 

One way to characterize Legendrian permutation braids is that if a braid is not a permutation braid, then after a Legendrian isotopy it contains a product $\sigma_i^2$, i.e. an eye shaped region as in the clasp move (C3).
This follows from  \cite[Proposition 2.4]{EM}, by making use of the above correspondence from positive braids in $B_n$ to Legendrian isotopy classes. 

\begin{lemma}[\cite{EM}] \label{lem:EM}
Let $B_\pi$ be a Legendrian permutation braid  for the permutation $\pi\in S_n$.  If the $i$-th and $(i+1)$-th strands at $x=1$  intersect in the braid, then there is  a Legendrian isotopy from $B_{\pi}$ to $B'\sigma_i$, where $B'$ is another Legendrian permutation braid.
\end{lemma}

\begin{corollary}\label{cor:eye}
If a Legendrian braid $B$ is not a permutation braid, then it is Legendrian isotopic to $B' \sigma_i^2 B''$ for some braid generator $\sigma_i$ and Legendrian braids $B'$ and $B''$.
\end{corollary}
\begin{proof}
 Write $B$ as a  braid word $B=\sigma_{i_1}\sigma_{i_2}\cdots \sigma_{i_n}$.
If $B$ is not a positive permutation braid, there is a $k$ such that $\sigma_{i_1}\sigma_{i_2}\cdots \sigma_{i_{k-1}}$ is a positive permutation braid but $\sigma_{i_1}\sigma_{i_2}\cdots \sigma_{i_{k}}$ is not.
Thus, the $i_k$ and $i_{k}+1$ strands (as labelled at the right side) cross in 
$\sigma_{i_1}\sigma_{i_2}\cdots \sigma_{i_{k-1}}$.
By Lemma \ref{lem:EM}, $\sigma_{i_1}\sigma_{i_2}\cdots \sigma_{i_{k-1}}$ is Legendrian isotopic to $B'\sigma_{i_k}$ and thus $B$ is Legendrian isotopic to the required form.
\end{proof}

\begin{definition}
We call a Legendrian permutation braid {\bf simple} if it is {\it not} possible to arrange for the product $\sigma_{i+1}\sigma_{i}\sigma_{i+1}$ to appear in its braid word by applying a sequence of commutation relations,  $\sigma_i\sigma_j=\sigma_j \sigma_i$ for 
$|i-j|\neq 1$. 
\end{definition}
A similar notion called {\it reduced} positive permutation braid (which instead forbids $\sigma_{i}\sigma_{i+1}\sigma_i$) is used by K\'{a}lm\'{a}n in \cite{Kalman}.

\begin{proposition}\label{prop:ex}
Any Legendrian permutation braid is Legendrian isotopic to one that is simple.
\end{proposition}
\begin{proof}
Start with any  positive permutation braid word $w$.
If $w$ is not simple, then after a sequence of commutation relations the product
$\sigma_{i+1}\sigma_{i}\sigma_{i+1}$ shows up.  Change it to  $\sigma_{i}\sigma_{i+1}\sigma_i$ via the braid relation, and then repeat the process.
Note that the braid relation decreases the sum of all the subscripts $j$ of braid generators $\sigma_j$ appearing in the word,  while the commutation relation preserves this sum. 
Thus, the procedure must stop at some point resulting in a simple braid word.
\end{proof}

Intuitively, a Legendrian permutation braid is simple if it only admits type $(b)$ but not type $(a)$ triangles in Figure~\ref{fig:triangle}.  The following lemma makes this statement precise.

\begin{figure}[!ht]
\labellist
\small
\pinlabel $(a)$ at 90 -15
\pinlabel $(b)$ at 380 -15
\endlabellist
\includegraphics[width=3in]{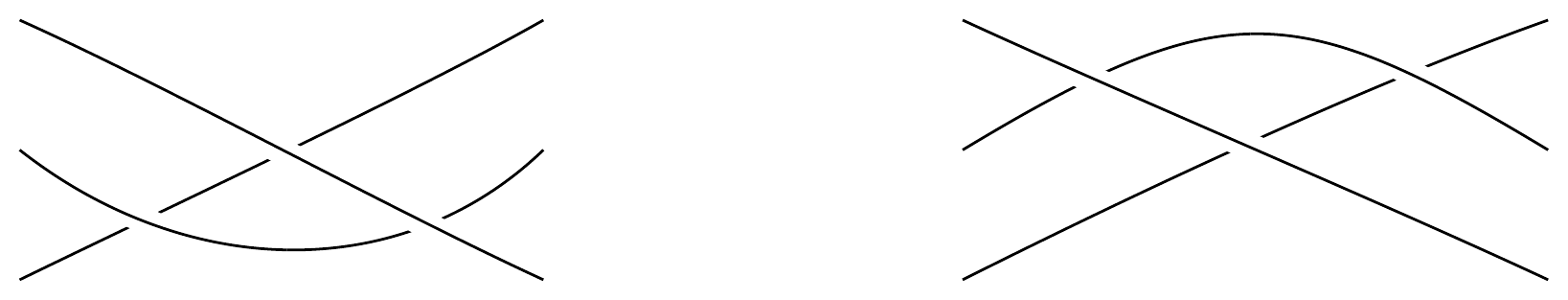}
\vspace{0.1in}
\caption{Simple braids only admit the type $(b)$ triangles but not type $(a)$ triangles.}
\label{fig:triangle}
\end{figure}

\begin{lemma} \label{lem:order}
Let $B_{\pi}$ be a simple Legendrian permutation braid for a permutation $\pi \in S_n$.
Denote by $c_{i,j}$ a crossing between the strands labeled $i$ and $j$ at the left boundary of $B_\pi$. 
For any $i<j<k$ such that $\pi(k)<\pi(j)<\pi(i)$, the crossing $c_{i,j}$  appears to the left of the crossing $c_{j,k}$ (as in the braid shown in Figure \ref{fig:triangle} (b)).
\end{lemma}

\begin{proof}
Suppose $B_{\pi}= \sigma_{i_1}\cdots \sigma_{i_m}$ 
and assume $\sigma_{i_1}\cdots\sigma_{i_{\ell-1}}$ satisfies the above condition but $\sigma_{i_1}\cdots \sigma_{i_\ell}$ does not. 
Say the crossing $\sigma_{i_\ell}$ is $c_{i,j}$ with $i<j$.
Among all the strands $k'$ such that the triple $(i,j,k')$ violates the condition in Lemma \ref{lem:order} (i.e., $i<j<k'$ and $\pi(k')<\pi(j)<\pi(i)$, but the crossing $c_{i,j}$ happens on the right of the crossing $c_{j,k'}$), take the strand  such that the crossing $c_{j,k'}$ is the closest one to the crossing $c_{i,j}$ in the braid word $B_{\pi}$; call this strand $k$. 
Thus the strands $i,j$ and $k$ form a  triangular region as shown in Figure \ref{fig:triangle}(a).  

We claim that no other strand intersects the triangular region.
The claim can be proved by contradiction.
If there was a strand $s$ that intersects the triangular region, there are three cases as follows.
\begin{itemize}
\item The strand $s$ enters the triangular region
from the upper left edge. 
If the strand goes out of the triangular region by crossing the bottom edge as shown in Figure \ref{fig:tri}(a), which implies $s<j$, then the triple $(s, j, k)$  violates the condition and thus contradicts the choice of $c_{i,j}$.
Otherwise, the strand $s$ leaves the triangular region along the upper right edge as shown in Figure \ref{fig:tri}(b). Then, $(i,s,k)$ forms another triple that violates the condition and again contradicts the choice of $c_{i,j}$.
\item The strand $s$ enters the triangular region from the upper right edge, which implies $s<i$ and thus $s<j<k$. The strand $s$ can only leave the triangular region along the bottom edge as shown in Figure \ref{fig:tri}(c). Therefore, we have $\pi(s)> \pi(j)$ and thus $\pi(s)>\pi(k)$, which implies that strand $s$ and strand $k$ intersect.
 Thus the triple $(s, j,k)$ violates  the condition and contradicts  the choice of $c_{i,j}$.  
\item The strand $s$ enters the triangular region from the bottom edge as shown in Figure \ref{fig:tri}(d) and (e), which implies $s>j$ and $\pi(s)<\pi(j)$. Thus, $s>i$ and $\pi(s)<\pi(i)$, 
so the triple $(i,j,s)$ violates the condition.  This contradicts the choice of $k$.
\end{itemize}
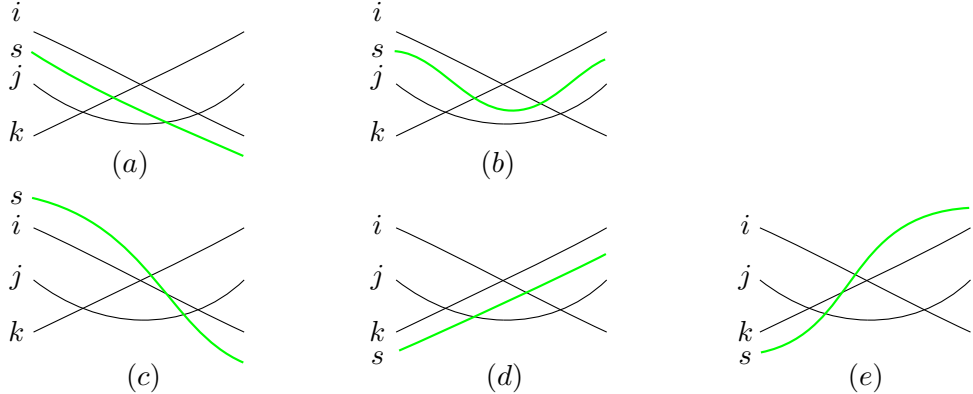
\begin{figure}[!ht]
\labellist
\small
\pinlabel $i$ at  -5 270
\pinlabel $j$ at -5 220
\pinlabel $k$ at -5 180
\pinlabel $s$ at -5 240
\pinlabel $(a)$ at  80 155

\pinlabel $i$ at  -5 110
\pinlabel $j$ at -5 70
\pinlabel $k$ at -5 30
\pinlabel $s$ at -5 130
\pinlabel $(c)$ at  90 -5

\pinlabel $i$ at  265 270
\pinlabel $j$ at 265 220
\pinlabel $k$ at 265 180
\pinlabel $s$ at 265 240
\pinlabel $(b)$ at  355 155

\pinlabel $i$ at  265 110
\pinlabel $j$ at 265 70
\pinlabel $k$ at 265 30
\pinlabel $s$ at 265 10
\pinlabel $(d)$ at  360 -5

\pinlabel $i$ at  540 110
\pinlabel $j$ at 540 70
\pinlabel $k$ at 540 30
\pinlabel $s$ at 540 10
\pinlabel $(e)$ at  630 -5

\endlabellist
\includegraphics[width=5in]{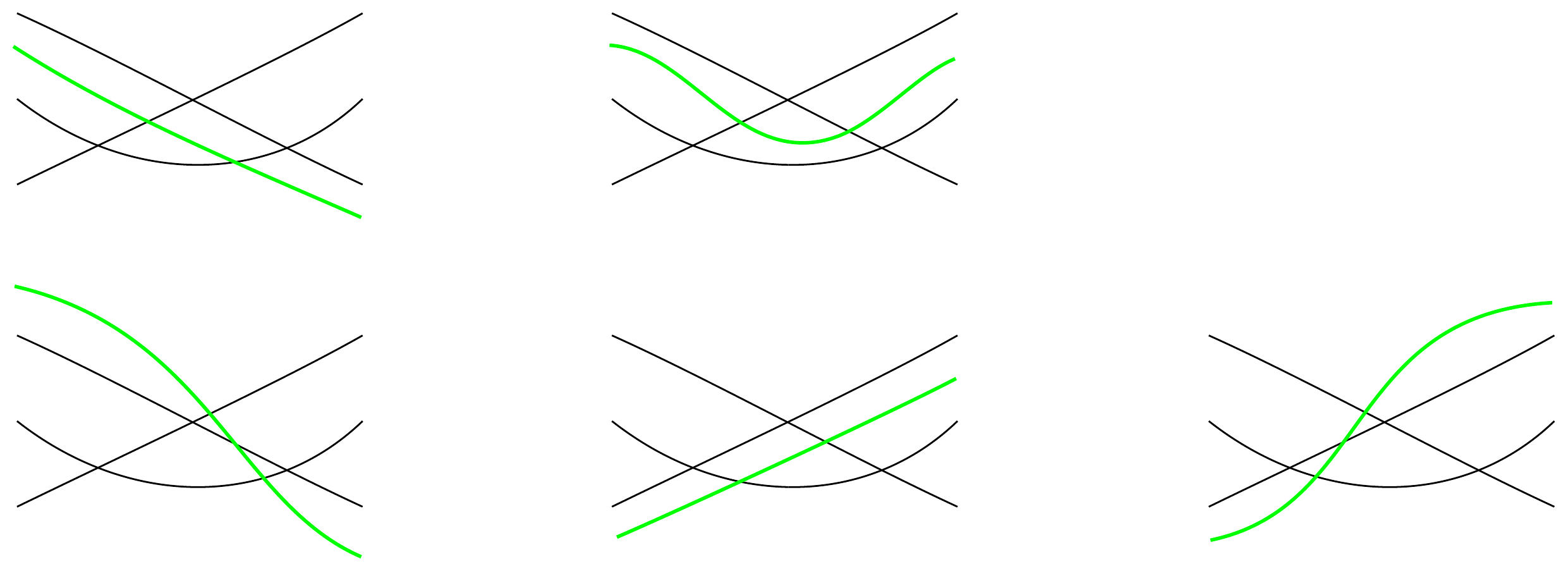}
\vspace{0.1in}

\caption{Possible cases for a strand $s$ to pass the triangular region formed by strands $i,j$ and $k$.}
\label{fig:tri}
\end{figure}

With the claim proven, we know that both $c_{j,k}$ and $c_{i,j}$ correspond to the braid generator $\sigma_{i_\ell}$ and $c_{i,k}$ corresponds to the braid generator $\sigma_{i_\ell-1}$.  Moreover, $B_{\pi}= B_1 \sigma_{i_\ell} B_2\sigma_{i_\ell-1}B_3 \sigma_{i_\ell}B_4$, where $B_2$ and $B_3$ are products of $\sigma_{t}$ with $t\neq i_\ell-1, i_\ell, i_\ell+1$.
Thus applying commutations we arrive at $B_1B_2\sigma_{i_\ell}\sigma_{i_\ell-1}\sigma_{i_\ell}B_3B_4$ showing that $B_\pi$ is not simple.
\end{proof}

Using the lemma, we can prove a uniqueness statement for simple braid words.
\begin{proposition}\label{prop:uni}
Any two simple Legendrian permutation braids with the same permutation $\pi\in S_n$ are related by
commutation relations $\sigma_i\sigma_j=\sigma_j\sigma_i$ for $|i-j| \geq 2$.
\end{proposition}
\begin{proof}
Suppose we have two simple Legendrian permutation braids $B_1$, $B_2$ that give the same permutation $\pi$. 
We can assume that they are different on the first braid letter, 
say $B_1=\sigma_iB_1'$, $B_2=\sigma_k B_2'$ and $i<k$.
By considering the location of the crossing $c_{k,k+1}$ in $B_1$, we will show that one can use the commutation relation 
to move $c_{k,k+1}$ to be the first crossing of $B_1$.  One can then repeat the argument on the new $B_1'$ and $B_2'$ to inductively complete the proof.

First note that $k\neq i+1$.  [If $k=i+1$, then  we would have $\pi(i+2)<\pi(i+1)<\pi(i)$.  Then, Lemma~\ref{lem:order} would fix an order for the crossings $c_{i,i+1}$ and $c_{k,k+1}$ in any simple permutation braid for $\pi$ so that they cannot both be the first crossing of a simple permutation braid.]
In order to show $c_{k,k+1}$ can be moved to the leftmost of $B_1$, we only need to prove that  
no strand intersects the strands $k$ or $k+1$ on the left of $c_{k,k+1}$ in $B_1$.
On one hand, since $c_{k,k+1}$ is the left most crossing of $B_2$, Lemma~\ref{lem:order} prevents any strand $s$ for $s<k$ from intersecting strand $k$ in any simple positive permutation braid for $\pi$.
On the other hand, as $c_{k,k+1}$ is a crossing in $B_1$, Lemma~\ref{lem:order} prevents any strand $s$ for $s>k+1$ to intersect strand $k+1$ on the left of $c_{k,k+1}$ in $B_1$.
These two facts permit us to move $c_{k,k+1}$ to the leftmost in $B_1$ using the commutation relations.
\end{proof}

\subsection{Standard form augmented braids}\label{sec:pro01}

\begin{definition}
For $n \geq 1$, a full augmented Legendrian $n$-braid $(B_\pi,\calC) \in  \mathit{Leg}^\rho_{\vec{\mu}}(J^1[0,1]; \mathbb{F})$ 
is called a {\bf standard form} if 
\begin{itemize}
\item  $B_{\pi}$ is a  simple Legendrian permutation braid for a permutation $\pi \in S_n$; and 
\item any handleslide of $\calC$ is located either:  (i) directly to the left of a crossing of $B_{\pi}$, connecting the crossing strands, or (ii) in a {\it properly ordered handleslide collection} (see Definition \ref{def:Properly}), $\mathcal{H}_1$, appearing to the right of all crossings of $B_{\pi}$.
\item Near $x=0$, each strand contains
one homology base point (possibly with coefficient $1$), and the top strand 
contains either one or zero spin base points located to the left of all homology base points.
\end{itemize}
For $n=0$, we say an augmented Legendrian $0$-tangle, $(\Lambda, \mathcal{C})$, is a {\bf standard form} if either 
\begin{itemize}
\item[(i)]  $\mbox{Char}\, \mathbb{F} \neq 2$ and either $\Lambda=\emptyset$ or $(\Lambda, \mathcal{C})$ is the augmented Legendrian unknot with spin invariant $1$ pictured in Figure \ref{fig:UnknotStd},
\item[(ii)] $\mbox{Char}\, \mathbb{F} = 2$, $\rho \neq 1$ and $\Lambda =\emptyset$, or 
\item[(iii)] $\mbox{Char}\, \mathbb{F} = 2$, $\rho = 1$ and $(\Lambda, \mathcal{C}) =(U, \mathcal{C}_b)$ is as in Figure \ref{fig:UnknotStd} for some $b \in \mathbb{F}$.
\end{itemize}
\end{definition}
See Figure \ref{fig:order} for an example of a standard form augmented $4$-braid.  Note that for $n\geq 1$, the MCFs of standard forms are in $SR$-form with respect to the normal ruling $\sigma_0$ of $B_\pi$ for which every strand is a fixed point strand; all crossings of $B_\pi$ are returns of $\sigma_0$.

\begin{figure}[!ht]
\labellist
\tiny
\pinlabel $-1$ [b] at 152 34
\pinlabel $-1$ [b] at 170 28
\pinlabel $b$ [b] at 26 32

\endlabellist

\quad 

\centerline{\includegraphics[scale=.6]{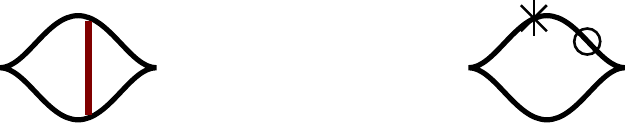}}

\caption{(left) The $1$-graded augmented Legendrian unknots $(U,\mathcal{C}_b)$, $b \in \mathbb{F}$.  (right) A Legendrian unknot with spin invariant $1$.}
\label{fig:UnknotStd}
\end{figure}

\begin{figure}[!ht]
\includegraphics[width=3in]{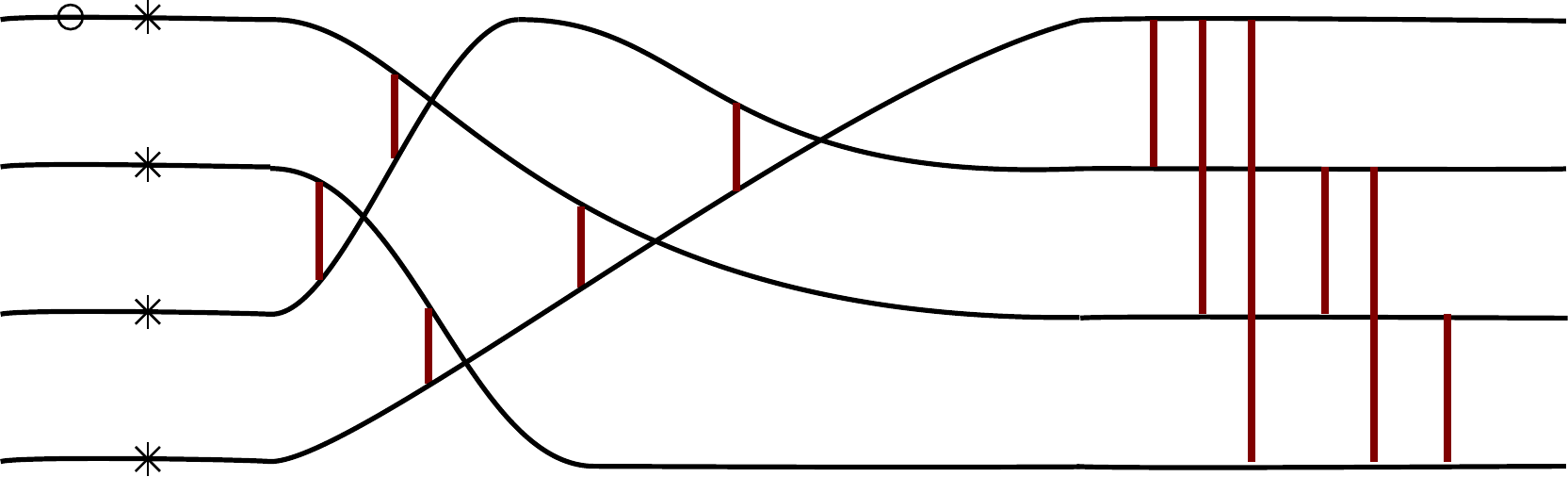}
\caption{An example of a standard form full augmented braid with permutation $\pi = (1\,3\,2\,4)$.  Coefficients of handleslides (resp.  basepoints) are not pictured, but may be any elements of $\mathbb{F}$ (resp. $\mathbb{F}^*$).}
\label{fig:order}
\end{figure}

Next, we compute the monodromy matrices of standard form augmented Legendrians.  In the statement, we use the perspective of Remark \ref{rmk:onemono} and view the graded monodromy matrix $\graded{M}_{B_{\pi}, \calC} \in GL(\graded{n},\mathbb{F})$ as a single matrix $M_{B_{\pi}, \calC} \in GL(n,\mathbb{F})$.  We reprise the notation $c_{i,j}$ for a crossing between strands $i$ and $j$ (as labeled at $x=0$) of a permutation braid.
\begin{proposition}\label{prop:mono}
For $n \geq 1$, the monodromy matrix of  a standard form full augmented $n$-braid $(B_{\pi}, \calC)$ satisfies 
\[
M_{B_{\pi}, \calC} = RT_\pi\Delta \Xi 
\]
where: 
\begin{itemize}
\item $R=I+\sum_{i<j} b_{i,j}E_{i,j}$ is upper-triangular with ones on the diagonal and $b_{i,j}\in \mathbb{F}$ is the coefficient of the $(i,j)$-handleslide in $\mathcal{H}_1$.
\item $T_\pi$ is obtained from the permutation matrix $Q_\pi=\sum_{j} E_{\pi(j),j}$ as
\begin{equation} \label{eq:Tpi}
T_{\pi}= Q_{\pi}+ \sum_{(i,j)\in X_{\pi}} h_{i,j} E_{\pi(i), j}
\end{equation}
where  $X_{\pi}=\{(i,j)\,|\, i<j \, \mbox{and} \,\, \pi(i)>\pi(j)\}$, i.e. the sum is over crossings of $B_\pi$, and   $h_{i,j} \in \mathbb{F}$ is the coefficient of the handleslide to the left of $c_{i,j}$.
\item $\Delta = \mathit{diag}(s_1, s_2, \ldots, s_n)$ where $s_i \in \mathbb{F}^*$ is the coefficient of the homology base point on strand $i$ with respect to the left-to-right co-orientation.  
\item $\Xi = \mathit{diag}(\pm1, 1, \ldots, 1)$ where the $(1,1)$-entry is $-1$ (resp. $1$) if the top sheet at $x=1$ has (resp. does not have) a spin basepoint.
\end{itemize}
\end{proposition}

\begin{proof}
Since the differentials of $\calC$ are all zero, the continuation map on homology is the same as the chain level map which can be computed as 
\[
f_{\tau_4} \circ f_{\tau_3} \circ f_{\tau_2} \circ f_{\tau_1}
\] 
where from left to right $\tau_1$ is the interval containing the spin base point (if it exists), $\tau_2$ contains the homology base points, $\tau_3$ contains the crossings of $B_\pi$ and their associated handleslides, and $\tau_4$ contains the collection of handleslides $\mathcal{H}_1$.   It is clear from the definition (see (\ref{eq:basepointmap})) that the matrices of $f_{\tau_1}$ and $f_{\tau_2}$ are $\Xi$ and $\Delta$ respectively, and Proposition \ref{prop:proper} (2) shows the matrix of $f_{\tau_4}$ is $R$.  

It remains to show that the matrix of $f_{\tau_3}$ is 
\[
T_{\pi}= Q_{\pi}+ \sum_{(i,j)\in X_\pi} h_{i,j} E_{\pi(i), j}.
\]
We use induction on the word length of the braid $B_\pi$.
When the word length is $0$, the braid is a trivial braid and the monodromy matrix is the identity matrix, which satisfies the formula.

For the inductive step, suppose that $B_{\pi}=B_{\pi'}\sigma_k$ and let $\calC'$ be the restriction of $\calC$ to the part of the interval $\tau_3$ containing $B_{\pi'}$.  The monodromy matrix of $(B_{\pi'}, \calC')$ has the stated form, $T_{\pi'}$, and
 the monodromy matrix for $(B_{\pi}, \calC)$ is
\[
M=Q_{(k\,k+1)}(I+h E_{k, k+1}) T_{\pi'}
\]
where $h:=h_{\pi^{-1}(k+1),\pi^{-1}(k)}$ is the coefficient of the handleslide before $\sigma_k$, or zero if it does not exist.

$$\begin{array}{rcl}
M&=&Q_{(k\,k+1)}(I+h E_{k, k+1}) (Q_{\pi'}+\displaystyle{\sum_{(i,j)\in X_{\pi'}} h_{i,j} E_{\pi'(i),j})}\\
&=&(Q_{(k\, k+1)}+ h E_{k+1, k+1})(Q_{\pi'}+\displaystyle{\sum_{(i,j)\in X_{\pi'}} h_{i,j} E_{\pi'(i),j})}.
\end{array}$$
Note from Lemma \ref{lem:order} that  there is no crossing between  
the strand $(\pi')^{-1}(k+1)$ and any strand $s$ for $s>(\pi')^{-1}(k+1)$.
Thus for $(i,j)\in X_{\pi'}$, the term $\pi'(i)$ cannot  
be $k+1$, so we have  
$$\begin{array}{rcl}
M&=&Q_{\pi}+ h E_{k+1, (\pi')^{-1}(k+1)} + \displaystyle{\sum_{(i,j)\in X_{\pi'}} h_{i,j} E_{\pi(i),j}} \\
&=& Q_{\pi}+  h E_{\pi(\pi^{-1}(k+1)), \pi^{-1}(k)} +\displaystyle{\sum_{(i,j)\in X_{\pi'}} h_{i,j} E_{\pi(i),j}}\\
&=&Q_{\pi}+\displaystyle{ \sum_{(i,j)\in X_{\pi}} h_{i,j} E_{\pi(i),j}}.
\end{array}$$
This completes the proof.
\end{proof}

\begin{remark} \label{rem:Tmatrix}
\begin{itemize}
\item[(i)]  The matrices $T_\pi$ have the following property:  Besides the $1$'s from the permutation matrix $Q_{\pi}$, all additional non-zero entries appear in positions $(i,j)$ that are both to the right of the $1$ from $Q_\pi$ in row $i$ and below the $1$ from $Q_\pi$ in column $j$.

\item[(ii)] If $B_{\pi}$ is allowed to be a Legendrian permutation braid that is {\it not simple}, then the monodromy matrix is not in the form shown in Proposition \ref{prop:mono}.
For example, Figure \ref{fig:simple} shows two full augmented braids with the same monodromy matrices.
$$\begin{pmatrix}0& 0& 1\\
0& 1& 1\\
1& 1& 1\end{pmatrix}$$
The  braid in part $(b)$ is simple and the monodromy matrix satisfies the formula in
Proposition \ref{prop:mono} while the braid in part (a) is not simple and the monodromy matrix does not have the form of $T_\pi$.
\end{itemize}
\end{remark}
\begin{figure}[!ht]
\labellist
\small
\pinlabel $(a)$ at 90 -20
\pinlabel $(b)$ at 380 -20
\endlabellist
\includegraphics[width=3in]{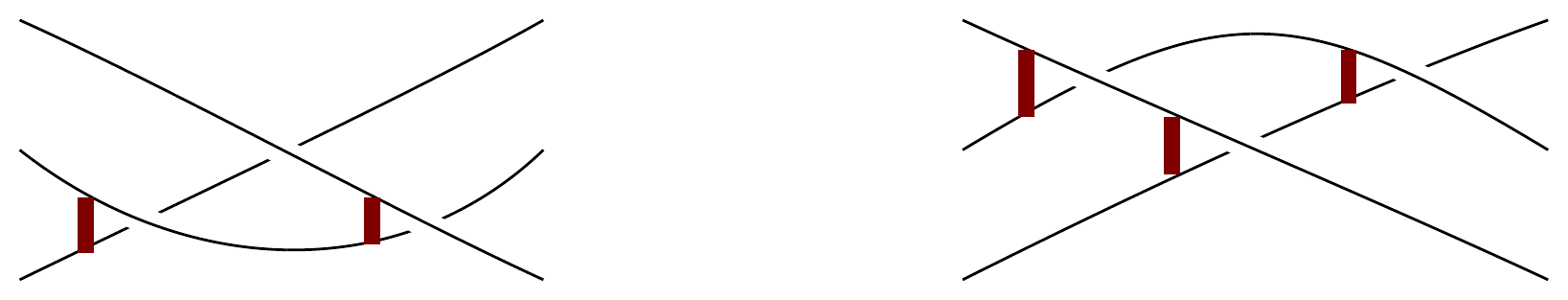}
\vspace{0.2in}

\caption{Two full augmented braids with the same monodromy matrix.  In both cases, all handleslide coefficients are $1$.}
\label{fig:simple}
\end{figure}

Next, we prove a uniqueness statement for matrix decompositions as in Proposition \ref{prop:mono} without the spin term, $\Xi$.  This is related to the Bruhat decomposition.  

\begin{lemma}\label{lem:uni}
Let $R_i, T_i, \Delta_i \in GL(n, \mathbb{F})$ for $i=1,2$ be such that
\begin{itemize}
\item $R_i$ is upper triangular with $1$'s on the diagonal;
\item $T_i$ has the form (\ref{eq:Tpi}) for some permutation $\pi_i \in S_n$; and
\item $\Delta_i$ is diagonal.
\end{itemize}
If $R_1T_1\Delta_1 = R_2T_2\Delta_2$, then $R_1=R_2$, $T_1=T_2$ (in particular, $\pi_1=\pi_2$), and $\Delta_1=\Delta_2$.
\end{lemma}

\begin{proof}
 Assume $R_1 T_1 \Delta_1= R_2 T_2 \Delta_2$, and write  
\[
T_1 = R T_2 \Delta  \quad \mbox{where} \quad R = R_1^{-1}R_2 \quad \mbox{and} \quad \Delta = \Delta_2 \Delta_1^{-1}.
\]

\medskip

\noindent {\bf Case 1.}   $\pi_1=\pi_2$.  Write $\pi = \pi_1=\pi_2$, and assume for contradiction that $R \neq I$.  Choose $i<j$ such that $r_{i,j} \neq 0$ and $r_{i,j'}=0$ for all $j<j'\leq n$ where $R=(r_{i,j})$.  Then, the $(i, \pi^{-1}(j))$-entry of $RT_2 \Delta$ is non-zero.  [This is because the $(j,\pi^{-1}(j))$-entry of $T_2 \Delta$ is the top most non-zero entry in $\mbox{col}_{\pi^{-1}(j)}(T_2 \Delta)$; see Remark \ref{rem:Tmatrix}.]  This contradicts that the $(i, \pi^{-1}(j))$-entry of $T_1$ is zero.    
Thus, $R =I$, and considering the $(\pi(j),j)$-entries in the identity $T_1=T_2\Delta$ shows $\Delta=I$ also.  

\medskip

\noindent {\bf Case 2.}  $\pi_1 \neq \pi_2$.  This case leads to a contradiction.  Let $k$ be chosen such that $\pi_1^{-1}(k) \neq \pi_2^{-1}(k)$.  Without loss of generality we may assume $\pi_1^{-1}(k) > \pi_2^{-1}(k)$, i.e. the $1$ in $\mbox{row}_k Q_{\pi_1}$ is to the right of the $1$ in $\mbox{row}_k Q_{\pi_2}$.  Now, we have
\begin{equation} \label{eq:row}
\mbox{row}_k T_1 = \mbox{row}_k (T_2\Delta) +  \sum_{k<i} r_{k,i}\mbox{row}_i (T_2\Delta).
\end{equation}
Note that the set of $i$ such that $k<i$ and $r_{k,i} \neq 0$ is non-empty since $\mbox{row}_k T_1 \neq \mbox{row}_k (T_2 \Delta)$ as the first non-zero entries in these rows are located in the same position as in $Q_{\pi_1}$ and $Q_{\pi_2}$ respectively.
Among these values of $i$, let $i'$ be the one having $\pi_2^{-1}(i)$ smallest, i.e. such that the first non-zero entry of $\mbox{row}_i(T_2\Delta)$ is farthest to the left.  Note that we must have $\pi^{-1}_2(i')< \pi_2^{-1}(k)$ as well.  [Otherwise, all entries below the (non-zero) $(k,\pi_2^{-1}(k))$-entry of $T_2 \Delta$ and having non-zero $r_{k,i}$ coefficient in (\ref{eq:row}) would be $0$.  This would show the $(k,\pi_2^{-1}(k))$-entry of $T_1$ is non-zero, contradicting that $\pi_2^{-1}(k) < \pi_1^{-1}(k)$.] Thus, the $\pi_2^{-1}(i')$ entry of the row vector on the right side of (\ref{eq:row}) is non-zero, contradicting that the corresponding entry on the left vanishes since $\pi_2^{-1}(i') < \pi_1^{-1}(k)$. 
\end{proof}

\begin{corollary} \label{cor:standard}  Suppose either $n \geq 1$ or $n=0$ and $\rho \neq 1$.  Two standard form $\rho$-graded augmented Legendrian $n$-tangles with the same spin invariant (when defined) and monodromy matrix are cobordant.
\end{corollary}
\begin{proof}
In the special case of $0$-tangles, this is obvious from the definition.  Given two such standard form augmented $n$-braids $(B_{\pi_1}, \calC_1)$ and $(B_{\pi_2}, \calC_2)$ with $n\geq 1$, 
computing their monodromy matrices as in Proposition \ref{prop:mono} we have
\[
R_1T_1(\Delta_1 \Xi_1) = R_2T_2(\Delta_2 \Xi_2). 
\]
	Applying Lemma \ref{lem:uni} we see that $R_1=R_2, T_1=T_2$ (in particular, $\pi_1=\pi_2$), and $\Delta_1\Xi_1=\Delta_2\Xi_2$.  Moreover, since $\pi_1=\pi_2$, the closures of $B_{\pi_1}$ and $B_{\pi_2}$ have the same number of components.  Therefore, the equality of spin invariants, $\xi(B_{\pi_1}, \calC_1) = \xi(B_{\pi_2}, \calC_2)$, implies that $\calC_1$ and $\calC_2$ have the same number of spin basepoints (this number being $1$ or $0$).  Thus, $\Xi_1 = \Xi_2$, so $\Delta_1=\Delta_2$ as well.  
		As the entries of the individual matrices $R$, $T_\pi$, $\Delta$, and $\Xi$ in Proposition \ref{prop:mono} together uniquely determine a standard form $(B_{\pi}, \calC)$ (up to applying commutation moves to braid crossings, in particular, up to cobordism), this completes the proof.
\end{proof}

\subsection{Constructing cobordisms to standard form}

The final ingredient before proving Theorem \ref{thm:J101} is 
the following:

\begin{proposition} \label{prop:standardform}
Any full augmented Legendrian $n$-tangle is cobordant to one in standard form.
\end{proposition}

We establish Proposition \ref{prop:standardform} in two steps.  First, Proposition \ref{prop:cusp} handles the case $n=0$, and shows that when $n\geq 1$  we can remove cusps by an augmented cobordism;  this transforms a given tangle with $n\geq1$ into an augmented Legendrian braid.  
Then, in Proposition \ref{prop:standard} we show that any augmented Legendrian braid is cobordant to one in standard form.  
Clearly the combination of these two statements proves Proposition \ref{prop:standardform}.

\begin{proposition} \label{prop:cusp}  
Assume that either $\rho$ is even, or $\Char \mathbb{F}=2$.
\begin{enumerate}
\item When $n \geq 1$, any {\it full} augmented Legendrian $n$-tangle in $J^1[0,1]$ is cobordant to one without cusps, i.e. to an augmented Legendrian braid. 
\item Any augmented $0$-tangle is cobordant to a standard form $0$-tangle.
\end{enumerate}
\end{proposition}

\begin{proof}  
We apply the inductive procedure from Section 7.3 of \cite{PanRu2}, used there for constructing cobordisms from augmented Legendrians in $\R^3$ over $\mathbb{F}_2$ to the empty set, and address the important new issues arising from (i) the setting of $n$-tangles instead of links in $\R^3$ and (ii) the coefficients in a general field $\mathbb{F}$ together with the presence of homology and spin basepoints.

The statement is established by induction on the number of cusps, $c$, with the $c=0$ case being tautological.  
Let $(\Lambda, \mathcal{C}) \subset J^1[0,1]$ be an $n$-tangle with $n \geq 0$ such that $\Lambda$ has at least one cusp, and suppose that the (1) and (2) are established for Legendrians with fewer cusps. Let $l$ be the left cusp of $\Lambda$ with largest $x$-coordinate.  We can write $\Lambda$ as a product of tangles in the form
\begin{equation} \label{eq:Xl}
\Lambda = X l \sigma_{k-1}\sigma_{k-2}\cdots \sigma_{k-s}\sigma_{k+1}\sigma_{k+2} \cdots \sigma_{k+t} Y
\end{equation}
where the sheets of $l$ are numbered $k$ and $k+1$ and $s,t \geq 0$.  (Initially, we can take $s=t=0$.)  Then, we apply induction on the word length $|Y|$, i.e. the number of crossings and right cusps in $Y$, to establish (1) and (2) for Legendrians of the form (\ref{eq:Xl}) with the current value of $c$.  In our $J^1[0,1]$ setting, it is crucial to note that the base case $|Y|=0$ is vacuous since $Y$ contains at least one right cusp.  
[This is because the two sheets of $l$ cancel in the homology of $(C_x,d_x)$, so immediately to the right of $l$ the number of sheets of $\Lambda$ is larger than the total dimension of $H^*(C_x,d_x)\cong H^*(C_1,d_1)$ which is $n$ since $\mathcal{C}$ is {\it full}.  Thus, there must be a right cusp between $l$ and $x=1$.]

The inductive (on $|Y|$) step is carried out by writing $Y= zY'$ and considering cases (numbered 1-10 in \cite{PanRu2}) based on the type (right cusp or crossing) and location of the singularity of $z$.  In all cases, it is shown how to construct an augmented cobordism from $(\Lambda, \mathcal{C})$ to an augmented Legendrian $(\Lambda', \mathcal{C}')$ such that $\Lambda'$ either can be written in the form (\ref{eq:Xl}) with $Y$ replaced by $Y'$ or has fewer left cusps than $\Lambda$.  For many of the cases (specifically for Cases 1, 4, 5b, 6, 7, 9, 10b) the underlying Legendrian cobordism is a Legendrian isotopy, and in our setting of $\mathbb{F}$ coefficients we apply Proposition \ref{prop:Rmove} to extend $\mathcal{C}$ over the cobordism.  
Cases 3 and 10a are possibilities for $z$ that cannot actually occur since $\Lambda$ would have a zig-zag stabilization, and hence could not have any MCF.  This point remains valid for MCFs over $\mathbb{F}$.  

In the Cases 2, 5c, 8, and 10c the first step in \cite{PanRu2} is to apply an equivalence to put $\mathcal{C}$ into $SR$-form with respect to a normal ruling for $\Lambda$.  In our setting, we first move all homology and spin basepoints away from the $l \sigma_{k-1}\sigma_{k-2}\cdots \sigma_{k-s}\sigma_{k+1}\sigma_{k+2} \cdots \sigma_{k+t} z$ part of $\Lambda$, and then we apply Proposition \ref{prop:SRalgorithm} to arrange the $SR$-form.  We should keep in mind that the normal ruling may now be {\it generalized}.  The construction of the cobordism then depends on whether the crossing $\sigma_{k-1}$ that appears next to $l$ is a departure or a switch.  [Note:  Even with a {\it generalized} normal ruling it is still impossible for $\sigma_{k-1}$ to be a return since the cusp strands are not fixed point strands.]  If the crossing is a departure, the cusp tangency move pictured in Figure \ref{fig:CuspTang}  is applied, noting that $SR$-form MCFs have no handleslides at departures.  If the crossing is a switch then a cobordism involving the Pinch move (C2) is applied to remove the switch.  The modification of this cobordism to $\mathbb{F}$ coefficients, allowing for the first handleslide coefficient at the switch to be arbitrary, makes use of an added homology curve and is pictured in Figure \ref{fig:SwCobrdism}.  Note that because the differentials of an $SR$-form MCF are in Barannikov normal form, the pictured arrow at the location of the pinch move, $x=x_0$, is the only matrix coefficient of $d_{x_0}$ involving the sheets of the pinch move.  Thus, the pinch move (C2) $\leftarrow$ can be applied once the base points are introduced to make the coefficient in the differential become $1$.

\begin{figure}[!ht]
\labellist
\tiny

\endlabellist
\centerline{\includegraphics[scale=.6]{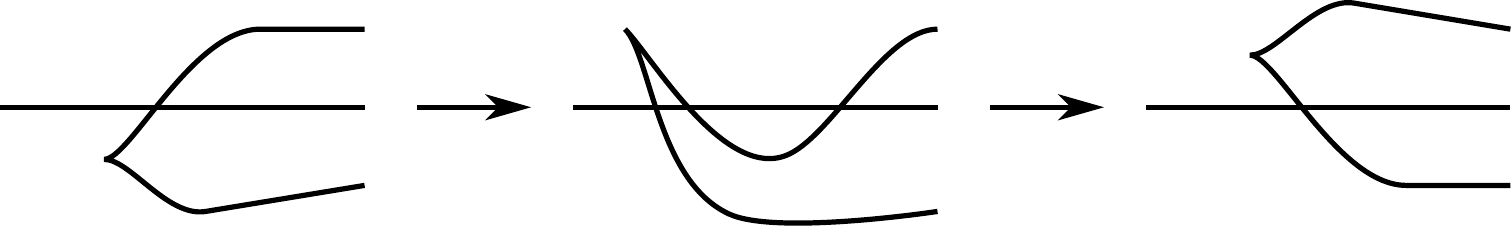}}

\caption{The Cusp Tangency Move is a combination of (R2) and the Clasp Move (C3). (No handleslides are present.)}
\label{fig:CuspTang}
\end{figure}

\begin{figure}[!ht]
\labellist
\tiny
\pinlabel $r$ [b] at 30 88
\pinlabel $-r^{-1}$ [b] at 66 88
\pinlabel $1$ [t] at 14 36
\pinlabel $r$ [l] at 108 26

\pinlabel $r$ [b] at  202 88
\pinlabel $-r^{-1}$ [b] at 238 88
\pinlabel $1$ [t] at 188 36
\pinlabel $r^{-1}$ [b] at 262 56
\pinlabel $r$ [b] at 294 56 
\pinlabel $1$ [l] at 280 26

\pinlabel $r$ [b] at 376 88
\pinlabel $-r^{-1}$ [b] at 404 82
\pinlabel $1$ [t] at 358 36
\pinlabel $r^{-1}$ [bl] at 420 58
\pinlabel $r$ [b] at 476 68

\pinlabel $r^{-1}$ [t] at 566 86
\pinlabel $r$ [b] at 628 68

\endlabellist
\centerline{\includegraphics[scale=.6]{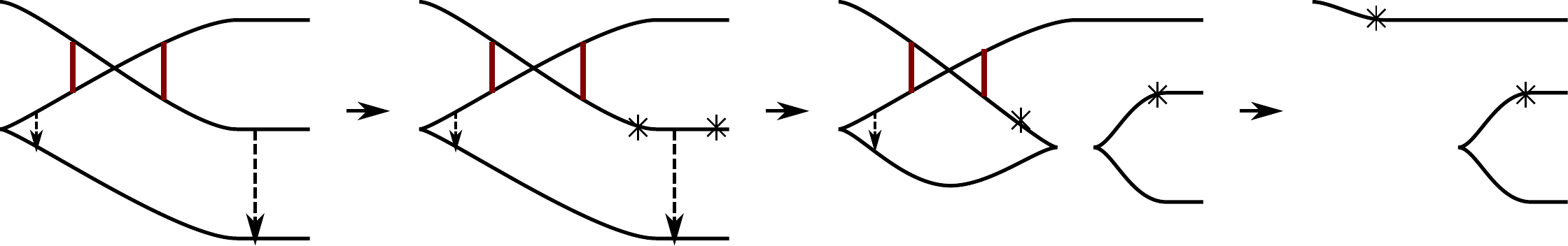}}

\caption{Cobordism to remove a switch next to a left cusp.  The coefficient $r$ is non-zero.  There may be one additional unpictured handleslide at the switch with upper endpoint on the lower sheet of the cusp, if the switch is in the first configuration pictured in Figure \ref{fig:SR}.}
\label{fig:SwCobrdism}
\end{figure}

In the last remaining Case 5A the left cusp $l$ is part of a Legendrian unknot, $U$, disjoint from the rest of $\Lambda$.  The new issue here is that $U$ may contain base points and, if $\rho =1$, a handleslide in its interior.  If $\Lambda$ has another component, $\Lambda'$, then we use the cobordism pictured in Figure \ref{fig:UnknotSpin} to absorb $U$ into $\Lambda'$.  This decreases the number of cusps of $\Lambda$ by $1$ so that the  outer induction on $c$ applies.  In the remaining case where $\Lambda = U$ we must have $n=0$.  Using (F1)-(F4) we move all of the homology and spin basepoints together so that we have at most $1$ base point of each type.
The coefficient of the homology basepoint is necessarily $(-1)^{s_0}$ where $s_0$ is the number of spin base points. [This is necessary in order for (A2) to be satisfied at both the left and right cusp of $U$.]  If $s_0=1$ and $\Lambda=U$ we are done, as $(\Lambda, \mathcal{C})$ is now the standard form pictured on the right side of Figure \ref{fig:UnknotStd}.  [There is no handleslide since the presence of spin base points implies $\Char \mathbb{F} \neq 2$, so that we have required that $\rho$ is even.]  
If $s_0=0$, then the homology base point has coefficient $1$ and can be erased, see Remark \ref{rem:F1}.  Then, if $U$ has no handleslide, we apply (C1) $\leftarrow$ to obtain a cobordism to $\emptyset$.  If there is a handleslide, then $\rho=1$ 
and  $(\Lambda, \mathcal{C})$ has the standard form pictured on the left side of Figure \ref{fig:UnknotStd}.

\begin{figure}[!ht]
\labellist
\tiny

\endlabellist

\quad 

\centerline{\includegraphics[scale=.6]{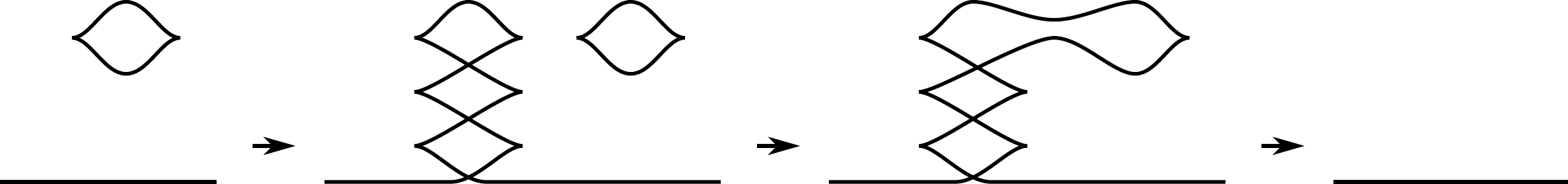}}

\quad

\caption{Cobordism to remove an unknot, $U$, in the presence of at least one additional component, $\Lambda'$.  Note that $U$ may contain basepoints or (if $\rho =1$) handleslides.  The 1st and 3rd arrows are augmented cobordisms constructed from Legendrian isotopies via Proposition \ref{prop:Rmove}.    The  (R1) moves are performed so that the Maslov potentials agree when applying the (C2) $\rightarrow$ move. 
}

\label{fig:UnknotSpin}
\end{figure}

\end{proof}

\begin{proposition}\label{prop:standard}
Any full augmented Legendrian $n$-braid $(B_0, \calC_0)$ in $J^1[0,1]$ with $n \geq 1$ is cobordant to a standard form $(B_{\pi}, \calC)$.
\end{proposition}

\begin{proof}  
{\bf Step 1.}  Construct a cobordism from $(B_0,\calC_0)$ to $(B_\pi,\calC_1)$ where $B_\pi$ is a simple permutation braid.  

To ease considerations we begin by using (F1)-(F4) to move all basepoints (homology and spin) to 
the far left side of $B_0$.  According to Corollary \ref{cor:eye}, if $B_0$ is not a positive permutation braid we can apply a Legendrian isotopy to $B_0$ to arrange for $\sigma_i^2$ to appear in its braid word.  As usual Proposition \ref{prop:Rmove} makes this Legendrian isotopy into an augmented cobordism.
One can then remove either one or two crossings from the $\sigma_i^2$ part of $B$ via an augmented Legendrian cobordism:  If there is a handleslide in the middle of the two crossings and connecting the crossing strands with coefficient $r \in \mathbb{F}^*$, then we apply the $D^-_4$ cobordism from Proposition \ref{prop:D4minus} reducing the number of crossings by $1$.  If not, a clasp move (C3) $\leftarrow$ applies to remove both crossings.
Repeating this procedure inductively, we arrive at $(B_1,\mathcal{C}_1)$ with $B_1$ a permutation braid, and then applying Proposition \ref{prop:ex} we may arrange that $B_1 = B_\pi$ is a simple permutation braid.

\medskip

\noindent {\bf Step 2.}  Move the handleslides of $\calC_1$ into standard form position.

To do this we simply apply Proposition \ref{prop:SRalgorithm}.  Note that since $(B_\pi, \calC_1)$ is full the differential at $x=0$ is $d_0=0$ which is in Barannikov normal form.  Thus,  $\calC_1$ is equivalent to an $SR$-form, $\calC_2$, without any properly ordered handleslide collection at the $x=0$ side of the braid.  

\medskip  

\noindent {\bf Step 3.}  Move the basepoints of $(B_\pi, \calC_2)$ into standard form position.

For each strand of $B_\pi$ we move the homological and spin basepoints to the far left and collect them together so that each strand has at most $1$ basepoint of each type.  To arrange the standard form, we need to move all of the spin basepoints to the top strand.  This is done using a sequence of cobordisms as in Figure \ref{fig:MovingSpin}.  

\end{proof}

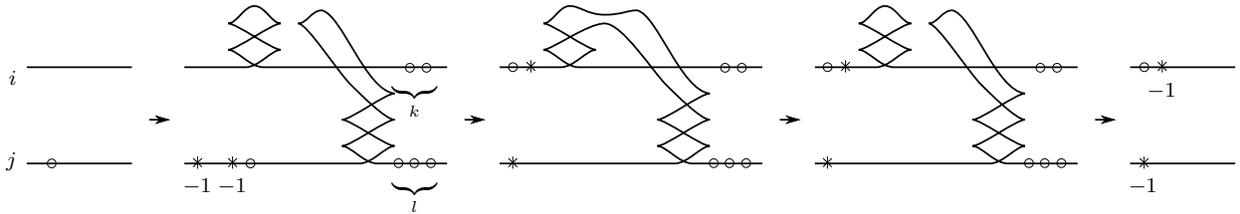
\begin{figure}[!ht]
\labellist
\tiny
\pinlabel $j$ [r] at -2 6
\pinlabel $i$ [r] at -2 78
\pinlabel $-1$ [t] at 146 -2
\pinlabel $-1$ [t] at 176 -2
\pinlabel $-1$ [t] at 956 -2
\pinlabel $-1$ [t] at 972 80
\pinlabel $\underbrace{\quad\quad}_{k}$ [t] at 332 80
\pinlabel $\underbrace{\quad\quad}_{l}$ [t] at 332 -2
\endlabellist

\quad 

\centerline{\includegraphics[scale=.45]{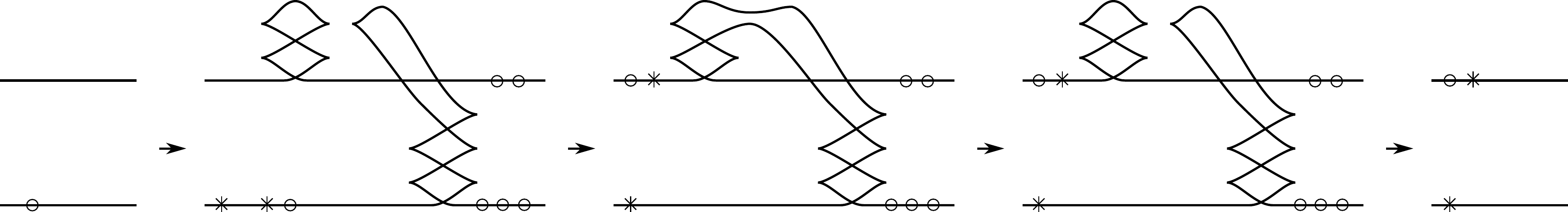}}

\quad

\caption{The figure shows how to move a spin basepoint near $x=0$ up to the next adjacent sheet, introducing in the process an additional homology basepoint with coefficient $-1$ on both sheets.  The Maslov potential on the two sheets is $i$ and $j$, and $k,l>0$ should be chosen so that $i+k = j+l$ mod $\rho$.
There are handleslides (unpictured) and new spin points arising from the $k+l$ (R1) moves that are performed.  When the pair of basepoints, one spin and one homology, are moved together the coefficients of differentials and handleslides remain unchanged.  Thus, the (C2) and (R1) moves can all be reversed at the 4-th arrow removing the extra $k+l$ spin points that were created at the 1-st arrow.    
}

\label{fig:MovingSpin}
\end{figure}

We now complete the proof of Theorem \ref{thm:J101}.

\begin{proof}[Proof of Theorem \ref{thm:J101}]  The statement (3) about the case when $n=0$ and $\rho=1$ is Proposition \ref{prop:cusp} (2).  For (1) and (2) we need to show $\Phi$ is an isomorphism.

\medskip

\noindent {\bf $\Phi$ is bijective:}  
The surjectivity of $\Phi$
was established in Proposition \ref{prop:surj}.  To establish injectivity, assume that two full augmented $n$-tangles $(\Lambda_1, \calC_1)$ and $(\Lambda_2, \calC_2)$ have the same monodromy matrix $M_1=M_2$ and (if $\mbox{Char}\,\mathbb{F} \neq 2$) the same spin invariant.  By Proposition \ref{prop:standardform}, we may assume the $(\Lambda_i, \calC_i)$ are in standard form, and then Corollary \ref{cor:standard} applies to show that $[(\Lambda_1,\calC_1)] = [(\Lambda_2,\calC_2)]$.

\medskip

\noindent {\bf $\Phi$ is a homomorphism:}   Let $[(\Lambda_1,\calC_1)],[(\Lambda_2,\calC_2)] \in \mathit{Cob}^\rho_{\vec{\mu}}(J^1[0,1]; \mathbb{F})$.  From Proposition \ref{prop:continuation} we get that $\graded{M}_{\Lambda_1*\Lambda_2, \mathcal{C}_1*\mathcal{C}_2} = \graded{M}_{\Lambda_2,\calC_2} \graded{M}_{\Lambda_1,\calC_1}$.  In the case that $\mbox{Char} \, \mathbb{F} \neq 2$, we also need to show that the spin invariant $\xi$ defines a homomorphism to $\Z/2$.  For the case $n=0$ this is clear as the operation on the cobordism group is just disjoint union.  When $n \geq 1$, using Proposition \ref{prop:cusp}, we may assume that $\Lambda_1$ and $\Lambda_2$ are Legendrian braids.  Clearly the number of spin basepoints of $(\Lambda_1 *\Lambda_2, \mathcal{C}_1*\mathcal{C}_2)$ is the sum of that of the factors.  As there are no right cusps, we just need to show that 
\[
c+n = (c_1+n) +(c_2 +n) \quad \mbox{mod $2$}
\]
where $c$, $c_1,$ and $c_2$ are  number of components of the closure of $\Lambda_1 *\Lambda_2$, $\Lambda_1$ and $\Lambda_2$.  The number of components, $c(b)$, of the closure of a braid $b \in B_n$ is the number of cycles in the cycle decomposition of the image of $b$ under the standard homomorphism $\theta:B_n \rightarrow S_n$.  Moreover, one computes that $c(b)+n$ mod $2$ is just the parity of $\theta(b)$, and hence is additive under concatenation.  [If $\theta(b)$ factors into disjoint cycles of length $\ell_1, \ldots, \ell_{c(b)}$, then it factors into $(\ell_1-1) + \cdots +(\ell_{c(b)}-1)$ transpositions.  Thus, 
\[
\mbox{parity of $\theta(b)$} = (\ell_1-1) + \cdots +(\ell_{c(b)}-1) = (\ell_1 +\cdots + \ell_{c(b)}) - c(b) = n+ c(b) \, \mbox{ (mod $2$).  ]}  
\]
\end{proof}

\section{Computation of cobordism classes in $J^1S^1$}  \label{sec:J1S1}

In this short section we obtain the following cobordism classification for augmented Legendrians in $J^1S^1$.

\begin{theorem}  \label{thm:S1main}  Assume that either $\rho$ is even, or $\mathit{Char}\,\mathbb{F} =2$ and $\rho \neq 1$.  Then,
two $\rho$-graded augmented Legendrians in $J^1S^1$ over $\mathbb{F}$ are cobordant if and only if 
\begin{itemize}
\item they have the same spin invariant (if $\mathit{Char}\, \mathbb{F} \neq 2$),
\item their fiber cohomologies have the same graded dimension, $\graded{n}:\Z/\rho \rightarrow \Z_{\geq 0}$, and
\item their monodromy matrices are conjugate in $GL(\graded{n}, \mathbb{F})$.
\end{itemize}
Moreover, with $\mathit{Cl}(GL(\graded{n}, \mathbb{F}))$ denoting conjugacy classes in $GL(\graded{n}, \mathbb{F})$ we have bijections:
\begin{enumerate}
\item When $\mbox{Char}\,\mathbb{F} \neq 2$ ($\rho$ is even),
\[
\begin{array}{ccc} \displaystyle \Psi:\mathit{Cob}^\rho(J^1S^1;\mathbb{F}) & \stackrel{\cong}{\rightarrow} & \displaystyle \left(\bigsqcup_{\graded{n}} \mathit{Cl}(GL(\graded{n}, \mathbb{F})) \right)\times \Z/2  
\\  
  \mbox{$[(\Lambda, \mathcal{C})]$} & \mapsto & \left([\graded{M}_{\Lambda, \mathcal{C}}], \xi(\Lambda, \mathcal{C}) \right). 
	\end{array}
\]
\item When $\mbox{Char}\,\mathbb{F} = 2$ and $\rho\neq 1$,
\[
\begin{array}{ccc} \displaystyle \Psi:\mathit{Cob}^\rho(J^1S^1;\mathbb{F}) & \stackrel{\cong}{\rightarrow} & \displaystyle \bigsqcup_{\graded{n}} \mathit{Cl}(GL(\graded{n}, \mathbb{F}))  
\\  
  \mbox{$[(\Lambda, \mathcal{C})]$} & \mapsto & [\graded{M}_{\Lambda, \mathcal{C}}]. 
	\end{array}
\]
\end{enumerate}
\end{theorem}

The proof will be given after the following lemma.  
Recall the closure operation from Section \ref{sec:fullaug}.

\begin{lemma} \label{lemma:diff}
Any augmented Legendrian in $(\Lambda, \calC) \subset J^1S^1$ is cobordant to the closure of a {\it full} augmented $n$-tangle in $J^1[0,1]$ 
 	whose boundary Maslov potential, $\vec{\mu}= (\mu_1, \ldots, \mu_n)$, is non-decreasing from top to bottom.
\end{lemma}

\begin{proof}
Applying Proposition \ref{prop:SRalgorithm} we can arrange that $(\Lambda, \calC)$ is in $SR$-form.  Then, we can choose some $x=x_0$ where $d_{x_0}$ is in Barannikov normal form.
That is, the sheets of $\Lambda$ at $x=x_0$ are divided into some number of pairs, $\{S_i,S_j\}$ satisfying $\partial S_j = r S_i$ for $i<j$ and $r \in \mathbb{F}^*$, and some number of fixed point sheets, $S_k$ with $\partial S_k=0$. 
Now, for a given pair $\{S_i,S_j\}$ 
we apply the cobordism in Figure \ref{fig:S1toFull}.  We repeat this procedure until only the fixed point strands are left.  At this point the differential $d_{x_0}$ vanishes, so that cutting along $x=x_0$ realizes $(\Lambda, \calC)$ as the closure of a full $n$-tangle.  Moreover, since $d_{x_0}=0$ we can apply clasp moves at $x=x_0$ to arrange that $\vec{\mu}$ is non-decreasing.

\end{proof}

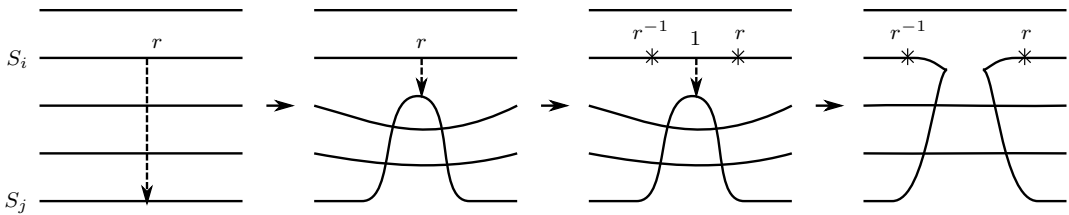
\begin{figure}[!ht]
\labellist
\tiny
\pinlabel $S_i$ [r] at -2 92
\pinlabel $S_j$ [r] at -2 2
\pinlabel $r$ [b] at 74 94
\pinlabel $r$ [b] at 240 94
\pinlabel $r^{-1}$ [b] at 384 100
\pinlabel $r$ [b] at 440 100
\pinlabel $1$ [b] at 412 96
\pinlabel $r^{-1}$ [b] at 546 100
\pinlabel $r$ [b] at 620 100
\endlabellist
\includegraphics[scale=.6]{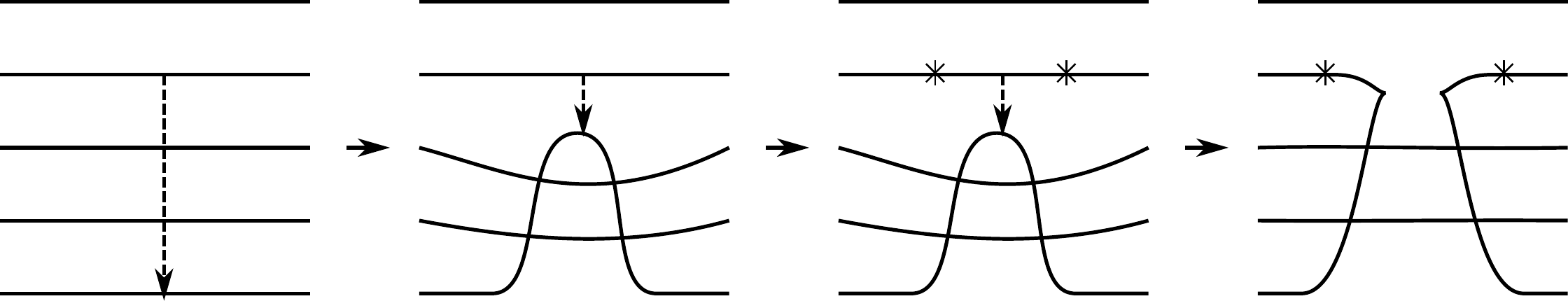}
\caption{The differential satisfies $\langle d_{x_0} S_j, S_i\rangle = r \in \mathbb{F}^*$.  Coefficients of basepoints are indicated with respect to the left to right co-orientation.}
\label{fig:S1toFull}
\end{figure}

\begin{proof}[Proof of Theorem \ref{thm:S1main}]  Surjectivity of $\Psi$ follows from Theorem \ref{thm:J101} since we can take the closure of an $n$-tangle with arbitrary monodromy matrix and spin invariant.  To prove injectivity, suppose that $(\Lambda_1, \mathcal{C}_1)$ and $(\Lambda_2, \mathcal{C}_2)$ have equal spin invariants (if $\mbox{Char}\,\mathbb{F} \neq 2$) and have monodromy matrices  $\graded{M}_{\Lambda_i, \mathcal{C}_i} \in GL(\graded{n}, \mathbb{F})$ that are conjugate,
\[
\graded{M}_{\Lambda_1, \mathcal{C}_1} = P \graded{M}_{\Lambda_2, \mathcal{C}_2} P^{-1}.
\]
Using Lemma \ref{lemma:diff} we can assume that each $(\Lambda_i, \mathcal{C}_i)$ is the closure of a full augmented $n$-tangle in $J^1[0,1]$ that we will denote with the same notation, and that the $\Lambda_i$ have the non-decreasing boundary Maslov potential, $\vec{\mu}$, determined by $\mathbf{n}$.  From Theorem \ref{thm:J101}, there exist two further full augmented $n$-tangles, $(\Lambda_P, \mathcal{C}_P)$ and $(\Lambda_{P^{-1}}, \mathcal{C}_{P^{-1}})$ in $\mathit{Leg}^\rho_{\vec{\mu}}(J^1[0,1]; \mathbb{F})$ with vanishing spin invariants (if $\mbox{Char}\,\mathbb{F} \neq 2$) 
and monodromy matrices $P$ and $P^{-1}$ respectively.  Now, working with cobordism classes in $J^1S^1$,  we can compute
\begin{align*}
(\Lambda_2, \mathcal{C}_2) &=   (\Lambda_P, \mathcal{C}_P) *(\Lambda_{P^{-1}}, \mathcal{C}_{P^{-1}}) * (\Lambda_2, \mathcal{C}_2) \\
 & =  (\Lambda_{P^{-1}}, \mathcal{C}_{P^{-1}}) * (\Lambda_2, \mathcal{C}_2) * (\Lambda_P, \mathcal{C}_P) \\
 & = (\Lambda_1, \mathcal{C}_1)
\end{align*}
where the first and third equalities are from cobordisms in $J^1[0,1]$ that arise from applying Theorem \ref{thm:J101} and the second equality is an isotopy in $J^1S^1$.
\end{proof}

\begin{remark}
When $\rho =1$, the statement (and proof) that cobordism classes of $1$-graded augmented Legendrians in $J^1S^1$ are determined by the dimension of their fiber cohomology, $\mathbf{n}$, and the conjugacy class of the monodromy matrix remains true {\it provided that} $\mathbf{n} \neq 0$.  As in the $J^1[0,1]$ case, cf. Theorem \ref{thm:J101}, any $1$-graded $(\Lambda, \mathcal{C}) \subset J^1S^1$ with  $\mathbf{n}=0$, is cobordant to $(U, \mathcal{C}_b)$ for some $b \in \mathbb{F}$, but it is not known that the $(U,\mathcal{C}_b)$ are non-cobordant.
\end{remark}

We close this section with a comparison between augmented Legendrian cobordism classes in $J^1S^1$ and ordinary Legendrian cobordism classes.  The following is an easy consequence of the work of Arnold \cite{Arnold1, Arnold2, ArnoldWave} and is likely well known, though we have not been able to locate an explicit discussion of the $J^1S^1$ case.\footnote{We remark that \cite{ArnoldWave} does state the cobordism classification for Legendrians in $ST^*\R^2$, but despite the contactomorphism $ST^*\R^2 \cong J^1S^1$ the cobordism classification is {\it not} the same.  The reason is that when working with Legendrians in $ST^*\R^2$ Arnold considers cobordisms that live in $ST^*( \R^2 \times[0,1])$, and this is not contactomorphic (or even homeomorphic) to $J^1(S^1 \times [0,1])$.}

\begin{theorem}
Two oriented Legendrian links in $J^1S^1$ are oriented cobordant if and only if 
\begin{itemize}
\item[(i)] they have the same rotation number (taken to be the sum of the rotation numbers of the components), and 
\item[(ii)] they represent the same homology class in $H_1(J^1S^1)$.
\end{itemize}
\end{theorem}

\begin{proof}
It is clear from various perspectives (eg., moves on slices, or an argument with the Maslov class) that both the rotation number and the homology class are cobordism invariants. To prove the converse statement, we make use of Arnold's result that two Legendrians in $J^1\R$ are oriented cobordant if and only if they have the same rotation number.  For Legendrians, $A, B \subset J^1S^1$ we write $A \sqcup B$ for $A$ stacked above $B$ in $J^1S^1$.   Note that any $\Lambda \subset J^1S^1$ with $[\Lambda] = [\widehat{1_m}] \in H_1(J^1S^1)$ for $m \in \Z$ (positive $m$ means oriented right, negative $m$ means oriented left)  is cobordant to $1_m \sqcup \Lambda'$ with $\Lambda' \subset J^1(0,1) \subset J^1S^1$ and $r(\Lambda) = r(\Lambda')$.  To see this, use the cobordisms 
\[
\Lambda  \sim 1_m \sqcup 1_{-m} \sqcup \Lambda \sim 1_m \sqcup \Lambda'
\]
where both cobordisms are constructed  as in Figure \ref{fig:S1toFull}.  [Without MCFs, the cobordism can be done between any two strands with opposite orientation.  Moreover, since $1_{-m} \sqcup \Lambda$ is null homologous the number of strands oriented left is equal to the number oriented right at any generic $x \in S_1$.  Thus, we can remove all of the strands at some $x$ resulting in $\Lambda'$ as required.]
Now, given any $\Lambda_0, \Lambda_1$ with the same homology class and rotation number we have
\[
\Lambda_0 \sim 1_m \sqcup \Lambda_0' \sim 1_m \sqcup \Lambda_1' \sim \Lambda_1
\]
where the second cobordism exists by the classification in $J^1\R$ since $\Lambda_0'$ and $\Lambda_1'$ have the same rotation number.
\end{proof}

As a consequence, there are many examples of augmented Legendrians $(\Lambda_i, \mathcal{C}_i)$, $i=0,1$ that are distinct in $\mathit{Cob}^\rho(J^1S^1;\mathbb{F})$ with $\rho$ even, but such that the underlying Legendrians $\Lambda_0$ and $\Lambda_1$ are oriented cobordant. In fact, in this case $\Lambda_0$ and $\Lambda_1$ are oriented cobordant if and only if their fiber cohomologies have the same Euler characteristic.  [The Euler characteristic determines the homology class of $\Lambda$ in $H_1(J^1S^1) \cong \Z$; the existence of an even graded MCF implies the existence of an even graded generalized normal ruling 
and as in \cite[Section 4]{Sabloff} it follows that the rotation number must be $0$.]  For example, the closure of the full $2$-tangle, $(\mathbf{1}_2, \mathcal{C})$, with boundary Maslov potential $\mu = (0,1)$ is nontrivial in $\mathit{Cob}^\rho(J^1S^1;\mathbb{F})$, while  $\widehat{\mathbf{1}_2}$ is oriented cobordant to $\emptyset$.

\section{Augmented Legendrian surfaces with prescribed monodromy representation}  \label{sec:representation}

As an application of the cobordism classification, we construct augmented Legendrian surfaces having arbitrary monodromy representations.  
Let $M$ be a connected, compact surface, and fix a base point $x_0 \in \mathit{Int}(M)$.  We allow for cases where $\partial M \neq 0$ and/or $M$ is non-orientable.  Recall the definition of the monodromy representation from Section \ref{sec:2-4}.

\begin{theorem}  Assume that $\rho$ is even or $\Char \mathbb{F} =2$.  
Let $\graded{n}: \mathbb{Z}/\rho \rightarrow  \Z_{\geq 0}$ with $n =\sum_l \graded{n}(l)$ finite, and let $\Phi: \pi_1(M, x_0) \rightarrow GL(\graded{n}, \mathbb{F})^\mathit{op}$ be a group homomorphism.  Then, there exists a $\rho$-graded augmented Legendrian surface $(\Sigma, \mathcal{C}) \subset J^1M$ whose monodromy representation satisfies $\Phi_{\Sigma, \mathcal{C}} = \Phi$.
\end{theorem}

\begin{proof}
Fix a handle decomposition of $M$ with a single $0$-handle having $x_0$ at its center.  
Above the $0$-handle, $h_0$, of $M$ define $\Sigma$ so that the front projection consists of $n$ non-singular, non-intersecting sheets, $S_1, \ldots, S_n$, with Maslov potential $\vec{\mu} = (\mu(S_1), \ldots, \mu(S_n))$ as specified by $\graded{n}$, eg. there are $n_l$ sheets with $\mu(S_i) = l$ and the Maslov potential is non-decreasing as $i$ increases.  Define $\mathcal{C}$ above $h_0$ to have $d_x \equiv 0$.  

For each $1$-handle $h^i_1$ with core curve $\sigma_i$, we apply Theorem \ref{thm:J101} to produce a full augmented, $n$-tangle $(\Lambda_i, \mathcal{C}_i) \subset J^1\sigma_i$, with boundary Maslov potential $\vec{\mu}$ and vanishing spin invariant, having monodromy matrix $\graded{M}_{\Lambda_i, \mathcal{C}_i} = \Phi([\sigma_i])$.  Define $(\Sigma, \mathcal{C})$ above $h^i_1$ to be the product cobordism under the identification $J^1h^i_1 \cong J^1(\sigma_i\times[0,1])$.  

For each $2$-handle, $h_2 \cong D^2$, we already have $(\Sigma, \mathcal{C})$ defined above $\partial D^2$.  Moreover, the monodromy matrix of $(\Lambda_{\partial}, \mathcal{C}_{\partial}) := (\Sigma, \mathcal{C})|_{\partial h_2}$ is the identity.  [The attaching map $\gamma:S^1 \rightarrow M\setminus \mathit{Int}(h_2)$ is homotopic in $M \setminus \mathit{Int}(h_2)$ to a loop $\widetilde{\gamma}$ that is product of the $\sigma_i$ and their inverses.  Since  the continuation isomorphisms $\phi_{\sigma_i}$ along these loops have matrices $\Phi([\sigma_i])$ we get from Proposition \ref{prop:continuation} that the matrix of $\phi_{\gamma}$ is $\Phi([\widetilde{\gamma}])$.  The latter is the identity matrix since $\gamma$ is null homotopic in $M$.]  
In addition, the spin invariant of $(\Lambda_{\partial}, \mathcal{C}_{\partial})$ viewed as a full $n$-tangle is $0$.  
Therefore, another application of Theorem \ref{thm:J101} produces a cobordism $(\Sigma_A, \mathcal{C}_A)$ from $(\Lambda_{\partial}, \mathcal{C}_{\partial})$ to the identity augmented $n$-tangle.  We then use $(\Sigma_A, \mathcal{C}_A)$ to fill $(\Sigma, \mathcal{C})$ into an annular neighborhood $A \subset h_2$ of $\partial h_2$, and complete the construction of $(\Sigma, \mathcal{C})$  by placing $n$ non-intersecting sheets with vanishing $d_x$ in the remaining disk $h_2 \setminus A$.  By construction, on the generating set $[\sigma_i] \in \pi_{1}(M,x_0)$ we have $\Phi_{\Sigma, \mathcal{C}}([\sigma_i]) = \graded{M}_{\Lambda_i, \mathcal{C}_i} = \Phi([\sigma_i])$.
\end{proof}

\begin{remark}  In the case that $\mbox{Char}\,\mathbb{F}=2$, the augmented Legendrian $(\Sigma, \mathcal{C})$ can be constructed so that $\Sigma$ is a Legendrian {\it $n$-weave} in the terminology of Casals and Zaslow \cite{CZ}.  To see this, note that above the $1$-skeleton of $M$ the $(\Lambda_i, \mathcal{C}_i)$ can be chosen with $\Lambda_i$ a Legendrian braid (in fact, by Propositions \ref{prop:cusp} and \ref{prop:standard} a standard form, permutation braid).  Moreover, since we start with an augmented Legendrian braid above $\partial h_2$, the cobordism used to fill in $(\Sigma, \mathcal{C})$ over the $2$-handle can be constructed using only Proposition \ref{prop:standard}.  The underlying Legendrian cobordism involves only clasp moves and $D_4^-$ singularities at Step 1 and Step 2, and since there are no spin base points the Legendrian remains unchanged during Step 3.  
\end{remark}

For example, Figure \ref{fig:TorusEx} presents a pair of three sheeted augmented Legendrian weaves over $\mathbb{F}_2$ in $J^1T^2$  with monodromy representations, $\Phi_1$ and $\Phi_2$, satisfying
\[
\Phi_1(\mathbf{l}) = \left[ \begin{array}{ccc} 1 & 0 & 0 \\ 1 & 1 & 1 \\ 0 & 0 & 1 \end{array} \right], \quad \Phi_1(\mathbf{m}) = \left[ \begin{array}{ccc} 0 & 0 & 1 \\ 0 & 1 & 0 \\ 1 & 0 & 0 \end{array} \right], \quad 
\Phi_2(\mathbf{l}) = \left[ \begin{array}{ccc} 0 & 1 & 1 \\ 0 & 1 & 0 \\ 1 & 1 & 0 \end{array} \right], \quad \Phi_2(\mathbf{m}) = \left[ \begin{array}{ccc} 0 &  1 & 1 \\ 1 & 0 & 1 \\ 0 & 0 & 1 \end{array} \right] 
\] 
where $\mathbf{l}$ and $\mathbf{m}$ are the generators of $\pi_1(T^2, (0,0))$ corresponding to the horizontal and vertical edges of the square respectively.

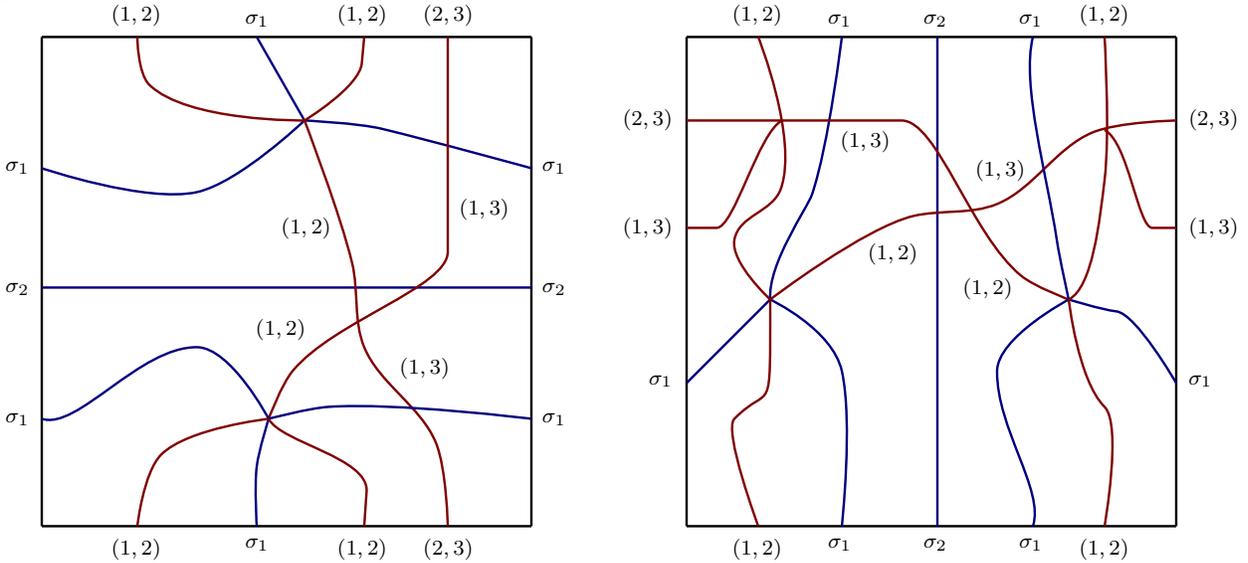
\begin{figure}[!ht]

\quad

\labellist
\tiny
\pinlabel $\sigma_1$ [r] at -2 68
\pinlabel $\sigma_2$ [r] at -2 150
\pinlabel $\sigma_1$ [r] at -2 226

\pinlabel $(1,2)$ [t] at 60 -2
\pinlabel $\sigma_1$ [t] at 136 -2
\pinlabel $(1,2)$ [t] at 202 -2
\pinlabel $(2,3)$ [t] at 256 -2

\pinlabel $\sigma_1$ [l] at 310 68
\pinlabel $\sigma_2$ [l] at 310 150
\pinlabel $\sigma_1$ [l] at 310 226

\pinlabel $(1,2)$ [b] at 60 310
\pinlabel $\sigma_1$ [b] at 136 310
\pinlabel $(1,2)$ [b] at 202 310
\pinlabel $(2,3)$ [b] at 256 310

\pinlabel $(1,2)$ [br] at 170 114
\pinlabel $(1,3)$ [bl] at 222 90
\pinlabel $(1,2)$ [tr] at 186 198
\pinlabel $(1,3)$ [l] at 258 200

\pinlabel $(1,2)$ [t] at 450 -2
\pinlabel $\sigma_1$ [t] at 502 -2
\pinlabel $\sigma_2$ [t] at 562 -2
\pinlabel $\sigma_1$ [t] at 622 -2
\pinlabel $(1,2)$ [t] at 668 -2

\pinlabel $(2,3)$ [r] at 402 256
\pinlabel $(1,3)$ [r] at 402 188
\pinlabel $\sigma_1$ [r] at 402 92

\pinlabel $(1,2)$ [b] at 450 310
\pinlabel $\sigma_1$ [b] at 502 310
\pinlabel $\sigma_2$ [b] at 562 310
\pinlabel $\sigma_1$ [b] at 622 310
\pinlabel $(1,2)$ [b] at 668 310

\pinlabel $(2,3)$ [l] at 716 256
\pinlabel $(1,3)$ [l] at 716 188
\pinlabel $\sigma_1$ [l] at 716 92

\pinlabel $(1,2)$ [tl] at 516 182
\pinlabel $(1,3)$ [t] at 518 254
\pinlabel $(1,2)$ [tr] at 614 160
\pinlabel $(1,3)$ [br] at 622 214

\endlabellist

\centerline{\includegraphics[scale=.6]{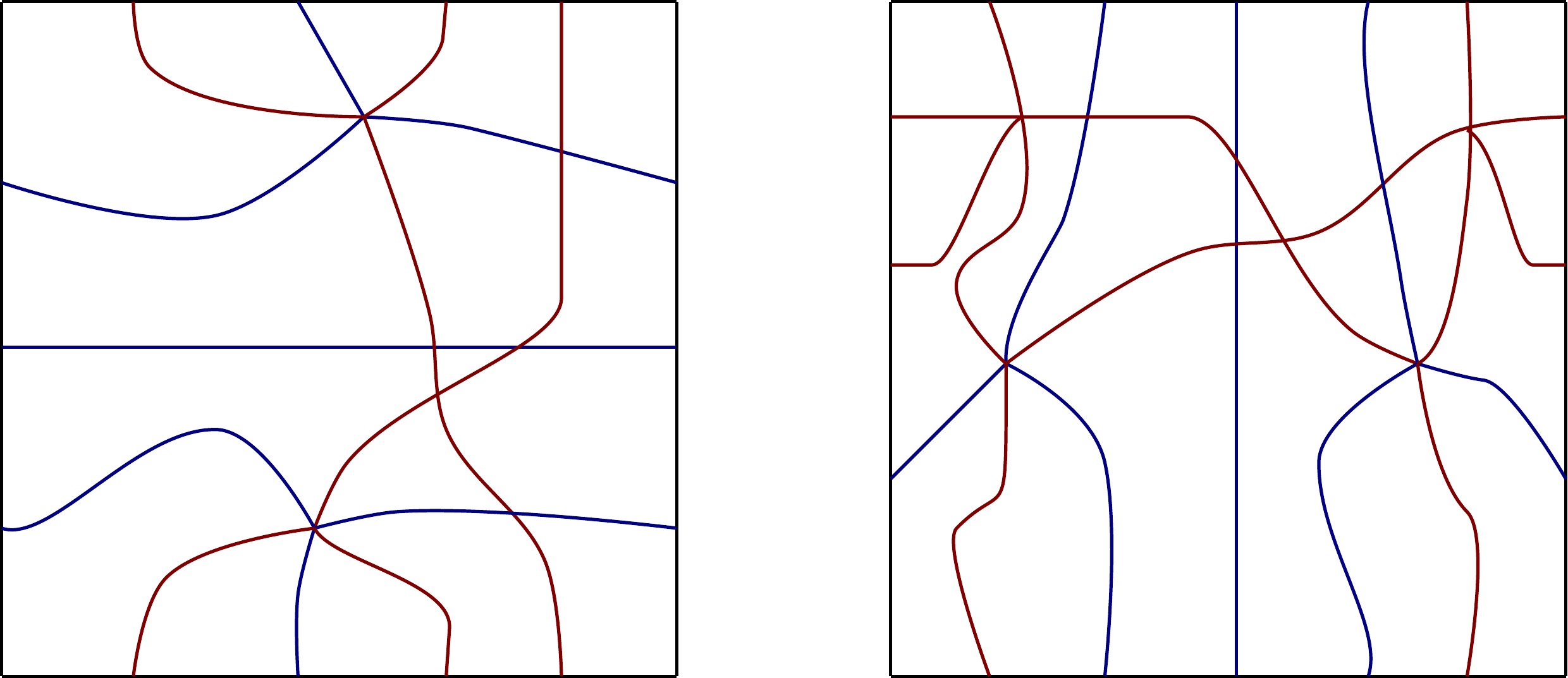}}

\quad 

\caption{A blue curves labeled with $\sigma_i$ denotes a crossing arc between sheets $i$ and $i+1$.  A red curve labeled with $(i,j)$ denotes an $(i,j)$-handleslide with coefficient $1\in \mathbb{F}_2$.}
\label{fig:TorusEx}
\end{figure}

\bibliographystyle{abbrv}
\bibliography{PR}

\end{document}